\documentclass[10pt]{article}
\usepackage{amsmath, amssymb, mathrsfs}
\usepackage{enumerate}

\usepackage{amsthm}
\theoremstyle{plain}
\newtheorem{definition}{Definition}[section]
\theoremstyle{plain}
\newtheorem{theorem}[definition]{Theorem}
\theoremstyle{plain}
\newtheorem{lemma}[definition]{Lemma}
\theoremstyle{plain}
\newtheorem{proposition}[definition]{Proposition}
\theoremstyle{plain}
\newtheorem{corollary}[definition]{Corollary}
\theoremstyle{definition}
\newtheorem{example}[definition]{Example}
\theoremstyle{definition}
\newtheorem{remark}[definition]{Remark}
\theoremstyle{definition}
\newtheorem{notation}[definition]{Notation}
\renewcommand{\qedsymbol}{$\blacksquare$}

\numberwithin{equation}{section}

\makeatletter
\def\institute#1{\def\@institute{#1}}
\def\email#1{\def\@email{#1}}
\def\@maketitle{
\begin{center}
{\Large\bf \@title \par}
\vspace{15pt}
{\large\bf \@author \par}
\vspace{10pt}
{ \@institute \par} 
\vspace{5pt}
{E-mail: \@email \par}
\end{center}
\par\vskip 10pt
}
\makeatother

\title{Vertex Algebras and the Equivariant Lie Algebroid Cohomology}
\author{Masanari Okumura\thanks{Research Fellow of Japan Society for the Promotion of Science}}
\date{}
\institute{Graduate School of Mathematical Sciences, The University of Tokyo, 3-8-1 Komaba, Meguro-ku, Tokyo 153-8914, Japan}  \email{okumura@ms.u-tokyo.ac.jp}
\begin{document}

\maketitle

\begin{abstract}
A vertex-algebraic analogue of the Lie algebroid complex is constructed, which generalizes the ``small" chiral de Rham complex  on smooth manifolds. 
The notion of VSA-inductive sheaves is also introduced. This notion  generalizes that of sheaves of vertex superalgebras. The complex mentioned above is constructed as a VSA-inductive sheaf.
With this complex, the equivariant Lie algebroid cohomology is generalized to a vertex-algebraic analogue, which we call the chiral equivariant Lie algebroid cohomology. 
In fact, the notion of the equivariant Lie algebroid cohomology contains that of the equivariant Poisson cohomology. Thus the chiral equivariant Lie algebroid cohomology is also a vertex-algebraic generalization of the equivariant Poisson cohomology.   
A special kind of complex is introduced and its properties are studied in detail. With these properties, 
some isomorphisms of cohomologies are developed, which enables us to compute the chiral equivariant Lie algebroid cohomology in some cases. Poisson-Lie groups are considered as such a special case. 
\end{abstract}

\tableofcontents

\section{Introduction}\label{section: Introduction}
The \textit{chiral de Rham complex} (\textit{CDR}) was introduced by Malikov-Schechtman-Vaintrob in \cite{MSV99}.  It is a sheaf of vertex superalgebras containing the usual de Rham complex. 
Gorbounov-Malikov-Schechtman generalized this notion by introducing the sheaf of \textit{chiral differential operators}, and studied the sheaf in a series of  papers \cite{GMS00,GMS03,GMS04}. 
Another construction of the CDR was done by means of formal loop spaces in \cite{KV04}  by Kapranov-Vasserot.  
Moreover the CDR was studied in relation to elliptic genera and mirror symmetry in \cite{BL02,Bo01,BL00}. 
Recently Lian-Linshaw introduced a new equivariant cohomology theory in \cite{LL}, and studied in detail the CDR in the $C^\infty$-setting. 
The CDR was also investigated in terms of \textit{SUSY vertex algebras} in \cite{BZHS08,EHKZ13,Hel09,HK07,HZ10,HZ11}. 

Let $G$ be a compact connected Lie group with complexified Lie algebra $\mathfrak{g}$ and let $M$ be a $G$-manifold. Then the  algebra $\Omega(M)$ of differential forms on $M$ has a canonical $G$-action. Together with the Lie derivatives and the interior products, this action of $G$ makes $\Omega(M)$ a $G^*$-\textit{algebra}. It is well- known that the equivariant cohomology of $M$ is computed as the equivariant cohomology of the $G^*$-algebra $\Omega(M)$ (see \cite{GS99}). As a vertex-algebraic analogue of this equivariant cohomology theory,  
the \textit{chiral equivariant cohomology} was introduced by Lian-Linshaw in \cite{LL}. This cohomology was defined for $O(\mathfrak{sg})$-\textit{algebras}, a vertex-algebraic analogue of $G^*$-algebras. 
The key to the construction is the fact that the \textit{semi-infinite Weil complex} $\mathcal{W}(\mathfrak{g})$ introduced by Feigin-Frenkel in \cite{FF91} has an $O(\mathfrak{sg})$-algebra structure. This fact was proved by Lian-Linshaw together with the fact that the space $\mathcal{Q}(M)$ of global sections of the CDR of  $M$ has an $O(\mathfrak{sg})$-algebra structure. 
Later  in \cite{LLS1}, Lian-Linshaw-Song introduced the notion of $\mathfrak{sg}[t]$-\textit{modules} as an analogue of $\mathfrak{g}$-differential complexes, a complex with a compatible action of the Lie superalgebra $\mathfrak{sg}=\mathfrak{g}\ltimes_{\mathrm{ad}} \mathfrak{g}$.
As a classical equivariant cohomology theory, the construction of the chiral equivariant cohomology was generalized for the case of  $\mathfrak{sg}[t]$-\textit{modules}. 
Moreover in \cite{LLS1}, Lian-Linshaw-Song also 
introduced a ``small" CDR as a subcomplex of the CDR, and pointed out that the global section of the subcomplex is an $\mathfrak{sg}[t]$-module. 
The small CDR itself is trivial since its vertex superalgebra structure is commutative. However when one consider the corresponding chiral equivariant cohomology, the vertex superalgebra structure is very complicated in general. Such a vertex superalgebra were studied in \cite{LLS1,LLS2}. 

An interesting example of $\mathfrak{g}$-differential complexes was considered in \cite{Gin99} by Ginzburg. This example comes from the space of multi-vector fields of a Poisson manifold with an action of $\mathfrak{g}$ and the corresponding equivariant cohomology is called the \textit{equivariant Poisson cohomology}. He also pointed out that one can define the same kind of equivariant cohomology  for Lie algebroids.  

In this paper, 
we construct $\mathfrak{sg}[t]$-modules from Lie algebroids with an action of $\mathfrak{g}$, generalizing the small CDR of Lian-Linshaw-Song. We then define the chiral equivariant Lie algebroid cohomology.  

The CDR as well as the small version constructed by Lian-Linshaw on smooth manifolds is actually not a sheaf but a presheaf with a property, called a \textit{weak sheaf} by Lian-Linshaw-Song. For this reason, one has a little ambiguity in the choice of morphisms or the gluing properties. Therefore we introduce the notion of VSA-inductive sheaves, which generalize that of sheaves of vertex superalgebras, and formulate the morphisms and the gluing properties. 
We construct a VSA-inductive sheaf associated with a Lie algebroid  and we obtain vertex-algebraic analogue of the Lie algebroid complex. 
Moreover we prove that the complex above has an $\mathfrak{sg}[t]$-module structure when the Lie algebroid has an action of $\mathfrak{g}$. This leads us to the definition of the chiral equivariant Lie algebroid cohomology. 
When the Lie algebroid is a tangent bundle, we recover the CDR of Lian-Linshaw-Song. 

In the classical equivariant cohomology theory, 
an important role is played by special complexes called  $W^*$-\textit{modules}. 
They have remarkable properties, which make it easy to compute their equivariant cohomologies (see \cite{GS99}). Motivated by this fact, we introduce the notion of chiral $W^*$-modules and prove that they have some properties analogous to those of $W^*$-modules. Moreover we prove that the complexes obtained from the VSA-inductive sheaves associated with a type of  Lie algebroid containing  the \textit{cotangent Lie algebroids} of Poisson-Lie groups have chiral $W^*$-module structures. We then compute their chiral equivariant Lie algebroid cohomologies by using the properties of chiral $W^*$-modules mentioned above. 

The article is organized as follows: 
in Section 2 we recall some basics of vertex superalgebras and the chiral equivariant cohomology. 
In Section 3, we introduce the chiral $W^*$-modules and 
a chiral Cartan model for $\mathfrak{sg[t]}$-modules. We mainly consider this chiral Cartan model when $\mathfrak{g}$ is commutative. 
The result in this section will be used in the last part of Section 6.    
In Section 4, we introduce the notion of VSA-inductive sheaves. Then we establish some gluing properties. Moreover we construct VSA-inductive sheaves from presheaves of degree-weight-graded vertex superalgebras with some properties. 
In Section 5, after recalling the notion of Lie algebroids and the Lie algebroid cohomology, we first construct an important VSA-inductive sheaf on $\mathbb{R}^m$  and its small version, which we denote  respectively by $\Omega_\mathrm{ch}(\mathbb{R}^{m|r})$ and $\Omega^{\gamma c}_\mathrm{ch}(\mathbb{R}^{m|r})$. Using  the gluing property proved in Section 4, we glue the small ones $\Omega^{\gamma c}_\mathrm{ch}(\mathbb{R}^{m|r})$   into the global VSA-inductive sheaf associated with an arbitrary vector bundle.  Next for a Lie algebroid, we construct a differential on the VSA-inductive sheaf associated with the Lie algebroid, using the vertex operators of the bigger one $\Omega_\mathrm{ch}(\mathbb{R}^{m|r})$. Thus we obtain a vertex-algebraic analogue of the Lie algebroid complex. Moreover in Section 6, we equip this complex with an $\mathfrak{sg}[t]$-module structure, when the Lie algebroid has an action of a Lie algebra $\mathfrak{g}$. 
For this construction, we also need the vertex operators of the bigger VSA-inductive sheaf.   
Then we introduce the chiral equivariant Lie algebroid cohomology. 
In the last part of Section 6, we compute this cohomology for an important Lie algebroid called a \textit{transformation Lie algebroid}. In particular, we compute that cohomology for the cotangent Lie algebroids associated with \textit{Poisson-Lie groups}.

\vspace{10pt}
Throughout this paper, $\mathbb{K}$ is the field of real numbers $\mathbb{R}$ or that of complex numbers $\mathbb{C}$, and we will work over $\mathbb{K}$. We assume that a grading on a super vector space is compatible with the super vector space structure.

\section{Preliminaries}\label{section: Preliminaries} 
\subsection{Vertex Superalgebras}

We first recall basic definitions and facts concerning vertex superalgebras, which was introduced in \cite{Bor86}. We will follow the formalism and results in \cite{FBZ04,Kac01,LL04}.

A \textit{vertex superalgebra} is a quadruple $(V, \mathbf{1},T,Y)$ consisting of a super vector space $V,$ an even vector $\mathbf{1} \in V$, called the vacuum vector, 
an even linear operator $T : V \to V$, called the translation operator, and
an even linear operation 
$
Y=Y(\, \cdot \, , z) : V \to (\mathrm{End} V)[[z^{\pm1}]],
$
taking each $A \in V$ to a field on $V$, called the vertex operator,
$$
Y(A,z)= \sum_{n \in \mathbb{Z}} A_{(n)} z^{-n-1},
$$
such that 
\begin{enumerate}[$\bullet$]
     \setlength{\topsep}{1pt}
     \setlength{\partopsep}{0pt}
     \setlength{\itemsep}{1pt}
     \setlength{\parsep}{0pt}
     \setlength{\leftmargin}{20pt}
     \setlength{\rightmargin}{0pt}
     \setlength{\listparindent}{0pt}
     \setlength{\labelsep}{3pt}
     \setlength{\labelwidth}{30pt}
     \setlength{\itemindent}{0pt}
\item
(vacuum axiom)\\
$Y(\mathbf{1} ,z) = \mathrm{id}_V$;\\
$
Y(A,z)\mathbf{1} \in V[[z]],
$
and 
$
Y(A,z)\mathbf{1} | _{z=0} = A
$
\quad for any $A \in V$;
\item
(translation axiom)\\
$T \mathbf{1} = 0$;\\
$
[T,Y(A,z)] = \frac{\mathit{d}}{\mathit{d}z} Y(A,z)
$
\quad for any $A \in V$; 
\item
(locality axiom)\\
For any $A, B \in V,$ $Y(A,z)$ and $Y(B,z)$ are mutually local.
\end{enumerate}
A vertex superalgebra  $(V, \mathbf{1}, T, Y)$ is said to be  
\textit{$\mathbb{Z}$-graded} 
when the super vector space 
$V$ is given a $\mathbb{Z}$-grading $V = \oplus_{n \in \mathbb{Z}} V[n]$ such that 
$\mathbf{1}$ is a vector of weight $0$, 
$T$ is a homogeneous linear operator of weight $1$, and
if $A \in V[n]$, then the field $Y(A,z)$ is homogeneous of conformal dimension $n$.  
We refer to such a grading as a \textit{weight-grading} on the vertex superalgebra.
Note that $\mathrm{Im}\,Y\subset(\mathrm{End}\,V)[[z^{\pm1}]]$ has a canonical structure of vertex superalgebra and that $Y: V\to \mathrm{Im}\,Y$ is an isomorphism of vertex superalgebras. We will often identify these two vertex superalgebras. 

For a vertex superalgebra $(V, \mathbf{1}, T, Y)$, the data of $Y$ is equivalent to that of bilinear maps 
$$
(n): V\times V\to V, \quad (A, B)\mapsto A_{(n)}B.
$$
Therefore vertex algebras are also written as $(V, \mathbf{1}, T, (n); n\in \mathbb{Z})$. 
Moreover the translation operator $T$ and the vertex operator $Y(A, z)$ are often denoted by $\partial$ and by $A(z)$, respectively.
A purely even vertex superalgebra is called simply a \textit{vertex algebra}.  

We use the following notation for the \textit{operator product expansion (OPE)} of the mutually local fields $A(z)$ and $B(z)$: 
$$
A(z)B(w)\sim \sum_{n\ge 0}C_n(w)(z-w)^{-n-1},
$$
where $C_n(w)$ are some fields. Note that the OPE formula gives the commutation relations among the coefficients of $A(z)$ and $B(w)$. (See  \cite{FBZ04} and \cite{Kac01} for details.) 

The following examples of vertex superalgebras will be used for the  construction of some kinds of cohomology later.

\begin{example}[affine vertex superalgebras]\label{ex:affine va}
Let $\mathfrak{g}$ be a Lie superalgebra with a supersymmetric invariant bilinear form $B$. 
Let $\Hat{\mathfrak{g}}=\mathfrak{g} [t^{\pm1}]\oplus \mathbb{K}K$ be 
the affine Lie algebra associated with $(\mathfrak{g}, B).$
Set 
$$
N(\mathfrak{g}, B):=U(\hat{\mathfrak{g}})\otimes_{U(\mathfrak{g}[t]\oplus \mathbb{K}K)}\mathbb{K}_1,
$$
where $\mathbb{K}_1$ is the one-dimensional $\mathfrak{g}[t]\oplus \mathbb{K}K$-module on which $\mathfrak{g}[t]$ acts by zero and $K$ by $1.$
This $\Hat{\mathfrak{g}}$-module $N(\mathfrak{g}, B)$ has a 
$\mathbb{Z}_{\ge 0}$-graded 
vertex superalgebra structure 
called the \textit{affine vertex superalgebra} associated with $\mathfrak{g}$ and $B$. 
Note that the operator $a_{(n)}$ has weight $-n$, where $a_{(n)}$  stands for the operator on $N(\mathfrak{g}, B)$ corresponding to $at^n\in \hat{\mathfrak{g}}$. 
The Lie superalgebra $\mathfrak{g}$ can be seen as a subspace of $N(\mathfrak{g}, B)$ by the injection $\mathfrak{g}\to N(\mathfrak{g}, B),\ a\mapsto a_{(-1)}\mathbf{1}$. We denote by $O(\mathfrak{g}, B)$ the corresponding vertex superalgebra $\mathrm{Im}(Y)=Y(N(\mathfrak{g}, B))\subset (\mathrm{End}\,N(\mathfrak{g}, B))[[z^{\pm1}]]$.
\end{example}

\begin{example}[$\beta \gamma$-systems]\label{ex:beta_gamma-systems}
Let $V$ be a finite-dimensional  vector space.
Let $\mathfrak{h}(V)$ $=(V[t^{\pm1}]\oplus V^*[t^{\pm1}]dt)\oplus \mathbb{K}\mathbf{\tau}$ be 
the \textit{Heisenberg Lie algebra} associated with $V$.
Set 
$$
\mathcal{S}(V):=U(\mathfrak{h}(V))\otimes_{U(V[t]\oplus V^*[t]dt\oplus \mathbb{K}\mathbf{\tau})}\mathbb{K}_1,
$$
where $\mathbb{K}_1$ is the one-dimensional $(V[t]\oplus V^*[t]dt\oplus \mathbb{K}\mathbf{\mathbf{\tau}})$-module in which $V[t]\oplus V^*[t]dt$ acts by zero and $\mathbf{\tau}$ by $1.$
We denote by $\beta^{v}_n,\gamma^{\phi}_n$ the elements $v\otimes t^n, \phi \otimes t^{n-1}dt \in \mathfrak{h}(V)$, respectively.
The $\mathfrak{h}(V)$-module $\mathcal{S}(V)$ has a $\mathbb{Z}_{\ge 0}$-graded vertex algebra structure 
called the $\beta \gamma$-\textit{system} associated with $V$. 
We sometimes denote $\beta^v_n$ and $\gamma^\phi_n$ by $\beta^v_{(n)}$ and $\gamma^\phi_{(n-1)}$, respectively.
\end{example}

\begin{example}[$bc$-systems]\label{ex:bc-systems}
Let $V$ be a finite-dimensional vector space.
We regard $V[t^{\pm1}]\oplus V^*[t^{\pm1}]dt$ as an odd abelian Lie algebra.
Consider the one-dimensional central extension
$
\mathfrak{j}(V)=(V[t^{\pm1}]\oplus V^*[t^{\pm1}]dt)\oplus \mathbb{K}\mathbf{\tau}
$
of that odd abelian Lie algebra 
with bracket
\begin{multline*}
[v_1\otimes f_1+\phi_1 \otimes g_1dt,v_2\otimes f_2+\phi_2 \otimes g_2dt] \\
=(\langle v_1,\phi_2\rangle\mathrm{Res}_{t=0}f_1g_2dt +\langle v_2,\phi_1\rangle\mathrm{Res}_{t=0}f_2g_1dt)\mathbf{\tau}.
\end{multline*}
Set 
$$
\mathcal{E}(V):=U(\mathfrak{j}(V))\otimes_{U(V[t]\oplus V^*[t]dt\oplus \mathbb{K}\mathbf{\tau})}\mathbb{K}_1,
$$
where $\mathbb{K}_1$ is the one-dimensional $(V[t]\oplus V^*[t]dt\oplus \mathbb{K}\mathbf{\tau})$-module in which $V[t]\oplus V^*[t]dt$ acts by zero and $\mathbf{\tau}$ by $1.$
We denote by $b^{v}_n, c^{\phi}_n$ the elements $v\otimes t^n, \phi \otimes t^{n-1}dt,$ respectively.
The $\mathfrak{j}(V)$-module $\mathcal{E}(V)$ has a $\mathbb{Z}_{\ge 0}$-graded vertex superalgebra structure called 
the $bc$-\textit{system} associated with $V$. 
We sometimes denote $b^v_n$ and $c^\phi_n$ by $b^v_{(n)}$ and $c^\phi_{(n-1)}$, respectively.
\end{example}

\begin{example}[semi-infinite Weil algebras]\label{ex:semi-infinite_Weil_algebras}
For a vector space $V$, the tensor product vertex superalgebra 
$$
\mathcal{W}(V):=\mathcal{E}(V)\otimes \mathcal{S}(V),
$$
is called the \textit{semi-infinite Weil algebra} associated with $V$ (\cite{FF91}). 
\end{example}

We recall some graded structures on vertex superalgebras.
A vertex superalgebra $V$ is \textit{degree-graded} if it is given a $\mathbb{Z}$-grading $V=\bigoplus_{p\in\mathbb{Z}}V^p$ such that 
$
A_{(n)}B\in V^{p+q}
$
for all $A\in V^p, B\in V^q$, $n\in\mathbb{Z}$ and $\mathbf{1}\in V^0$.
Recall that a $\mathbb{Z}$-grading $V=\bigoplus_{n\in \mathbb{Z}}V[n]$ on a vertex superalgebra $V$ is called a weight-grading if 
$
A_{(k)}B\in V[n+m-k-1]
$
for all $A\in V[n]$, $B\in V[m]$ and $k\in \mathbb{Z}$, and $\mathbf{1}\in V[0]$.
A vertex superalgebra $V$ is \textit{degree-weight-graded} if $V$ is both degree and weight-graded and the gradings are compatible, that is, $V=\bigoplus_{p, n\in \mathbb{Z}}V^p[n]$, where $V^p[n]=V^p\cap V[n]$.

\begin{example}\label{ex: another grading on N(g, 0)}
We can define a degree-weight-grading on the affine vertex superalgeba $N(\mathfrak{g}, B)$ when $\mathfrak{g}$ has two compatible  $\mathbb{Z}$-grading and the invariant bilinear form $B$ is $0$. 
We call  the one grading on $\mathfrak{g}$ the weight-grading and the other the degree-grading. 
Then $N(\mathfrak{g}, 0)$ becomes a degree-weight-graded vertex superalgebra if we give the weight-grading by 
$
\mathrm{wt}\, a^1_{(n_1)}\dots a^r_{(n_r)}\mathbf{1} :=\sum_{i=1}^r(-n_i+\mathrm{wt}_\mathfrak{g}a^i),
$ and the degree-grading by 
$
\mathrm{deg}\, a^1_{(n_1)}\dots a^r_{(n_r)}\mathbf{1} :=\sum_{i=1}^r\mathrm{deg}_\mathfrak{g}a^i,
$ 
for degree-weight-homogeneous elements $a^1, \dots, a^r \in \mathfrak{g}$ and $n_1, \dots, n_r\in \mathbb{Z}_{<0}$. 
We call this grading the degree-weight-grading on the vertex superalgebra $N(\mathfrak{g}, 0)$ associated with the grading on $\mathfrak{g}$. 
\end{example}

In the sequel, we will always assume that an action of a degree-weight-graded vertex superalgebra on a degree-weight-graded super vector space is compatible with the gradings.

We give some lemmas used in Section \ref{section: Chiral Lie Algebroid Cohomology} and \ref{section: Chiral Equivariant Lie Algebroid Cohomology}. 
Recall the notion of vertex superalgebra derivation. 
A \textit{derivation} on a vertex superalgebra $V$ with parity $\bar{i}$ is an endomorphism $d$ on $V$ with  parity $\bar{i}$ such that 
$
[d, Y(A, z)]=Y(d\cdot A, z)
$
for any $A\in V$. 

\begin{lemma}\label{lem: odd derivation is 0 if so on generators}
Let $V$ be a vertex superalgebra generated by a subset $S\subset V$. Let $D$ be an odd derivation on the vertex superalgebra $V$ such that $D^2|_S=0$. Then $D^2=0$ hold on $V$.
\end{lemma}
\begin{proof}
Since the operator $[D,D]=2D^2$ is also a derivation, our assertion holds.
\end{proof}

\begin{lemma}\label{lem: sufficient condition for B-linearilty}
Let $f: V\to W$ be a morphism of vertex superalgebras. Let $N$ be a non-negative integer.  Suppose $V$ is generated by a subset $S\subset V$. Let $(A_{(n)})_{0\le n \le N}$ and $(B_{(n)})_{0\le n \le N}$ be linear maps on $V$ and $W$, respectively. Suppose the following hold:
\begin{gather}\label{eq: B-commutation relation}
[A_{(n)},v_{(k)}]=\sum_{i\ge0}\binom{n}{i}(A_{(i)}v)_{(n+k-i)}, \\ \label{eq: B'-commutation relation}
[B_{(n)},f(v)_{(k)}]=\sum_{i\ge0}\binom{n}{i}(B_{(i)}f(v))_{(n+k-i)}, 
\end{gather}
for all $v\in S$, $0\le n\le N$, and $k\in\mathbb{Z}$
Then if $f\circ A_{(n)}=B_{(n)}\circ f$ on $S$ for all $0\le n\le N$, then  $f\circ A_{(n)}=B_{(n)}\circ f$ holds on $V$ for any $0\le n\le N$.
\end{lemma}
\begin{proof}
It suffices to show that $(f\circ A_{(n)})(v)=(B_{(n)}\circ f)(v)$ for $v\in V$ of the form $s^1_{(n_1)}\cdots s^r_{(n_r)}\mathbf{1}$ with $s^1, \dots, s^r\in S$ and $n_1, \dots, n_r\in\mathbb{Z}$. 
The assertion is proved by induction on $r$. Note that $A_{(n)}\mathbf{1}=0$ and $B_{(n)}\mathbf{1}=0$ are proved by induction on $n\in\mathbb{N}$ with \eqref{eq: B-commutation relation} and \eqref{eq: B'-commutation relation}. 
\end{proof}

\begin{lemma}\label{lem: generalized commutant}
Let $V$ be a vertex superalgebra. Let $\mathcal{A}=(A_{(m)}^\lambda)_{m\ge0, \lambda\in\Lambda}$ be a family of $\mathbb{Z}/2\mathbb{Z}$-homogeneous linear maps on $V$ such that 
$$
[A_{(m)}^\lambda, v_{(k)}]=\sum_{i\ge0}\binom{m}{i}(A_{(i)}^\lambda v)_{(m+k-i)},
$$
for all $m\ge0$, $\lambda\in\Lambda$, $k\in\mathbb{Z}$ and $v\in V$. 
Then $V^\mathcal{A}:=\{v\in V \bigm| A_{(m)}^\lambda v=0\ \textrm{for all}\ m\ge0\ \textrm{and}\ \lambda\in\Lambda\}$ is a subalgebra of $V$. 
\end{lemma}
\begin{proof}
The equality $A_{(m)}^\lambda \mathbf{1}=0$ is proved by induction on $n$. 
The subspace $V^\mathcal{A}$ is closed under the $n$-th product due to the assumption. 
\end{proof}

\begin{lemma}\label{lem: tensor commutant}
Let $V$ and $W$ be vertex superalgebras. Let $\mathcal{A}=(A_{(m)}^\lambda)_{m\ge0, \lambda\in\Lambda}$ be a family of $\mathbb{Z}/2\mathbb{Z}$-homogeneous linear maps on $V$ and $\mathcal{B}=(B_{(m)}^\lambda)_{m\ge0, \lambda\in\Lambda}$ a family of $\mathbb{Z}/2\mathbb{Z}$-homogeneous linear maps on $W$ such that 
\begin{gather*}
[A_{(m)}^\lambda, v_{(k)}]=\sum_{i\ge0}\binom{m}{i}(A_{(i)}^\lambda v)_{(m+k-i)}, \\ 
[B_{(m)}^\lambda, w_{(k)}]=\sum_{i\ge0}\binom{m}{i}(B_{(i)}^\lambda w)_{(m+k-i)},
\end{gather*}
for all $m\ge0$, $\lambda\in\Lambda$, $k\in\mathbb{Z}$, $v\in V$ and $w\in W$. Then the family of $\mathbb{Z}/2\mathbb{Z}$-homogeneous linear maps on $V\otimes W$, $(A_{(m)}^\lambda\otimes \mathrm{id}+\mathrm{id}\otimes B_{(m)}^\lambda)_{m\ge0, \lambda\in\Lambda}$, satisfies the relation 
$$
[A_{(m)}^\lambda\otimes \mathrm{id}+\mathrm{id}\otimes B_{(m)}^\lambda, x_{(k)}]=\sum_{i\ge0}\binom{m}{i}\bigl((A_{(m)}^\lambda\otimes \mathrm{id}+\mathrm{id}\otimes B_{(m)}^\lambda) x\bigr)_{(m+k-i)},
$$
for any $m\ge0$, $\lambda\in\Lambda$, $k\in\mathbb{Z}$, $x\in V\otimes W$.
\end{lemma}
\begin{proof}
The assertion is proved by  
 direct computations. 
\end{proof}

\subsection{Chiral Equivariant Cohomology}\label{subsection: Chiral Equivariant Cohomology}

We next recall the definition of the chiral equivariant cohomology. We  refer the reader to \cite{LL,LLS1} and partly to \cite{FF91}, 
 for more details.

Let $\mathfrak{g}$ be a Lie algebra. The Lie superalgebra $\mathfrak{sg}$ is defined by 
$$
\mathfrak{sg}:=\mathfrak{g}\ltimes\mathfrak{g}_{-1},
$$
where $\mathfrak{g}_{-1}$ is the adjoint representation of $\mathfrak{g}$.

Let $O(\mathfrak{sg}, 0)$ be the affine vertex superalgebra associated with the Lie superalgebra $\mathfrak{sg}$ with an invariant bilinear form $0$. The Lie superalgebra derivation
$$
\mathfrak{sg} \to \mathfrak{sg}, \quad (\xi, \eta)\mapsto (\eta, 0),
$$
induces a vertex superalgebra derivation
$$
\mathbf{d}: O(\mathfrak{sg}, 0)\to O(\mathfrak{sg}, 0),\quad  (\xi, \eta)(z)\mapsto (\eta, 0)(z).
$$
This makes $O(\mathfrak{sg}):=(O(\mathfrak{sg},0), \mathbf{d})$ a differential degree-weight-graded vertex algebra, that is, a degree-weight-graded vertex superalgebra with a square-zero odd vertex superalgera derivation of degree $1$, where the gradings are given by $\mathrm{deg} ((\xi, 0)(z))=0, \mathrm{deg} ((0, \eta)(z))=-1$ and $\mathrm{wt} ((\xi, \eta)(z))=1$. 

Recall the notion of $O(\mathfrak{sg})$-algebras from \cite{LL}. 
An $O(\mathfrak{sg})$-\textit{algebra} is a differential degree-weight-graded vertex superalgebra $(\mathcal{A}, d)$ equipped with a morphism of differential degree-weight-graded vertex superalgebras $\Phi_\mathcal{A}: O(\mathfrak{sg})\to (\mathcal{A}, d)$. 
Next we recall  from \cite{LLS1} the notion of $\mathfrak{sg}[t]$-modules containing that of 
$O(\mathfrak{sg})$-algebras. 
Recall the Lie superalgebra $\mathfrak{sg}[t]=\mathfrak{sg}\otimes\mathbb{K}[t]$ has a differential $d$ defined by 
$$
d: \mathfrak{sg}[t] \to \mathfrak{sg}[t], \quad (\xi, \eta)t^n\mapsto (\eta, 0)t^n.
$$
An $\mathfrak{sg}[t]$-\textit{module} is a degree-weight-graded complex $(\mathcal{A}, d_{\mathcal{A}})$ equipped with a Lie superalgebra morphism 
$$
\rho: \mathfrak{sg}[t]\to \mathrm{End} (\mathcal{A}), \quad (\xi, \eta)t^n\mapsto \rho((\xi, \eta)t^n)=L_{\xi, (n)}+\iota_{\eta, (n)},
$$
such that for all $x\in \mathfrak{sg}[t]$ we have 
\begin{enumerate}[$\bullet$]
     \setlength{\topsep}{1pt}
     \setlength{\partopsep}{0pt}
     \setlength{\itemsep}{1pt}
     \setlength{\parsep}{0pt}
     \setlength{\leftmargin}{20pt}
     \setlength{\rightmargin}{0pt}
     \setlength{\listparindent}{0pt}
     \setlength{\labelsep}{3pt}
     \setlength{\labelwidth}{30pt}
     \setlength{\itemindent}{0pt}
\item $\rho(dx)=[d_\mathcal{A},\rho(x)]$; 
\item $\rho(x)$ has degree $0$ whenever $x$ is even in $\mathfrak{sg}[t]$, and degree $-1$ whenever $x$ is odd, and has weight $-n$ if $x\in \mathfrak{sg}t^n$.
\end{enumerate}

In this paper, 
we assume that the differential has odd parity and that the action of $\mathfrak{sg}[t]$ on $\mathcal{A}$ with the discrete topology is continuous, that is, for any $v\in\mathcal{A}$, $t^k\mathfrak{sg}[t]\cdot v=0$ for some sufficiently large $k\in\mathbb{N}$.   Moreover we call an $\mathfrak{sg}[t]$-module
a  \textit{
differential} $\mathfrak{sg}[t]$-\textit{module}, emphasizing its  differential. 

For a differential $\mathfrak{sg}[t]$-module $(\mathcal{A}, d_\mathcal{A})$, we often write $L_{\xi, (n)}$ and $\iota_{\xi, (n)}$ as $L_{\xi, (n)}^\mathcal{A}$ and $\iota_{\xi, (n)}^\mathcal{A}$, respectively.

We will recall the notion of the chiral equivariant cohomology.  
Consider the semi-infinite Weil algebra $\mathcal{W}=\mathcal{W}(\mathfrak{g})$ associated with a finite-dimensional Lie algebra $\mathfrak{g}$ (see Example \ref{ex:semi-infinite_Weil_algebras}). Let $(\xi_i)_i$ be a basis of $\mathfrak{g}$ with the dual basis $(\xi^*_i)_i$ for $\mathfrak{g}^*$. Recall that the vertex superalgebra  $\mathcal{W}(\mathfrak{g})$ 
is degree-weight-graded. The weight and degree-grading come from the diagonalizable operator ${\omega_{\mathcal{W}}}_{(1)}=(\omega_{\mathcal{E}}+\omega_{\mathcal{S}})_{(1)}$ and the operator ${j_{bc}}_{(0)}+2{j_{\beta\gamma}}_{(0)}$, respectively. 
Here $\omega_{\mathcal{S}}:=\sum_{i=1}^{\dim \mathfrak{g}}\beta^{x^i}_{-1}\partial\gamma^{x^{*}_i}_{0}\mathbf{1}$,  $\omega_{\mathcal{E}}:=-\sum_{i=1}^{\dim \mathfrak{g}}b^{x^i}_{-1}\partial c^{x^{*}_i}_0\mathbf{1}$ and $j_{bc}:=-\sum_{i=1}^{\dim \mathfrak{g}}b^{\xi_i}_{-1}c^{\xi^*_i}_0\mathbf{1}$,  $j_{\beta\gamma}:=\sum_{i=1}^{\dim \mathfrak{g}}\beta^{\xi_i}_{-1}\gamma^{\xi^*_i}_0\mathbf{1}$.

Set
\begin{gather}
D:=J+K, \\ 
\quad J:=-\sum_{i, j=1}^{\dim \mathfrak{g}}\beta^{[\xi_i,\xi_j]}_{-1}\gamma^{\xi^*_j}_0 c^{\xi^*_i}_0\mathbf{1}-\frac{1}{2}\sum_{i, j=1}^{\dim \mathfrak{g}}c^{\xi^*_i}_0c^{\xi^*_j}_0b^{[\xi_i,\xi_j]}_{-1}\mathbf{1}, \quad K:=\sum_{i=1}^{\dim \mathfrak{g}}\gamma^{\xi^*_i}_0b^{\xi_i}_{-1}\mathbf{1}.
\end{gather}
Then the operator $D_{(0)}$ is a differential on $\mathcal{W}$.
The differential degree-weight-graded vertex superalgebra $(\mathcal{W}^\bullet, d_\mathcal{W})$ is called the \textit{semi-infinite Weil complex}, where $d_\mathcal{W}=D_{(0)}$ (\cite{FF91}).

Set
\begin{gather}
\Theta_\mathcal{W}^\xi:=\Theta_\mathcal{E}^\xi+\Theta_\mathcal{S}^\xi, \quad\\
 \Theta_\mathcal{E}^\xi:=\sum_{i=1}^{\dim \mathfrak{g}}b^{[\xi,\xi_i]}_{-1}c^{\xi^*_i}_0\mathbf{1}, \quad \Theta_\mathcal{S}^\xi:=-\sum_{i=1}^{\dim \mathfrak{g}}\beta^{[\xi,\xi_i]}_{-1}\gamma^{\xi^*_i}_0\mathbf{1},
\end{gather}
for $\xi\in \mathfrak{g}$. 
The following theorem is proved in \cite[theorem 5.11]{LL04}. 

\begin{theorem}[Lian-Linshaw]
The  vertex superalgebra morphism 
$O(\mathfrak{sg})\to(\mathcal{W}(\mathfrak{g}), D_{(0)}),$ $(\xi, \eta)(z)$ $\mapsto\Theta_{\mathcal{W}}^\xi(z)+b^\eta(z)$ defines an $O(\mathfrak{sg})$-algebra structure on $\mathcal{W}(\mathfrak{g})$.
\end{theorem}

We will use the following relations proved in \cite[Lemma 5.12]{LL}.

\begin{lemma}[Lian-Linshaw]
Let $\eta, \xi$ be elements of $\mathfrak{g}$ and $\xi^*$ an element of $\mathfrak{g}^*$. Then 
the following hold:
\begin{gather}
\Theta_\mathcal{W}^\xi(z)c^{\xi^*}(w)\sim c^{\mathrm{ad}^*\xi\cdot\xi^*}(w)(z-w)^{-1}, \\
D_{(0)}c^{\xi^*}_0\mathbf{1}=-\frac{1}{2}\sum_{i=1}^{\dim{\mathfrak{g}}}c^{\mathrm{ad}^*\xi_i\cdot\xi^*}_0c^{\xi^*_i}_0\mathbf{1}+\gamma^{\xi^*}_0\mathbf{1}, \\ D_{(0)}\gamma^{\xi^*}_0\mathbf{1}=\sum_{i=1}^{\dim{\mathfrak{g}}}\gamma^{\mathrm{ad}^*\xi_i\cdot\xi^*}_0c^{\xi^*_i}_0\mathbf{1},
\end{gather}
where the first formula stands for the OPE of the fields $\Theta_\mathcal{W}^\xi(z)$ and $c^{\xi^*}(z)$, and $\mathrm{ad}^*$ is the coadjoint action of $\mathfrak{g}$ on $\mathfrak{g}^*$. 
\end{lemma}

Recall the notion of the chiral horizontal, invariant and basic subspaces from \cite{LL,LLS1}.   
Let $(\mathcal{A}, d)$ be a differential $\mathfrak{sg}[t]$-module. The \textit{chiral horizontal, invariant} and \textit{basic subspaces} of $\mathcal{A}$ are respectively 
\begin{align*}
\mathcal{A}_{hor}&:=\bigl\{a\in \mathcal{A} \bigm| \iota_{\eta, (n)}a=0 \ \text{for all}\ \eta\in\mathfrak{g}, n\ge0 \bigr\}, \\
\mathcal{A}_{inv}&:=\bigl\{a\in \mathcal{A} \bigm| L_{\xi, (n)}a=0\ \text{for all}\ \xi\in\mathfrak{g}, n\ge0 \bigr\}, \ \text{and} \\
\mathcal{A}_{bas}&:=\mathcal{A}_{hor}\cap\mathcal{A}_{inv}.
\end{align*}
Note that if $(\mathcal{A}, d)$ is a differential $\mathfrak{sg}[t]$-module then the subspaces $\mathcal{A}_{hor}$ and $\mathcal{A}_{bas}$ are subcomplexes of $(\mathcal{A}, d)$.

We then recall the definitions of the chiral basic and equivariant cohomologies. 
Let $G$ be a compact connected Lie group. Set $\mathfrak{g}=\mathrm{Lie}(G)^{\mathbb{K}}$.
Let $(\mathcal{A}, d)$ be a differential $\mathfrak{sg}[t]$-module. Its \textit{chiral basic cohomology} $\mathbf{H}_{bas}{(\mathcal{A})}$ is the cohomology of the complex $(\mathcal{A}_{bas}, d|_{\mathcal{A}_{bas}})$. The \textit{chiral equivariant cohomology} $\mathbf{H}_{G}{(\mathcal{A})}$ of $(\mathcal{A}, d)$ is the chiral basic cohomology of the tensor product 
$
(\mathcal{W}(\mathfrak{g})\otimes \mathcal{A}, d_\mathcal{W}\otimes1+1\otimes d_\mathcal{A}).
$

\section{Chiral $W^*$-Modules}\label{section: Chiral $W^*$-Modules}

\subsection{Definition of Chiral $W^*$-Modules and the Chiral  Cartan Model}
Let $\mathfrak{g}$ be a finite-dimensional Lie algebra. 
We denote by $\langle c, \gamma \rangle$ or $\mathcal{W}'$ the subalgebra of the semi-infinite Weil algebra $\mathcal{W}=\mathcal{W}(\mathfrak{g})$ generated by $c^{\xi^*}_0\mathbf{1}$, $\gamma^{\xi^*}_0\mathbf{1}$ with $\xi^*\in\mathfrak{g}^*$.  
Note that $\mathcal{W}'$ is preserved by the differential $d_\mathcal{W}$. Therefore we have a 
subcomplex $(\mathcal{W}', d_{\mathcal{W}'})$, where $d_{\mathcal{W}'}:=d_{\mathcal{W}}|_{\mathcal{W}'}$. 
Note that $(\mathcal{W}', d_{\mathcal{W}'})$ is acyclic. This follows from the same argument as that for the proof of the acyclicity of  $(\mathcal{W}, d_\mathcal{W})$ in \cite[Proposition 5]{Akm93}. (See Section \ref{subsection: Chiral Equivariant Cohomology} for the definition of the semi-infinite Weil complex $(\mathcal{W}(\mathfrak{g}), d_\mathcal{W})$.)

We denote by $\delta(z-w)_-$ the formal distribution $\sum_{n\ge 0}z^{-n-1}w^n$. 

\begin{definition}\label{df: chiral W^*-modules}
A \textbf{chiral} $W^*$\textbf{-module} (with respect to  $\mathfrak{g}$) is a differential $\mathfrak{sg}[t]$-module $(\mathcal{A}, d_{\mathcal{A}})$ given a module structure over the vertex superalgebra $\langle c, \gamma \rangle$
$$
Y^\mathcal{A}(\ , z): \langle c, \gamma \rangle \to (\mathrm{End} \mathcal{A})[[z^{\pm1}]], 
$$
such that 
\begin{enumerate}[$(1)$\ ]
\setlength{\topsep}{1pt}
     \setlength{\partopsep}{0pt}
     \setlength{\itemsep}{1pt}
     \setlength{\parsep}{0pt}
     \setlength{\leftmargin}{35pt}
     \setlength{\rightmargin}{0pt}
     \setlength{\listparindent}{0pt}
     \setlength{\labelsep}{3pt}
     \setlength{\labelwidth}{15pt}
     \setlength{\itemindent}{0pt}
\item $[d_\mathcal{A}, Y^\mathcal{A}(x, z)]= Y^\mathcal{A}(d_{\mathcal{W}'}x, z)$, \quad for all  $x \in \langle c, \gamma \rangle$,
\item $[L_\xi^\mathcal{A}(z)_{-},Y^\mathcal{A}(c^{\xi^*}_0\mathbf{1}, w)]=Y^\mathcal{A}(c^{\mathrm{ad}^*\xi \cdot\xi^*}_0\mathbf{1}, w)\delta(z-w)_{-}$, for all elements $\xi$ of $\mathfrak{g}$ and all elements $\xi^*$ of $\mathfrak{g}^*$,
\item $[\iota_\xi^\mathcal{A}(z)_{-},Y^\mathcal{A}(c^{\xi^*}_0\mathbf{1}, w)]=\langle {\xi^*}, \xi \rangle \delta(z-w)_{-}$, for all $\xi\in \mathfrak{g}$ and all ${\xi^*} \in \mathfrak{g}^*$, 
\end{enumerate}
where 
$L_{\xi}^\mathcal{A}(z)_{-}:=\sum_{n\ge0}L_{\xi, (n)}^\mathcal{A}z^{-n-1}$ and 
$\iota_\xi^\mathcal{A}(z)_{-}:=\sum_{n\ge0}\iota_{\xi, (n)}^\mathcal{A}z^{-n-1}$ 
for an element $\xi$ of $\mathfrak{g}$.
\end{definition}

For a $\langle c, \gamma \rangle$-module $(\mathcal{A},  Y^\mathcal{A})$, we often use the following notation: 
\begin{align*}
Y^\mathcal{A}(c^{\xi^*}_0\mathbf{1}, z)&=c^{\xi^*, \mathcal{A}}(z)=\sum_{n\in\mathbb{Z}}c^{\xi^*, \mathcal{A}}_{(n)}z^{-n-1}, \\
Y^\mathcal{A}(\gamma^{\xi^*}_0\mathbf{1}, z)&=\gamma^{\xi^*, \mathcal{A}}(z)=\sum_{n\in\mathbb{Z}}\gamma^{\xi^*, \mathcal{A}}_{(n)}z^{-n-1}.
\end{align*}
for an element $\xi^*$ of $\mathfrak{g}^*$.

Note that the formula (2) in the preceding definition is equivalent to the commutation relations   
$[L_{\xi, (m)}^\mathcal{A},c_{(n)}^{\xi^*, \mathcal{A}}]=c_{(m+n)}^{\mathrm{ad}^*\xi\cdot\xi^*, \mathcal{A}}$
for all $m\in\mathbb{Z}_{\ge0}$ and $n\in\mathbb{Z}$. Similarly the formula (3) is equivalent to the relations $[\iota_{\xi, (m)}^\mathcal{A},c_{(n)}^{\xi^*, \mathcal{A}}]=\langle {\xi^*}, \xi \rangle\delta_{m+n, -1}$ for all $m\in\mathbb{Z}_{\ge0}$ and $ n\in\mathbb{Z}$.

Let $(\xi _i)_i$ be a basis of $\mathfrak{g}$ and $(\xi^*_i)_i$  the  dual basis for $\mathfrak{g}^*$.
Let $(\mathcal{A}, d_\mathcal{A}, Y^\mathcal{A})$ be a chiral $W^*$-module and $(\mathcal{B}, d_\mathcal{B})$  a differential $\mathfrak{sg}[t]$-module.
Set 
$$
\Phi=\Phi_{\mathcal{A}, \mathcal{B}}:=\mathrm{exp}(\phi(0)_{\ge 0}) \in GL(\mathcal{A}\otimes\mathcal{B}), 
$$
where $\phi(0)_{\ge 0}:= \sum_{i=1}^{\dim \mathfrak{g}}\sum_{n\ge 0}c^{\xi^*_i, \mathcal{A}}_{(-n-1)}\otimes\iota^\mathcal{B}_{\xi_i, (n)}$. 
Set
\begin{gather}
C(\mathcal{A}; \mathcal{B}):=\Phi((\mathcal{A}\otimes\mathcal{B})_{bas}), \\
d=d_{\mathcal{A}, \mathcal{B}}:=\Phi\circ(d_\mathcal{A}\otimes1 +1 \otimes d_\mathcal{B})\circ\Phi^{-1}|_{C(\mathcal{A}; \mathcal{B})}.
\end{gather}
Note that $C(\mathcal{A}; \mathcal{B})$ is degree-weight-graded as a subspace of the degree-weight-graded super vector space  $\mathcal{A}\otimes\mathcal{B}$ since $\Phi$ preserves the degree and weight-gradings. Then we have the following.  

\begin{lemma}\label{lem: CHIRAL CARTAN MODEL}
Let $(\mathcal{A}, d_\mathcal{A}, Y^\mathcal{A})$ be a chiral $W^*$-module and $(\mathcal{B}, d_\mathcal{B})$  a differential $\mathfrak{sg}[t]$-module. 
Then the map $\Phi=\Phi_{\mathcal{A}, \mathcal{B}}$ restricted to the chiral basic subspace is an isomorphism of degree-weight-graded complexes:
$$
\Phi: ((\mathcal{A}\otimes\mathcal{B})_{bas}, (d_\mathcal{A}\otimes1 +1\otimes d_\mathcal{B})|_{(\mathcal{A}\otimes\mathcal{B})_{bas}})\to (C(\mathcal{A}; \mathcal{B}), d_{\mathcal{A}, \mathcal{B}}).
$$
\end{lemma}

We call the complex $(C(\mathcal{A}; \mathcal{B}), d_{\mathcal{A}, \mathcal{B}})$ the \textbf{chiral Cartan model} for the differential $\mathfrak{sg}[t]$-module  
$(\mathcal{B}, d_\mathcal{B})$ with respect to the chiral $W^*$-module $(\mathcal{A}, d_\mathcal{A}, Y^\mathcal{A})$.

The following proposition is a variation of \cite[Theorem 4.6]{LL}. 

\begin{proposition}\label{prop: chiral Cartan model for chiral W^*-modules}
Let $(\mathcal{A}, d_\mathcal{A}, Y^\mathcal{A})$ be a chiral $W^*$-module and $(\mathcal{B}, d_\mathcal{B})$  a differential $\mathfrak{sg}[t]$-module.
Then the following equalities hold in $\mathrm{End} (\mathcal{A}\otimes\mathcal{B})$:
\begin{multline}\label{eq: the transformation of the differential}
\Phi\circ(d_\mathcal{A}\otimes 1+ 1\otimes d_\mathcal{B})\circ \Phi^{-1}\\
=d_\mathcal{A}\otimes 1+ 1\otimes d_\mathcal{B} - \sum_{i=1}^{\dim \mathfrak{g}}\sum_{n\ge0}\gamma^{\xi^*_i, \mathcal{A}}_{(-n-1)}\otimes\iota^\mathcal{B}_{\xi_i, (n)}+\sum_{i=1}^{\dim \mathfrak{g}}\sum_{n\ge0}c^{\xi^*_i, \mathcal{A}}_{(-n-1)}\otimes L^\mathcal{B}_{\xi_i, (n)}\\
+\sum_{i, j=1}^{\dim \mathfrak{g}}\sum_{m,n\ge0}c_{(n)}^{\xi^*_i, \mathcal{A}}c_{(-n-m-2)}^{\xi^*_j, \mathcal{A}}\otimes\iota_{[\xi_i,\xi_j], (m)}^\mathcal{B},
\end{multline}
and
\begin{multline}\label{eq: the transformation of L}
\Phi\circ(L^\mathcal{A}_{\xi, (n)}\otimes 1+1\otimes L^\mathcal{B}_{\xi, (n)})\circ\Phi^{-1} \\ 
\quad =L^\mathcal{A}_{\xi, (n)}\otimes 1+ 1\otimes L^\mathcal{B}_{\xi, (n)} +\sum_{i=1}^{\dim \mathfrak{g}}\sum_{0\le k<n}c^{\xi^*_i, \mathcal{A}}_{(n-k-1)}\otimes \iota^\mathcal{B}_{[\xi,\xi_i], (k)}, 
\end{multline}
\begin{equation}\label{eq: the transformation of iota}
\Phi \circ (\iota^\mathcal{A}_{\xi, (n)}\otimes 1 + 1\otimes \iota^\mathcal{B}_{\xi, (n)})\circ \Phi^{-1}=\iota^\mathcal{A}_{\xi, (n)}\otimes 1,
\end{equation}
for all $n \ge 0$, $\xi\in \mathfrak{g}$.
\end{proposition}
\begin{proof}
The assertion is proved by direct computations of $\mathrm{ad}(\phi(0)_{\ge0})^l$, 
where $\mathrm{ad}(\phi(0)_{\ge0})=[\phi(0)_{\ge0}, \ \ ]$.  
Note that we have $\mathrm{ad}(\phi(0)_{\ge0})^3(d_\mathcal{A}\otimes 1+ 1\otimes d_\mathcal{B})=0$, 
$\mathrm{ad}(\phi(0)_{\ge0})^2 (L^\mathcal{A}_{\xi, (n)}\otimes 1+1\otimes L^\mathcal{B}_{\xi, (n)})=0$
 and $(\mathrm{ad}(\phi(0)_{\ge0})^2(\iota^\mathcal{A}_{\xi, (n)}\otimes 1 + 1\otimes \iota^\mathcal{B}_{\xi, (n)})=0$. 
\end{proof}

By Proposition \ref{prop: chiral Cartan model for chiral W^*-modules}, 
 we have 
$
C(\mathcal{A}; \mathcal{B})=(\mathcal{A}_{hor}\otimes\mathcal{B})^{\Phi\mathfrak{g}[t]\Phi^{-1}},
$
where the right-hand side stands for the invariant subspace under the modified action of $\mathfrak{g}[t]$: 
$$
\mathfrak{g}[t]\to \mathrm{End}(\mathcal{A}\otimes\mathcal{B})\stackrel{\mathrm{Ad}(\Phi)}{\longrightarrow} \mathrm{End}(\mathcal{A}\otimes\mathcal{B}).
$$ 

\subsection{Commutative Cases}
Consider the case when the Lie group $G=T$ is commutative. Set $\mathfrak{t}=\mathrm{Lie}(T)^\mathbb{K}$. The \textit{small chiral  Cartan model} for $O(\mathfrak{st})$-algebras was introduced in \cite{LL}.
By Lemma \ref{lem: CHIRAL CARTAN MODEL}, we can also define the small chiral Cartan model for any differential $\mathfrak{st}[t]$-module. 
Note that the action of $\mathfrak{t}[t]$ on $\mathcal{W}(\mathfrak{t})$ is trivial since $\mathfrak{t}$ is commutative. 

Let $(\mathcal{A}, d_\mathcal{A})$ be a differential $\mathfrak{st}[t]$-module.
Set
$$
\mathcal{C}(\mathcal{A}):=\langle\gamma\rangle\otimes\mathcal{A}_{inv}\subset C(\mathcal{W}(\mathcal{\mathfrak{t}}); \mathcal{A}),
$$
where we denote by $\langle \gamma \rangle$ the subalgebra of $\mathcal{W}'$ generated by $\gamma^{\xi^*}_0\mathbf{1}$ with $\xi^*\in\mathfrak{g}$. 
Since $\mathfrak{t}$ is commutative, $d_\mathcal{W}|_{\langle \gamma \rangle}=0$ and $[\iota^{\mathcal{A}}_\xi(z)_-,L^{\mathcal{A}}_\eta(w)_- ]=0$ for any $\xi, \eta\in\mathfrak{t}$.   
Therefore it follows that $\mathcal{C}(\mathcal{A})$ is preserved by the differential $d_{\mathcal{W}(\mathfrak{t}), \mathcal{A}}$. 
We can prove the following lemma by the same argument as in \cite[Theorem 6.4]{LL}, where the case when $\mathcal{A}$ is an $O(\mathfrak{st})$-algebra is considered.

\begin{proposition}\label{prop: small Cartan model}
The inclusion
$$
(\mathcal{C}(\mathcal{A}), d_{\mathcal{W}(\mathfrak{t}), \mathcal{A}}|_{\mathcal{C}(\mathcal{A})})\to(C(\mathcal{W}(\mathcal{\mathfrak{t}}); \mathcal{A}), d_{\mathcal{W}(\mathfrak{t}), \mathcal{A}}), 
$$
is a quasi-isomorphism.
\end{proposition}

We call $(\mathcal{C}(\mathcal{A}), d_{\mathcal{W}(\mathfrak{t}), \mathcal{A}}|_{\mathcal{C}(\mathcal{A})})$ the \textbf{small chiral  Cartan model} for the differential $\mathfrak{st}[t]$-module $(\mathcal{A},  d_\mathcal{A})$.

The following lemma is proved by an argument similar to that in \cite[Lemma 2.6]{LLS1}, where the case when $\mathcal{A}$ is contained in an $O(\mathfrak{st})$-algebra is considered.

\begin{lemma}\label{lem: small Cartan model}
$$
\Phi_{\mathcal{W}(\mathfrak{t}), \mathcal{A}}^{-1}(\mathcal{C}(\mathcal{A}))=(\langle c, \gamma\rangle\otimes\mathcal{A}_{inv})_{hor}.
$$
\end{lemma}

Let $(\mathcal{A}, d_\mathcal{A}, Y^\mathcal{A})$ be a chiral $W^*$-module. 
Set
$$
C'(\mathcal{A}, \mathcal{W}(\mathfrak{t})):=(\Phi_{\mathcal{A}, \mathcal{W}(\mathfrak{t})}\circ\tau\circ\Phi_{\mathcal{W}(\mathfrak{t}), \mathcal{A}}^{-1})(\mathcal{C}(\mathcal{A}))
\subset C(\mathcal{A}, \mathcal{W}(\mathfrak{t})),
$$
where $\tau: \mathcal{W}(\mathfrak{t})\otimes\mathcal{A}\to\mathcal{A}\otimes\mathcal{W}(\mathfrak{t})$ is the switching map. Then the following proposition follows from  Proposition \ref{prop: small Cartan model}.

\begin{proposition}\label{prop: small Cartan model for chiral W^*-modules}
There exists a canonical isomorphism
$$
H(C'(\mathcal{A}; \mathcal{W}(\mathfrak{t})), d'_{\mathcal{A}, \mathcal{W}(\mathfrak{t})})\cong \mathbf{H}_{T}{\mathcal{(A)}},
$$
where $d'_{\mathcal{A}, \mathcal{W}(\mathfrak{t})}:=d_{\mathcal{A}, \mathcal{W}(\mathfrak{t})}|_{C'(\mathcal{A}, \mathcal{W}(\mathfrak{t}))}$.
\end{proposition}

The following lemma leads us to the next important proposition.

\begin{lemma}
$$
C'(\mathcal{A}, \mathcal{W}(\mathfrak{t}))=\mathcal{A}_{bas}\otimes\langle c, \gamma\rangle.
$$
\end{lemma}
\begin{proof}
The operators $c^{\xi^*, \mathcal{A}}_{(n)}$ and $L_{\xi, (k)}^\mathcal{A}$ commute with each other since $\mathfrak{t}$ is commutative. Therefore the operators  $c^{\xi^*, \mathcal{A}}_{(n)}$ preserve the subspace $\mathcal{A}^{\mathfrak{t}[t]}$. 
Thus we have $\Phi_{\mathcal{A}, \mathcal{W}(\mathfrak{t})}(\mathcal{A}^{\mathfrak{t}[t]}\otimes \langle c, \gamma\rangle)$ $\subset \mathcal{A}^{\mathfrak{t}[t]}\otimes \langle c, \gamma\rangle$. Therefore  
the subspace
$$
\Phi_{\mathcal{A}, \mathcal{W}(\mathfrak{t})}\bigl(\big( \mathcal{A}^{\mathfrak{t}[t]}\otimes \langle c, \gamma\rangle\big)_{hor}\bigr)=\Bigl(\Phi_{\mathcal{A}, \mathcal{W}(\mathfrak{t})}\bigl(\mathcal{A}^{\mathfrak{t}[t]}\otimes \langle c, \gamma\rangle\bigr)\Bigr)^{\Phi_{\mathcal{A}, \mathcal{W}(\mathfrak{t})} \mathfrak{t}_{-1}[t] \Phi_{\mathcal{A}, \mathcal{W}(\mathfrak{t})}^{-1}}, 
$$
is contained in $\bigl(\mathcal{A}^{\mathfrak{t}[t]}\otimes\langle c, \gamma \rangle \bigr)^{\Phi_{\mathcal{A}, \mathcal{W}(\mathfrak{t})} \mathfrak{t}_{-1}[t] \Phi_{\mathcal{A}, \mathcal{W}(\mathfrak{t})}^{-1}}$. 
By the formula \eqref{eq: the transformation of iota}, we have
\begin{equation}\label{eq: rewrite A_bas c gamma}
\bigl(\mathcal{A}^{\mathfrak{t}[t]}\otimes\langle c, \gamma \rangle \bigr)^{\Phi_{\mathcal{A}, \mathcal{W}(\mathfrak{t})} \mathfrak{t}_{-1}[t] \Phi_{\mathcal{A}, \mathcal{W}(\mathfrak{t})}^{-1}}=\mathcal{A}_{bas}\otimes \langle c, \gamma \rangle. 
\end{equation}
Thus we have 
\begin{equation}\label{eq: C'(A; W) is contained in A_bas c gamma}
\Phi_{\mathcal{A}, \mathcal{W}(\mathfrak{t})}\bigl(\big( \mathcal{A}^{\mathfrak{t}[t]}\otimes \langle c, \gamma\rangle\big)_{hor}\bigr)\subset\mathcal{A}_{bas}\otimes \langle c, \gamma \rangle.
\end{equation}
On the other hand, we have
\begin{align*}
\Phi_{\mathcal{A}, \mathcal{W}(\mathfrak{t})}^{-1}(\mathcal{A}_{bas}\otimes \langle c, \gamma \rangle)&=\Phi_{\mathcal{A}, \mathcal{W}(\mathfrak{t})}^{-1}\Bigl(\bigl(\mathcal{A}^{\mathfrak{t}[t]}\otimes\langle c, \gamma \rangle \bigr)^{\Phi_{\mathcal{A}, \mathcal{W}(\mathfrak{t})} \mathfrak{t}_{-1}[t] \Phi_{\mathcal{A}, \mathcal{W}(\mathfrak{t})}^{-1}}\Bigr) \\
&=\Bigl(\Phi_{\mathcal{A}, \mathcal{W}(\mathfrak{t})}^{-1}\bigl(\mathcal{A}^{\mathfrak{t}[t]}\otimes\langle c, \gamma \rangle \bigr)\Bigr)^{\mathfrak{t}_{-1}[t]} \\
&\subset \bigl(\mathcal{A}^{\mathfrak{t}[t]}\otimes\langle c, \gamma \rangle \bigr)^{\mathfrak{t}_{-1}[t]}=\bigl(\mathcal{A}^{\mathfrak{t}[t]}\otimes\langle c, \gamma \rangle \bigr)_{hor}.
\end{align*}
The first equality follows from \eqref{eq: rewrite A_bas c gamma} and  the  inclusion follows from the formula $\Phi_{\mathcal{A}, \mathcal{W}(\mathfrak{t})}^{-1}=\mathrm{exp}\bigl(-\sum_{i=1}^{\dim \mathfrak{t}}\sum_{n\ge 0}c^{\xi^*_i, \mathcal{A}}_{(-n-1)}\otimes\iota^\mathcal{W}_{\xi_i, (n)}\bigr)$. Together with \eqref{eq: C'(A; W) is contained in A_bas c gamma}, we have 
$$
\Phi_{\mathcal{A}, \mathcal{W}(\mathfrak{t})}\bigl(\big( \mathcal{A}^{\mathfrak{t}[t]}\otimes \langle c, \gamma\rangle\big)_{hor}\bigr)=\mathcal{A}_{bas}\otimes \langle c, \gamma \rangle.
$$
By Lemma \ref{lem: small Cartan model}, the left-hand side is equal to $C'(\mathcal{A}, \mathcal{W}(\mathfrak{t}))$. This completes the proof.
\end{proof}

The following proposition is a chiral analogue of  \cite[Theorem 4.3.1]{GS99}.

\begin{proposition}
Let $(\mathcal{A}, d_\mathcal{A}, Y^\mathcal{A})$ be a chiral $W^*$-module. 
The inclusion
\begin{equation}\label{eq: basic into basic otimes W'}
(\mathcal{A}_{bas}, d_\mathcal{A})\to (C'(\mathcal{A}; \mathcal{W}(\mathfrak{t})), d'_{\mathcal{A}, \mathcal{W}(\mathfrak{t})})
, \quad a\mapsto a\otimes \mathbf{1}, 
\end{equation}
is a quasi-isomorphism. 
\end{proposition}
\begin{proof}
Set 
\begin{gather*}
d:=d'_{\mathcal{A}, \mathcal{W}(\mathfrak{t})}=d_1+d_2, \\
d_1:=1\otimes d_{\mathcal{W'}}=\sum_{i=1}^{\dim \mathfrak{t}}\sum_{n\ge 0}1\otimes\gamma^{\xi^*_i, \mathcal{W}}_{(-n-1)}b^{\xi_i, \mathcal{W}}_{(n)}, \\
d_2:=d_\mathcal{A}\otimes1 - \sum_{i=1}^{\dim \mathfrak{t}}\sum_{n\ge 0}\gamma^{\xi^*_i, \mathcal{A}}_{(-n-1)}\otimes b^{\xi_i, \mathcal{W}}_{(n)}.
\end{gather*}
Since $(\mathcal{W}', d_{\mathcal{W}'})$ is acyclic, we have 
\begin{equation}\label{eq: the cohomology of (A_bas otimes W', d_1)}
H^i(\mathcal{A}_{bas}\otimes \langle c, \gamma \rangle, d_1)=
\begin{cases}
\mathcal{A}_{bas}\otimes\mathbb{K}\mathbf{1}, & \text{when}\ i=0,\\
0, & \text{otherwise}.
\end{cases}
\end{equation}
For $i, j \ge 0$, we set 
$
C^i:=\mathcal{A}_{bas}\otimes {\mathcal{W}'}^i, 
C_j:=\bigoplus_{0\le i \le j}C^i,
$
and $C_{-1}:=0$.
Note that $d_1$ has degree $1$ with respect to this grading $C'(\mathcal{A}; \mathcal{W}(t))=\bigoplus_{i\ge 0}C^i$ and $d_2$ preserves that filtration $C'(\mathcal{A}; \mathcal{W}(t))=\bigcup_{j\ge 0}C_j$.
First we prove that for any $j\ge0$ if $\mu\in C_j$ and $d\mu=0$ then there exist an element $\nu\in C_{j-1}$ and $a\in \mathcal{A}_{bas}$ such that $\mu=d\nu+a\otimes \mathbf{1}$ and $d_\mathcal{A}a=0$. 
This implies that the map induced by the inclusion \eqref{eq: basic into basic otimes W'} on the cohomology is surjective. We use induction on $j$. Assume $j=0$. Let $\mu\in C_0$ with $d\mu=0$. Since $C_0=C^0=\mathcal{A}_{bas}\otimes\mathbb{K}\mathbf{1}$, we have an element $a\in\mathcal{A}_{bas}$ such that $\mu=a\otimes\mathbf{1}$. 
Considering the degree in the formula $0=d\mu=d_1\mu+d_2\mu$, we have $d_1\mu=0$. Therefore $d_2\mu=0$. From   $d_2\mu=d_\mathcal{A}a\otimes\mathbf{1}$, we see $d_\mathcal{A}a=0$. Thus the proof for $j=0$ is completed. Next we assume $j>0$. Let $\mu\in C_j$ with $d\mu=0$. We can write $\mu$ as 
$
\mu=\mu_j+\mu_{j-1}+\dots+\mu_0
$
for some $\mu_i\in C^i$ with $i=1, \dots, j$. 
Since $d_1\mu_j$ is the component of $d\mu$ with the maximum degree, we have $d_1\mu_j=0$. By \eqref{eq: the cohomology of (A_bas otimes W', d_1)} and $j\neq0$, we have an element $\nu_{j-1}\in C^{j-1}$ such that $\mu_j=d_1\nu_{j-1}$. Therefore we have 
\begin{align*}
\mu&=d_1\nu_{j-1}+ (\text{terms in}\ C_{j-1}) \\
&=(d\nu_{j-1}-d_2\nu_{j-1})+(\text{terms in}\ C_{j-1}) \\
&=d\nu_{j-1}+\mu',
\end{align*}
where $\mu'$ is some element of $C_{j-1}$. From $d\mu=0$, we have $d\mu'=0$. By the induction hypothesis, we have an element $\nu'\in C_{j-2}$ and $a'\in \mathcal{A}_{bas}$ such that $\mu'=d\nu'+a'\otimes\mathbf{1}$ and $d_\mathcal{A}a'=0$. Therefore we have 
$\mu=d\nu_{j-1}+\mu'=d(\nu_{j-1}+\nu')+a'\otimes\mathbf{1}$.

It remains to show that the map induced by \eqref{eq: basic into basic otimes W'} on the cohomology is injective. It suffices to show that if $a$ is an element of $\mathcal{A}_{bas}$ such that  $d_\mathcal{A}a=0$ and $a\otimes\mathbf{1}=d\nu$ for some $\nu\in \mathcal{A}_{bas}\otimes\mathcal{W}'$ then there exists an element $b$ of $\mathcal{A}_{bas}$ such that $a=d_\mathcal{A}b$.
Denote by $\mathcal{W}'^{(i, j)}$ the subspace of $\mathcal{W}'$ of degree $i$ and $j$ with respect to the operators ${j_{bc}}(0)_{\ge0}:=\sum_{i=1}^{\dim \mathfrak{t}}\sum_{n\ge 0}c^{\xi^*_i}_{(-n-1)}b^{\xi_i}_{(n)}$ and ${j_{\beta\gamma}}(0)_{\ge0}:=\sum_{i=1}^{\dim \mathfrak{t}}\sum_{n\ge 0}\gamma^{\xi^*_i}_{(-n-1)}\beta^{\xi_i}_{(n)}$, respectively. Note that  $\mathcal{W}'{(0, 0)}=\mathbb{K}\mathbf{1}$ and  $\mathcal{W}'=\bigoplus_{i, j \in \mathbb{N}}\mathcal{W}'^{(i, j)}$. Set 
$
D^{(i, j)}:= \mathcal{A}_{bas}\otimes \mathcal{W}'^{(i, j)}.
$
Then we have the following homogeneous operators: 
\begin{equation}\label{eq: degree of d_1, d_A, d_3}
\begin{aligned}
d_1: D^{(i, j)}\to D^{(i-1, j+1)}, \\
d_\mathcal{A}\otimes1: D^{(i, j)}\to D^{(i, j)}, \\
d_3: D^{(i, j)}\to D^{(i-1, j)},
\end{aligned}
\end{equation}
where $d_3:=-\sum_{i=1}^{\dim \mathfrak{t}}\sum_{n\ge 0}\gamma^{\xi^*_i, \mathcal{A}}_{(-n-1)}\otimes b^{\xi_i, \mathcal{W}}_{(n)}$. 

Let $a$ be an element of $\mathcal{A}_{bas}$ such that $d_\mathcal{A}a=0$ and $a\otimes \mathbf{1}=d\nu$, where $\nu$ is an element of $\mathcal{A}_{bas}\otimes\mathcal{W}'$. We can write $\nu$ as 
$
\nu=\sum_{i, j\ge0}\nu^{(i, j)},
$
where $\nu^{(i, j)}$ is an element of $D^{(i, j)}$ and the elements $\nu^{(i, j)}$ are $0$ for all but finitely many $(i, j)$. From $a\otimes\mathbf{1} =d\nu$ and \eqref{eq: degree of d_1, d_A, d_3}, we have 
\begin{equation}\label{eq: D^(0, 0) term}
a\otimes\mathbf{1}=(d_\mathcal{A}\otimes1)\nu^{(0, 0)}+d_3\nu^{(1, 0)},
\end{equation}
and 
\begin{equation}\label{eq: D^(i, j) term}
0=(d_\mathcal{A}\otimes1)\nu^{(i, j)}+d_3\nu^{(i+1, j)}+d_1\nu^{(i+1, j-1)},
\end{equation}
for all $i, j\ge0$ with $i>0$ or $j>0$, where we set $\nu^{(i, -1)}:=0$. 
From \eqref{eq: D^(i, j) term} with $j=0$, we have 
\begin{equation}\label{eq: relations in D^(i, 0)}
0=(d_\mathcal{A}\otimes1)\nu^{(i, 0)}+d_3\nu^{(i+1, 0)},
\end{equation}
for all $i>0$.
Note that we have
\begin{equation}\label{eq: rewrite d_3}
d_3=\Bigl[d_\mathcal{A}\otimes1,\sum_{i=1}^{\dim \mathfrak{t}}\sum_{n\ge0}c^{\xi^*_i, \mathcal{A}}_{(-n-1)}\otimes b^{\xi_i, \mathcal{W}}_{(n)}\Bigr].
\end{equation}
Indeed by the definition of chiral $W^*$-modules and the commutativity of $\mathfrak{t}$, we have   
$
\gamma^{\xi^*_i, \mathcal{A}}_{(-n-1)}=[d_\mathcal{A},c^{\xi^*_i, \mathcal{A}}_{(-n-1)}]
$
for any $n\ge0$ and $i=1, \dots, \dim \mathfrak{t}$. The formula \eqref{eq: rewrite d_3} follows from this relation and the definition of $d_3$. We set $S:=\sum_{i=1}^{\dim \mathfrak{t}}\sum_{n\ge0}c^{\xi^*_i, \mathcal{A}}_{(-n-1)}\otimes b^{\xi_i, \mathcal{W}}_{(n)}$. Note that we have $[S, [d_\mathcal{A}\otimes1,S]]=0$. Then we have 
\begin{equation}\label{eq: d_3v^(1, 0) is exact}
d_3\nu^{(1, 0)}=(d_\mathcal{A}\otimes1)\sum_{i=1}^N S^{[i]}\nu^{(i, 0)}-S^{[N]}(d_\mathcal{A}\otimes1)\nu^{(N, 0)}, 
\end{equation}
for any $N>0$. 
This is proved by induction on $N$ with formulae \eqref{eq: relations in D^(i, 0)}, \eqref{eq: rewrite d_3} and $[S, [d_\mathcal{A}\otimes1,S]]=0$. We have a natural number $N(>0)$ such that $\nu^{(N, 0)}=0$. Therefore  from the formulae \eqref{eq: D^(0, 0) term} and \eqref{eq: d_3v^(1, 0) is exact}, we have 
$
a\otimes\mathbf{1}=(d_\mathcal{A}\otimes1)\Bigl(\nu^{(0, 0)}+\sum_{i=1}^N S^{[i]}\nu^{(i, 0)}\Bigr).
$
Notice that $S^{[l]}\nu^{(l, 0)}$ belongs to $D^{(0, 0)}$ since $S$ maps $D^{(i, j)}$ into $D^{(i-1, j)}$.
Therefore we have 
$
\nu^{(0, 0)}+\sum_{i=1}^N S^{[i]}\nu^{(i, 0)}=b\otimes\mathbf{1},
$ 
for some $b\in \mathcal{A}_{bas}$. Thus we have $a=d_\mathcal{A}b$. This completes the proof. 
\end{proof}

By the preceding proposition and Proposition \ref{prop: small Cartan model for chiral W^*-modules}, we have the following theorem.

\begin{theorem}\label{thm: CHIRAL BASIC=CHIRAL EQUIVARIANT}
Let $G$ be a compact connected Lie group with the Lie algebra $\mathfrak{g}=\mathrm{Lie}(G)^\mathbb{K}$. 
Let $(\mathcal{A}, d_\mathcal{A}, Y^\mathcal{A})$ be a chiral $W^*$-module with respect to $\mathfrak{g}$. Assume that $G$ is commutative.  Then there exists a canonical isomorphism 
\begin{equation}\label{eq: ch basic = ch equiv}
\mathbf{H}_{bas}{(\mathcal{A})}\cong\mathbf{H}_{G}{(\mathcal{A})}.
\end{equation}
\end{theorem}

\subsection{More on Chiral $W^*$-Modules}
Let $\mathfrak{g}$ be a finite-dimensional Lie algebra. Let $(\xi_i)_i$ be a basis of $\mathfrak{g}$ and $(\xi^*_i)_i$ the  dual basis for $\mathfrak{g}^*$. 

Denote by $\langle c \rangle$ the subalgebra of $\mathcal{W}'$ generated by $c^{\xi^*}_0\mathbf{1}$ with $\xi^*\in\mathfrak{g}$.

\begin{lemma}\label{lem: sufficient condition for d-compatibility}
Let $(A, d_\mathcal{A})$ be a differential $\mathfrak{sg}[t]$-module and $Y^\mathcal{A}$ a $\mathcal{W}'$-module structure on $\mathcal{A}$. Assume 
\begin{equation}\label{eq: d-compatibility for c}
Y^\mathcal{A}(d_{\mathcal{W}'}c^{\xi^*}_0\mathbf{1}, z)=[d_\mathcal{A},c^{\xi^*, \mathcal{A}}(z)],
\end{equation}
for all $\xi^*\in\mathfrak{g}$. Then the following holds:
\begin{equation}\label{eq: d-compatibility for x}
Y^\mathcal{A}(d_{\mathcal{W}'}x, z)=[d_\mathcal{A},Y^\mathcal{A}(x, z)],
\end{equation}
for any $x\in \mathcal{W}'$.
\end{lemma}
\begin{proof}
Let $S\subset \mathcal{W}'$ be a subset. Using the fact that $d_{{\mathcal{W}'}}$ commutes with the translation operator, we can check by induction that if \eqref{eq: d-compatibility for x} holds for all $x\in S$ then \eqref{eq: d-compatibility for x} holds for all $x\in \langle S \rangle$. 
Therefore it suffices to show that \eqref{eq: d-compatibility for x} holds for $x=\gamma^{\xi^*}_0\mathbf{1}$ with $\xi^*\in\mathfrak{g}^*$. 
Note that \eqref{eq: d-compatibility for x} holds for all $x\in \langle c \rangle$ by the assumption \eqref{eq: d-compatibility for c}. Let $\xi^*\in\mathfrak{g}^*$. From the formula 
\begin{equation}\label{eq: formula for d_W c}
\gamma^{\xi^*}_0\mathbf{1}=d_{\mathcal{W}'}c^{\xi^*}_0\mathbf{1}+\frac{1}{2}\sum_{i=1}^{\dim \mathfrak{g}}c^{\mathrm{ad}^*\xi_i\cdot\xi^*}_0c^{\xi^*_i}_0\mathbf{1},
\end{equation}
we have 
$$
Y^\mathcal{A}(d_{\mathcal{W}'}\gamma^{\xi^*}_0\mathbf{1}, z)=\frac{1}{2}\sum_{i=1}^{\dim \mathfrak{g}}Y^\mathcal{A}(d_{\mathcal{W}'}c^{\mathrm{ad}^*\xi_i\cdot\xi^*}_0c^{\xi^*_i}_0\mathbf{1}, z).
$$
Since \eqref{eq: d-compatibility for x} holds for all $x\in \langle c \rangle$, the right-hand side equals 
$$
\Bigl[d_\mathcal{A},\frac{1}{2}\sum_{i=1}^{\dim \mathfrak{g}}Y^\mathcal{A}(c^{\mathrm{ad}^*\xi_i\cdot\xi^*}_0c^{\xi^*_i}_0\mathbf{1}, z)\Bigr].
$$ 
From \eqref{eq: d-compatibility for c} and \eqref{eq: formula for d_W c}, we can see this is equal to $[d_\mathcal{A},\gamma^{\xi^*, \mathcal{A}}(z)]$. 
\end{proof}

The following proposition is useful for checking that a differential $\mathfrak{sg}[t]$-module with a $\mathcal{W}'$-module structure is  a chiral $W^*$-module.

\begin{proposition}\label{prop: sufficient condition for chiral W^*-modules}
Let $(A, d_\mathcal{A})$ be a differential $\mathfrak{sg}[t]$-module and $Y^\mathcal{A}$ a $\mathcal{W}'$-module structure on $\mathcal{A}$. Assume the following:
\begin{gather}
\label{assump: [iota,gamma]=0}
[\iota^\mathcal{A}_\xi(z)_{-}, \gamma^{\xi^*, \mathcal{A}}(w)]=0, \\
\label{assump: Y(dc)=[d,Y(c)]}
Y^\mathcal{A}(d_{\mathcal{W}'}c^{\xi^*}_0\mathbf{1}, z)=[d_\mathcal{A},c^{\xi^*, \mathcal{A}}(z)], \\
\label{assump: [iota,c]=delta}
[\iota^\mathcal{A}_\xi(z)_{-}, c^{\xi^*, \mathcal{A}}(w)]=\langle \xi^*, \xi \rangle \delta(z-w)_{-},
\end{gather}
for all $\xi\in\mathfrak{g}$ and $\xi^*\in\mathfrak{g}^*$. Then the triple $(\mathcal{A}, d_\mathcal{A}, Y^\mathcal{A})$ is a chiral $W^*$-module. 
\end{proposition}
\begin{proof}
By Lemma \ref{lem: sufficient condition for d-compatibility}, it suffices to check 
$$
[L_\xi^\mathcal{A}(z)_{-},c^{\xi^*, \mathcal{A}}(w)]=c^{\mathrm{ad}^*\xi\cdot\xi^*, \mathcal{A}}(w)\delta(z-w)_{-},
$$
for all $\xi\in\mathfrak{g}$ and $\xi^*\in\mathfrak{g}^*$. 
Let $\xi\in\mathfrak{g}$ and $\xi^*\in\mathfrak{g}^*$. Applying $\mathrm{ad}(d_\mathcal{A})$ to the both sides of \eqref{assump: [iota,c]=delta}, we have 
$$
[d_\mathcal{A},[\iota^\mathcal{A}_\xi(z)_{-},c^{\xi^*, \mathcal{A}}(w)]]=0.
$$
Therefore from \eqref{assump: Y(dc)=[d,Y(c)]} and $[d_\mathcal{A},\iota^\mathcal{A}_\xi(z)_{-}]=L_\xi^\mathcal{A}(z)_{-}$, we have 
$$
[L_\xi^\mathcal{A}(z)_{-},c^{\xi^*, \mathcal{A}}(w)]=[\iota_\xi^\mathcal{A}(z)_{-},Y^\mathcal{A}(d_{\mathcal{W}'}c^{\xi^*}_0\mathbf{1}, w)].
$$
By \eqref{assump: [iota,c]=delta}, \eqref{assump: [iota,gamma]=0} and the formula $d_{\mathcal{W}'}c^{\xi^*}_0\mathbf{1}=\gamma^{\xi^*}_0\mathbf{1}-1/2\sum_{i=1}^{\dim \mathfrak{g}}c^{\mathrm{ad}^*\xi_i\cdot\xi^*}_0c^{\xi^*_i}_0\mathbf{1}$, we can see the right-hand side equals  
$$
-\frac{1}{2}\sum_{i=1}^{\dim \mathfrak{g}}\langle \mathrm{ad}^* \xi_i \cdot\xi^*, \xi \rangle \delta(z-w)_{-}c^{\xi^*_i}(w)+ 
\frac{1}{2}\sum_{i=1}^{\dim \mathfrak{g}}c^{\mathrm{ad}^*\xi_i\cdot\xi^*}(w)\langle  \xi^*_i, \xi\rangle \delta(z-w)_{-}.
$$
This is just equal to $c^{\mathrm{ad}^*\xi\cdot\xi^*, \mathcal{A}}(w)\delta(z-w)_{-}$.
\end{proof}

The following is useful when we equip a differential $\mathfrak{sg}[t]$-module with a chiral $W^*$-module structure. 

\begin{proposition}\label{prop: construction of chiral W^*-modules}
Let $(\mathcal{A}, d_\mathcal{A})$ be a differential $\mathfrak{sg}[t]$-module. Suppose given a module structure of the vertex superalgebra $\langle  c \rangle$ on $\mathcal{A}$
$$
Y^\mathcal{A}_0: \langle  c \rangle \to (\mathrm{End} \mathcal{A})[[z^{\pm1}]],
$$
such that 
\begin{equation}\label{eq: [c,[d,c]]=0}
[ c^{\xi^*, \mathcal{A}}(z),[d_\mathcal{A}, c^{\eta^*, \mathcal{A}}(w)]]=0,
\end{equation}
for all $\xi^*, \eta^*\in \mathfrak{g}^*$.
Then there exists a unique $\langle  c, \gamma \rangle$-module structure on $\mathcal{A}$,
$$
Y^\mathcal{A}: \langle  c, \gamma \rangle \to (\mathrm{End} \mathcal{A})[[z^{\pm1}]],
$$
such that $Y^\mathcal{A}|_{\langle  c \rangle}=Y^\mathcal{A}_0$ and $[d_\mathcal{A}, Y^\mathcal{A}(x, z)]=Y^\mathcal{A}(d_{\mathcal{W}'}x, z)$ for all $x\in \mathcal{W}'$.
Moreover if the operation $Y^\mathcal{A}$ satisfies 
\begin{align}
[\iota^\mathcal{A}_\xi(z)_{-}, c^{\xi^*, \mathcal{A}}(w)]&=\langle \xi^*, \xi \rangle \delta(z-w)_{-},\\
\label{eq: curvature is horizontal in prop: construction of chiral W^*-modules}
[\iota^\mathcal{A}_\xi(z)_{-},\gamma^{\xi^*, \mathcal{A}}(w)]&=0, 
\end{align}
for all $\xi \in \mathfrak{g}$ and $\xi^*\in\mathfrak{g}^*$, then the triple 
$(\mathcal{A}, d_\mathcal{A}, Y^\mathcal{A})$ is a chiral $W^*$-module.
\end{proposition}
\begin{proof}
The uniqueness of the operation $Y^\mathcal{A}$ follows from the formula $d_{\mathcal{W}'} c^j_0\mathbf{1}=\gamma^j_0\mathbf{1}-1/2\sum_{i, k=1}^{\dim{g}}\Gamma_{i k}^j  c^{i}_0c^{k}_0\mathbf{1}$, where $\Gamma_{i j}^k$ is the structure constants of the Lie algebra $\mathfrak{g}$,  that is, $[\xi_i, \xi_j]=\sum_{k=1}^{\dim \mathfrak{g}}\Gamma_{i j}^k\xi_k$ for $i, j=1, \dots, \dim \mathfrak{g}$.
We check the existence of such a $\mathcal{W}'$-module structure $Y^\mathcal{A}$.
We set
\begin{equation}\label{eq: df of gamma(z)}
\gamma^{\xi^*_j, \mathcal{A}}(z):=[d_\mathcal{A},  c^{\xi^*_j, \mathcal{A}}(z)]+\frac{1}{2}\sum_{i, k=1}^{\dim \mathfrak{g}}\Gamma_{i k}^j \,{\baselineskip0pt\lineskip0.3pt\vcenter{\hbox{$\cdot$}\hbox{$\cdot$}}\,c^{\xi^*_i, \mathcal{A}}(z) c^{\xi^*_k, \mathcal{A}}(z)\,\vcenter{\hbox{$\cdot$}\hbox{$\cdot$}}},
\end{equation}
for $j=1, \dots, \dim \mathfrak{g}$. 
Note that these operators have even parity. We check $[c^{\xi^*_i, \mathcal{A}}(z), \gamma^{\xi^*_j, \mathcal{A}}(w)]=0$ and $[\gamma^{\xi^*_i, \mathcal{A}}(z),\gamma^{\xi^*_j, \mathcal{A}}(w)]=0$ for $i, j=1, \dots, \dim \mathfrak{g}$. This implies the existence of a $\mathcal{W}'$-module structure $Y^\mathcal{A}$  such that $Y^\mathcal{A}|_{\langle c \rangle}=Y^\mathcal{A}_0$. Applying $\mathrm{ad}( d_\mathcal{A})$ to the both sides of \eqref{eq: [c,[d,c]]=0}, we have 
\begin{equation}\label{eq: [[d,c],[d,c]]=0}
\bigl[[d_\mathcal{A},c^{\xi^*, \mathcal{A}}(z)],[d_\mathcal{A},c^{\eta^*, \mathcal{A}}(w)]\bigr]=0,
\end{equation}
for all $\xi^*, \eta^*\in\mathfrak{g}^*$. 
We have
$
[\gamma^{\xi^*_j, \mathcal{A}}(z),\gamma^{\xi^*_{\Tilde{j}}, \mathcal{A}}(w)]=0
$
by \eqref{eq: [c,[d,c]]=0} and \eqref{eq: [[d,c],[d,c]]=0}.
The formula $[\gamma^{\xi^*_i, \mathcal{A}}(z),\gamma^{\xi^*_j, \mathcal{A}}(w)]=0$ follows directly from \eqref{eq: [c,[d,c]]=0}. Thus we have a $\mathcal{W}'$-module structure $Y^\mathcal{A}$ such that $Y^\mathcal{A}|_{\langle c \rangle}=Y^\mathcal{A}_0$. It remains to check $[d_\mathcal{A},Y^\mathcal{A}(x, z)]=Y^\mathcal{A}(d_{\mathcal{W}'} x, z)$ for all $x\in \mathcal{W}'$. By the construction of $Y^\mathcal{A}$, this holds for $x=c^{\xi^*}_0\mathbf{1}$ with $\xi^*\in\mathfrak{g}^*$. Therefore by Lemma \ref{lem: sufficient condition for d-compatibility}, it holds for all $x\in\mathcal{W}'$.  Thus we proved the existence part. The latter half of our assertion follows from Proposition \ref{prop: sufficient condition for chiral W^*-modules}. 
\end{proof}

\begin{remark}
The condition \eqref{eq: curvature is horizontal in prop: construction of chiral W^*-modules} in Proposition \ref{prop: construction of chiral W^*-modules} can be replaced by the following condition:
\begin{equation}
[L^\mathcal{A}_\xi(z)_-, c^{\xi^*, \mathcal{A}}(w)]= c^{\mathrm{ad}^*\xi\cdot\xi^*}(w)\delta(z-w)_{-},
\end{equation}
for all $\xi\in\mathfrak{g}$ and $\xi^*\in\mathfrak{g}^*$.
\end{remark}

\section{VSA-Inductive Sheaves}\label{section: VSA-inductive Sheaves}
In this section, we introduce the VSA-inductive sheaves. In the next section, we will construct a vertex-algebraic analogue of the Lie algebroid complex as a VSA-inductive sheaf. 

\subsection{Ind-Objects}
Let $\mathcal{C}$ be a category. 
Recall the category $\mathrm{Ind} (\mathcal{C})$ of ind-objects of $C$ introduced by Grothendieck (\cite{AGV72}). 
An \textit{inductive system} of $\mathcal{C}$ is a functor 
$$
X: A \to \mathcal{C}, \quad \mathrm{Ob}(A)\ni \alpha \mapsto X(\alpha)=X_\alpha \in \mathrm{Ob}(\mathcal{C}),
$$
from a small filtered category $A$ to $\mathcal{C}$.
An inductive system $X: A \to \mathcal{C}$ is also written as $(X_\alpha)_{\alpha\in A}$. 
An \textit{ind-object} associated to an inductive system $(X_\alpha)_{\alpha\in A}$ is a symbol $``\displaystyle\varinjlim_{\alpha\in A}" X_\alpha$.  The objects of the category $\mathrm{Ind}(\mathcal{C})$ are the ind-objects of $\mathcal{C}$.
We often express $``\displaystyle\varinjlim_{\alpha\in A}" X_\alpha$ by the corresponding functor $X$ like $X=``\displaystyle\varinjlim_{\alpha\in A}" X_\alpha$.
The morphisms of $\mathrm{Ind}(\mathcal{C})$ are defined by 
$$
\mathrm{Hom}_{\mathrm{Ind}({\mathcal{C}})}(``\varinjlim_{\alpha\in A}" X_\alpha, ``\varinjlim_{\beta\in B}" Y_\beta):=\varprojlim_{\alpha\in A}\varinjlim_{\beta\in B}\mathrm{Hom}_{\mathcal{C}}(X_\alpha, Y_\beta),
$$
where the limits in the right-hand side stand for those in the category of sets. 
For a morphisms of ind-objects $F: (X_\alpha)_{\alpha\in A}\to (Y_\beta)_{\beta\in B}$, $F$ is written as 
$
F=\bigl([F_\alpha^{j(\alpha)}]\bigr)_{\alpha\in A},
$
where $j: \mathrm
{Ob}(A) \to \mathrm{Ob}(B)$ is a map and $[F_\alpha^{j(\alpha)}]$ is an equivalence class of a morphism $F_\alpha^{j(\alpha)}: X_\alpha \to Y_{j(\alpha)}$ in $\displaystyle\varinjlim_{\beta\in B} \mathrm{Hom}_{\mathcal{C}}(X_\alpha, Y_\beta)$. The composition is defined by 
$
F\circ G := \bigl([F_{j_G(\alpha)}^{j_F(j_G(\alpha))}\circ G_\alpha^{j_G(\alpha)}]\bigr)_{\alpha\in A}
$
for morphisms of ind-objects $F=\bigl([F_\beta^{j_F(\beta)}]\bigr)_{\beta\in B}: (Y_\beta)_{\beta \in B}\to (Z_\gamma)_{\gamma\in \Gamma}$ and $G=\bigl([G_\alpha^{j_G(\alpha)}]\bigr)_{\alpha\in A}: (X_\alpha)_{\alpha\in A}\to (Y_\beta)_{\beta\in B}$. 
For a small filtered category $A$, we will use the notation $f_{\alpha' \alpha}$ to express a morphism $f\in \mathrm{Hom}_A(\alpha, \alpha')$, emphasizing the source and the target.

Let $\mathit{Presh}_X(\mathit{Vec}^{\mathrm{super}}_{\mathbb{K}})$ be the category of presheaves on a topological space $X$ of super vector spaces over  $\mathbb{K}$. Consider the category of ind-objects  of the category $\mathit{Presh}_X(\mathit{Vec}^{\mathrm{super}}_{\mathbb{K}})$, $\mathrm{Ind}(\mathit{Presh}_X(\mathit{Vec}^{\mathrm{super}}_{\mathbb{K}}))$. 
There exists a functor 
\begin{equation}\label{eq: functor IndPresh to Presh}
\underrightarrow{\mathrm{Lim}}\, : \mathrm{Ind}(\mathit{Presh}_X(\mathit{Vec}^{\mathrm{super}}_\mathbb{K}))\to \mathit{Presh}_X(\mathit{Vec}^{\mathrm{super}}_{\mathbb{K}}),
\end{equation}
sending an object $\mathcal{F}=``\displaystyle\varinjlim_{\alpha\in A}"\mathcal{F}_\alpha$ to the inductive limit presheaf $\underrightarrow{\mathrm{Lim}}\, \mathcal{F}:=\varinjlim_{\alpha\in A}\mathcal{F}_\alpha$, but not its sheafification, and sending a morphism $F=( [F_\alpha^{j(\alpha)}])_{\alpha\in A}:  ``\displaystyle\varinjlim_{\alpha\in A}"\mathcal{F}_\alpha \to ``\displaystyle\varinjlim_{\beta\in B}"\mathcal{G}_\beta$ to the morphism of presheaves $\underrightarrow{\mathrm{Lim}}\,  F$ defined by  
\begin{equation}\label{eq: DEFINITION OF injlimF}
\underrightarrow{\mathrm{Lim}}\,  F(U): \varinjlim_{\alpha\in A}\mathcal{F}_\alpha(U) \to \varinjlim_{\beta \in B}\mathcal{G}_\beta(U), \quad [x_{\alpha}] \mapsto [F_{\alpha}^{j(\alpha)}x_{\alpha}],
\end{equation}
for each open subset $U \subset X$.
Set
\begin{multline}\label{df: bilinear morphisms of ind-objects}
\mathrm{Bilin}_{\mathrm{Ind}(\mathit{Presh}_X(\mathit{Vec}^{\mathrm{super}}_\mathbb{K}))}\bigl(``\varinjlim_{\alpha\in A}"\mathcal{F}_\alpha,  ``\varinjlim_{\beta\in B}"\mathcal{G}_\beta; ``\varinjlim_{\gamma\in\Gamma}"\mathcal{H}_\gamma\bigr) \\ 
:=\varprojlim_{(\alpha, \beta)\in A\times B} \Biggl(\varinjlim_{\gamma\in\Gamma}\mathrm{Bilin}_{\mathit{Presh}_X(\mathit{Vec}^{\mathrm{super}}_{\mathbb{K}})}(\mathcal{F}_\alpha, \mathcal{G}_\beta; \mathcal{H}_\gamma)\Biggr),
\end{multline}
for  ind-objects $``\displaystyle\varinjlim_{\alpha\in A}"\mathcal{F}_\alpha,  ``\varinjlim_{\beta\in B}"\mathcal{G}_\beta,  ``\varinjlim_{\gamma\in\Gamma}"\mathcal{H}_\gamma$ of the category $\mathit{Presh}_X(\mathit{Vec}^{\mathrm{super}}_{\mathbb{K}})$,  where   $\mathrm{Bilin}_{\mathit{Presh}_X(\mathit{Vec}^{\mathrm{super}}_{\mathbb{K}})}(\mathcal{F}_\alpha, \mathcal{G}_\beta; \mathcal{H}_\gamma)$ is the set of all bilinear morphisms of presheaves from $\mathcal{F}_\alpha\times \mathcal{G}_\beta$ to $\mathcal{H}_\gamma$.
We call an element $F$ of the set \eqref{df: bilinear morphisms of ind-objects} a \textit{bilinear morphism} of ind-objects and write it as
$$
F: ``\varinjlim_{\alpha\in A}"\mathcal{F}_\alpha\times  ``\varinjlim_{\beta\in B}"\mathcal{G}_\beta\to ``\varinjlim_{\gamma\in\Gamma}"\mathcal{H}_\gamma.
$$ 
Each ind-object $``\displaystyle\varinjlim_{\gamma\in\Gamma}"\mathcal{H}_\gamma$ gives rise to a canonical contravariant functor  
\begin{align*}
&\mathrm{Ind}(\mathit{Presh}_X(\mathit{Vec}^{\mathrm{super}}_{\mathbb{K}})\bigr)\times \mathrm{Ind}(\mathit{Presh}_X(\mathit{Vec}^{\mathrm{super}}_{\mathbb{K}}))\to \mathit{Set}, \\ 
&\Bigl(``\varinjlim_{\alpha\in A}"\mathcal{F}_\alpha, ``\varinjlim_{\beta\in B}"\mathcal{G}_\beta \bigr)\mapsto \mathrm{Bilin}_{\mathrm{Ind}(C)}\bigl(``\varinjlim_{\alpha\in A}"\mathcal{F}_\alpha,  ``\varinjlim_{\beta\in B}"\mathcal{G}_\beta;  ``\varinjlim_{\gamma\in\Gamma}"\mathcal{H}_\gamma\bigr), 
\end{align*}
where $\mathit{Set}$ is the category of sets. Similarly, 
each pair of ind-objects $\Bigl(``\displaystyle\varinjlim_{\alpha\in A}"\mathcal{F}_\alpha,$ $ ``\displaystyle\varinjlim_{\beta\in B}"\mathcal{G}_\beta \bigr)$ induces a canonical covariant functor 
\begin{align*}
\mathrm{Ind}(\mathit{Presh}_X(\mathit{Vec}^{\mathrm{super}}_{\mathbb{K}})) &\to \mathit{Set}, \\
 ``\varinjlim_{\gamma\in\Gamma}"\mathcal{H}_\gamma &\mapsto \mathrm{Bilin}_{\mathrm{Ind}(C)}\bigl(``\varinjlim_{\alpha\in A}"\mathcal{F}_\alpha,  ``\varinjlim_{\beta\in B}"\mathcal{G}_\beta; ``\varinjlim_{\gamma\in\Gamma}"\mathcal{H}_\gamma\bigr).
\end{align*}
Moreover,  
a bilinear morphism of ind-objects 
$$
F=\bigl([F_{(\alpha, \gamma)}^{j(\alpha, \gamma)}]\bigr)_{(\alpha, \gamma)\in A\times B}: ``\varinjlim_{\alpha\in A}"\mathcal{F}_\alpha\times  ``\varinjlim_{\beta\in B}"\mathcal{G}_\beta\to ``\varinjlim_{\gamma\in\Gamma}"\mathcal{H}_\gamma,
$$
induces a bilinear morphism of presheaves 
$$
\underrightarrow{\mathrm{Lim}}\,  F: \varinjlim_{\alpha\in A}\mathcal{F}_\alpha\times  \varinjlim_{\beta\in B}\mathcal{G}_\beta\to \varinjlim_{\gamma\in\Gamma}\mathcal{H}_\gamma,
$$
in the same way as in \eqref{eq: DEFINITION OF injlimF}. 

Let $\mathit{Sh}_X(\mathit{Vec}^{\mathrm{super}}_{\mathbb{K}})$ be the full subcategory of $\mathit{Presh}_X(\mathit{Vec}^{\mathrm{super}}_{\mathbb{K}})$ consisting of sheaves on $X$ of super vector spaces over $\mathbb{K}$. Then the corresponding category of ind-objects  $\mathrm{Ind}(\mathit{Sh}_X(\mathit{Vec}^{\mathrm{super}}_{\mathbb{K}}))$ is a full subcategory of $\mathrm{Ind}(\mathit{Presh}_X(\mathit{Vec}^{\mathrm{super}}_{\mathbb{K}}))$.
Note that the category $\mathrm{Ind}(\mathit{Sh}_X(\mathit{Vec}^{\mathrm{super}}_{\mathbb{K}}))$ is bigger than the category of \textit{ind-sheaves} introduced by Kashiwara and Schapira (\cite{KS99}). The latter is the category of ind-objects of the category of sheaves with  compact supports.

\subsection{Definition of VSA-Inductive Sheaves}

We denote by $\mathbb{K}_X$ the presheaf on a topological space $X$ of constant $\mathbb{K}$-valued functions. 
We regard $\mathbb{K}_X$ as an inductive system indexed by a set with one element and denote by $``\displaystyle\varinjlim"\mathbb{K}_X$ the corresponding ind-object.

\begin{definition}
A \textbf{vertex superalgebra inductive sheaf (VSA-inductive sheaf)} on a topological space $X$ is a quadruple $\bigl( \mathcal{F}, \underline
{\mathbf{1}}, \underline{T}, \underline{(n)}; n\in \mathbb{Z}\bigr)$ consisting of 
\begin{enumerate}[$\bullet$]
     \setlength{\topsep}{1pt}
     \setlength{\partopsep}{0pt}
     \setlength{\itemsep}{1pt}
     \setlength{\parsep}{0pt}
     \setlength{\leftmargin}{20pt}
     \setlength{\rightmargin}{0pt}
     \setlength{\listparindent}{0pt}
     \setlength{\labelsep}{3pt}
     \setlength{\labelwidth}{30pt}
     \setlength{\itemindent}{5pt}
\item an ind-object of  $\mathit{Sh}_X(\mathit{Vec}^{\mathrm{super}}_{\mathbb{K}})$, 
$
\mathcal{F}=``\displaystyle\varinjlim_{\alpha\in A}"\mathcal{F}_\alpha,
$
\item an even morphism of ind-objects $\underline{\mathbf{1}}: ``\varinjlim"\mathbb{K}_X\to \mathcal{F}$, 
\item an even morphism of ind-objects $\underline{T}: \mathcal{F}\to\mathcal{F}$,
\item an even bilinear morphisms of ind-objects $\underline{(n)}: \mathcal{F}\times \mathcal{F} \to \mathcal{F}$,
\end{enumerate}
such that 
the map $\mathcal{F}(f)$ is even and injective for any morphism $f$ in $A$ and the quadruples  
$
\bigl(\varinjlim_{\alpha\in A}\mathcal{F}_\alpha(U), \mathbf{1}, T, (n); n\in\mathbb{Z}\bigr)
$
are vertex superalgebras for all open subsets $U\subset X$, 
where $\mathbf{1}=\mathbf{1}(U):=\bigl(\displaystyle\underrightarrow{\mathrm{Lim}}\, {\underline{\mathbf{1}}}(U)\bigr)(1)$, $T=T(U):=\displaystyle\underrightarrow{\mathrm{Lim}}\, \underline{T}(U)$, $(n)=(n)(U):=\displaystyle\underrightarrow{\mathrm{Lim}}\, \underline{(n)}(U)$.
\end{definition}

Let $\mathcal{V}_1=\bigl( \mathcal{F}_1, \underline
{\mathbf{1}}_1, \underline{T}_1, \underline{(n)}_1; n\in \mathbb{Z}\bigr)$ and $\mathcal{V}_2=\bigl( \mathcal{F}_2, \underline
{\mathbf{1}}_2, \underline{T}_2, \underline{(n)}_2; n\in \mathbb{Z}\bigr)$ be  VSA-inductive sheaves on the same topological space $X$.
We call a morphism of ind-objects $\Phi=([\Phi_\alpha^{j(\alpha)}])_{\alpha\in \mathrm{Ob}(A)}: \mathcal{F}_1\to \mathcal{F}_2$ a \textbf{base-preserving morphism} of VSA-inductive sheaves from $\mathcal{V}_1$ to $\mathcal{V}_2$ if $\Phi$ satisfies $\Phi\circ\underline{\mathbf{1}}_1=\underline{\mathbf{1}}_2$, $\Phi\circ\underline{T}_1=\underline{T}_2\circ\Phi$ and $\Phi\circ\underline{(n)}_1=\underline{(n)}_2\circ(\Phi\times\Phi)$ for all $n\in\mathbb{Z}$, where $\Phi\times\Phi$ is the morphism of ind-objects given by $\Phi\times\Phi:=\bigl([\Phi_\alpha^{j(\alpha)}\times\Phi_{\alpha'}^{j(\alpha')}]\bigr)_{(\alpha, \alpha')\in\mathrm{Ob}(A)\times\mathrm{Ob}(A)}$.

\begin{remark}\label{rem: VSA-inductive sheaves form a category}
Let $X$ be a topological space. VSA-inductive sheaves on $X$ form a category with base-preserving morphisms of VSA-inductive sheaves.  This category is a subcategory of $\mathrm{Ind}(\mathit{Sh}_{X}(\mathit{Vec}^{\mathrm{super}}_{\mathbb{K}}))$ and hence of $\mathrm{Ind}(\mathit{Presh}_{X}(\mathit{Vec}^{\mathrm{super}}_{\mathbb{K}}))$.
\end{remark}

\begin{notation}
Denote by $\mathit{VSA_{\mathbb{K}}}\textit{-IndSh}_X$ the category of VSA-inductive sheaves on a topological space $X$ obtained in Remark \ref{rem: VSA-inductive sheaves form a category}.
\end{notation}

\begin{lemma}\label{lem: PRESHEAF associated with A VSA-inductive sheaf}
Let $\mathcal{V}=\bigl( \mathcal{F}, \underline
{\mathbf{1}}, \underline{T}, \underline{(n)}; n\in \mathbb{Z}\bigr)$ be a VSA-inductive sheaf on a topological space $X$. Then the assignment 
$$
U\to \bigl(\underrightarrow{\mathrm{Lim}}\, \mathcal{F}(U), \mathbf{1}, T, (n); n\in\mathbb{Z}\bigr),
$$
with restriction maps of the presheaf $\underrightarrow{\mathrm{Lim}}\, \mathcal{F}$ defines a presheaf on $X$ of vertex superalgebras.  
\end{lemma}
\begin{proof}
We must check the restriction maps are vertex superalgebra morphisms. We can see this since  $\displaystyle\underrightarrow{\mathrm{Lim}}\, {\underline{\mathbf{1}}}$, $\displaystyle\underrightarrow{\mathrm{Lim}}\, \underline{T}$, $\displaystyle\underrightarrow{\mathrm{Lim}}\, \underline{(n)}$ are morphisms of presheaves.
\end{proof}

Let $\mathcal{V}$ and $\mathcal{V}'$ be VSA-inductive sheaves.
We write a morphism $\Phi$ as $\Phi: \mathcal{V}\to\mathcal{V}'$ even when $\Phi$ is not a morphism of VSA-inductive sheaves but simply a morphism of the underlying ind-objects of sheaves.  When we say that $\Phi: \mathcal{V}\to\mathcal{V}'$ is a morphism of ind-objects, we mean $\Phi$ is a morphism between the underlying ind-objects of sheaves.

\begin{lemma}
Let $\Phi: \mathcal{V}_1\to \mathcal{V}_2$ be a base-preserving morphism of VSA-inductive sheaves on the same topological space. Then the map $\underrightarrow{\mathrm{Lim}}\, \Phi: \underrightarrow{\mathrm{Lim}}\, {\mathcal{V}_1} \to \underrightarrow{\mathrm{Lim}}\, {\mathcal{V}_2}$ is a morphism of presheaves of vertex superalgebras.
\end{lemma}
\begin{proof}
This follows directly from the definition of the morphisms.
\end{proof}

\begin{remark}
When we restrict the functor given in \eqref{eq: functor IndPresh to Presh}
$$
\underrightarrow{\mathrm{Lim}}\, : \mathrm{Ind}(\mathit{Presh}_{X}(\mathit{Vec}^{\mathrm{super}}_{\mathbb{K}})) \to \mathit{Presh}_X(\mathit{Vec}_{\mathbb{K}}^{\mathrm{super}}), 
$$
to the subcategory $\mathit{VSA_{\mathbb{K}}}\textit{-IndSh}_X$, we have a functor
\begin{equation}\label{eq: functor from VSA-ISh}
\mathit{VSA_{\mathbb{K}}}\textit{-IndSh}_X \to \mathit{Presh}_X(\mathit{VSA_{\mathbb{K}}}),
\end{equation}
where $\mathit{Presh}_X(\mathit{VSA_{\mathbb{K}}})$ is the category of presheaves on $X$ of vertex superalgebras over $\mathbb{K}$. 
\end{remark}
We also denote by $\underrightarrow{\mathrm{Lim}}\, $ the functor \eqref{eq: functor from VSA-ISh}. 

\begin{remark}\label{rem: LOCAL CALCULATION OF THE PRESHEAF associated with A Vsa IndSh}
Let $\mathcal{V}=\bigl( \mathcal{F}, \underline
{\mathbf{1}}, \underline{T}, \underline{(n)}; n\in \mathbb{Z}\bigr)$ be a VSA-inductive sheaf. Let $U\subset X$ be an open subset and $U=\bigcup_{\lambda\in\Lambda}U_{\lambda}$ an open covering. Then the map induced by restriction maps
$$
\underrightarrow{\mathrm{Lim}}\, {\mathcal{V}}(U)\to \prod_{\lambda \in\Lambda}\underrightarrow{\mathrm{Lim}}\, {\mathcal{V}}(U_\lambda), 
$$
is injective since $\mathcal{F}_\alpha$ are sheaves, where $\mathcal{F}= ``\displaystyle\varinjlim_{\alpha\in A}"\mathcal{F}_\alpha$,  and the map $\mathcal{F}(f)$ is injective for any morphism $f$ in $A$.   
\end{remark}

\begin{definition}\label{df: grading operator on a VSA-inductive sheaf}
Let $\mathcal{V}=\bigl( \mathcal{F}=``\displaystyle\varinjlim_{\alpha\in A}"\mathcal{F}_{\alpha}, \underline
{\mathbf{1}}, \underline{T}, \underline{(n)}; n\in \mathbb{Z}\bigr)$ be a VSA-inductive sheaf on a topological space $X$. 
\begin{enumerate}[(i)]
 \setlength{\topsep}{1pt}
     \setlength{\partopsep}{0pt}
     \setlength{\itemsep}{1pt}
     \setlength{\parsep}{0pt}
     \setlength{\leftmargin}{20pt}
     \setlength{\rightmargin}{0pt}
     \setlength{\listparindent}{0pt}
     \setlength{\labelsep}{3pt}
     \setlength{\labelwidth}{15pt}
     \setlength{\itemindent}{0pt}
     \renewcommand{\makelabel}{\upshape}
\item A \textbf{Hamiltonian} or a \textbf{weight-grading operator} on $\mathcal{V}$ is an even morphism $\underline{H}: \mathcal{F}\to\mathcal{F}$ of ind-objects such that there exists a family $(H^\alpha_\alpha: \mathcal{F}_\alpha\to\mathcal{F}_\alpha)_{\alpha\in \mathrm{Ob}(A)}$ of even morphisms of sheaves satisfying the following conditions: 
\begin{enumerate}[$(1)$]
     \setlength{\topsep}{1pt}
     \setlength{\partopsep}{0pt}
     \setlength{\itemsep}{1pt}
     \setlength{\parsep}{0pt}
     \setlength{\leftmargin}{35pt}
     \setlength{\rightmargin}{0pt}
     \setlength{\listparindent}{0pt}
     \setlength{\labelsep}{3pt}
     \setlength{\labelwidth}{15pt}
     \setlength{\itemindent}{0pt} 
\item $\underline{H}=([H_\alpha^\alpha])_{\alpha\in\mathrm{Ob}(A)}$.
\item Each $H_\alpha^\alpha$ is a diagonalizable operator on $\mathcal{F}_\alpha^\alpha$, namely, the operator  $H_\alpha^\alpha(U): \mathcal{F}_\alpha(U)\to\mathcal{F}_\alpha(U)$ is diagonalizable for any open subset $U\subset X$.
\item For all $n\in \mathbb{Z}$, 
\begin{equation}\label{eq: Hamiltonian}
\underline{(n)}\circ(\mathrm{id}\times\underline{H}+\underline{H}\times\mathrm{id})=(\underline{H}-(-n-1))\circ\underline{(n)}.
\end{equation}
\end{enumerate}
\item  A \textbf{degree-grading operator} on $\mathcal{V}$ is an even morphism $\underline{J}: \mathcal{F}\to\mathcal{F}$ of ind-objects such that there exists a family $(J^\alpha_\alpha: \mathcal{F}_\alpha\to\mathcal{F}_\alpha)_{\alpha\in \mathrm{Ob}(A)}$ of even morphisms of sheaves satisfying the following conditions: 
\begin{enumerate}[$(1)$]
     \setlength{\topsep}{1pt}
     \setlength{\partopsep}{0pt}
     \setlength{\itemsep}{1pt}
     \setlength{\parsep}{0pt}
     \setlength{\leftmargin}{35pt}
     \setlength{\rightmargin}{0pt}
     \setlength{\listparindent}{0pt}
     \setlength{\labelsep}{3pt}
     \setlength{\labelwidth}{15pt}
     \setlength{\itemindent}{0pt}
\item $\underline{J}=([J_\alpha^\alpha])_{\alpha\in\mathrm{Ob}(A)}$.
\item Each $J_\alpha^\alpha$ is a diagonalizable operator on $\mathcal{F}_\alpha^\alpha$. 
\item For all $n\in \mathbb{Z}$, 
\begin{equation}\label{eq: degree-grading operator}
\underline{(n)}\circ(\mathrm{id}\times\underline{J}+\underline{J}\times\mathrm{id})=\underline{J}\circ\underline{(n)}.
\end{equation}
\end{enumerate}
\end{enumerate}
\end{definition}

\begin{remark}
In the above definition, the family $(H^\alpha_\alpha)_{\alpha\in\mathrm{Ob}(A)}$ and $(J^\alpha_\alpha)_{\alpha\in\mathrm{Ob}(A)}$ are unique by the relation $(1)$ and the injectivity of the morphisms in the inductive system $\mathcal{F}$.  
\end{remark}

\begin{definition}
\begin{enumerate}[(i)]
 \setlength{\topsep}{1pt}
     \setlength{\partopsep}{0pt}
     \setlength{\itemsep}{1pt}
     \setlength{\parsep}{0pt}
     \setlength{\leftmargin}{24pt}
     \setlength{\rightmargin}{0pt}
     \setlength{\listparindent}{0pt}
     \setlength{\labelsep}{3pt}
     \setlength{\labelwidth}{15pt}
     \setlength{\itemindent}{0pt}
     \renewcommand{\makelabel}{\upshape}
\item A $\mathbb{Z}$-\textbf{graded VSA-inductive sheaf} is a VSA-inductive sheaf given a Hamiltonian with only integral eigenvalues. 
\item A \textbf{degree-graded VSA-inductive sheaf} is a VSA-inductive sheaf given a degree-grading operator with only integral eigenvalues. 
\item A VSA-inductive sheaf $\mathcal{V}=\bigl( \mathcal{F}=``\displaystyle\varinjlim_{\alpha\in A}"\mathcal{F}_{\alpha}, \underline
{\mathbf{1}}, \underline{T}, \underline{(n)}; n\in \mathbb{Z}\bigr)$ is said to be \textbf{degree-weight-graded} if $\mathcal{V}$ is given a Hamiltonian $\underline{H}=([H_\alpha^\alpha])_{\alpha\in\mathrm{Ob}(A)}$ and a degree-grading operator $\underline{J}=([J_\alpha^\alpha])_{\alpha\in\mathrm{Ob}(A)}$ 
such that for each $\alpha\in\mathrm{Ob}(A)$,  $\mathcal{F}_\alpha$ $=\bigoplus_{n, l\in\mathbb{Z}} \mathcal{F}_\alpha^l[n]$, where $\mathcal{F}_\alpha^l[n]:=\mathcal{F}_\alpha[n]\cap\mathcal{F}_\alpha^l$. Here $\mathcal{F}_\alpha[n]$ and $\mathcal{F}_\alpha^l$ are the subsheaves $\mathrm{Ker}(H_\alpha^\alpha-n)$ and $\mathrm{Ker}(J_\alpha^\alpha-l)$, respectively. 
\end{enumerate}
\end{definition}

\begin{lemma}
If $(\mathcal{V}, \underline{H})$ is a $\mathbb{Z}$-graded VSA-inductive sheaf, then $(\varinjlim\mathcal{V}, \underrightarrow{\mathrm{Lim}}\, \underline{H})$ is a presheaf of  $\mathbb{Z}$-graded vertex superalgebras. The same type of assertion holds in the  degree-graded case and in the degree-weight-graded case. 
\end{lemma}
\begin{proof}
We will prove the assertion in the weight-graded case. The others are proved similarly. Let $\mathcal{V}=\bigl( \mathcal{F}=``\displaystyle\varinjlim_{\alpha\in A}"\mathcal{F}_{\alpha}, \underline
{\mathbf{1}}, \underline{T}, \underline{(n)}; n\in \mathbb{Z}\bigr)$ be a $\mathbb{Z}$-graded VSA-inductive sheaf with a Hamiltonian $\underline{H}=([H_\alpha^\alpha])_{\alpha\in\mathrm{Ob}(A)}$. Let $f_{\alpha, \alpha'}: \alpha'\to\alpha$ be a morphism in $A$. Then the corresponding morphism $\mathcal{F}(f_{\alpha, \alpha'}): \mathcal{F}_{\alpha'}\to\mathcal{F}_\alpha$ preserves the grading. Indeed, from the injectivity of the morphisms in the inductive system $\mathcal{F}$ and the relation $H_\alpha^\alpha\circ\mathcal{F}(f_{\alpha, \alpha'})\sim H_{\alpha'}^{\alpha'}$ in $\varinjlim_{\beta\in A}\mathrm{Hom}(\mathcal{F}_{\alpha'}, \mathcal{F}_\beta)$, we have $H_\alpha^\alpha\circ\mathcal{F}(f_{\alpha, \alpha'})=\mathcal{F}(f_{\alpha, \alpha'})\circ H_{\alpha'}^{\alpha'}$. Thus we have $\varinjlim_{\alpha\in A}\mathcal{F}_\alpha$ $=\varinjlim_{\alpha\in A}(\bigoplus_{n\in\mathbb{Z}}\mathcal{F}_\alpha[n])$ $=\bigoplus_{n\in\mathbb{Z}}(\varinjlim_{\alpha\in A}\mathcal{F}_\alpha[n])$. Therefore $\underrightarrow{\mathrm{Lim}}\,  \underline{H}$ is a diagonalizable operator with only integral eigenvalues on the presheaf $\underrightarrow{\mathrm{Lim}}\,  \mathcal{F}=\varinjlim_{\alpha\in A}\mathcal{F}_\alpha$. From the relation \eqref{eq: Hamiltonian}, the operator $\underrightarrow{\mathrm{Lim}}\,  \underline{H}$ is a  Hamiltonian on $\underrightarrow{\mathrm{Lim}}\,  \mathcal{V}$. 
\end{proof}

\begin{notation}
Denote by $\textit{DegWt-}\mathit{VSA_{\mathbb{K}}}\textit{-IndSh}_X$ the category of degree-weight-graded VSA-inductive sheaves on a topological space $X$. Its morphisms are morphisms of VSA-inductive sheaves on $X$ commuting with the Hamiltonians and the degree-grading operators. 
\end{notation}

\subsection{Gluing Inductive Sheaves}
Let $X$ be a topological space and $A$ a small filtered category. 
We consider 
the  subcategory of $\textit{DegWt-}\mathit{VSA_{\mathbb{K}}}\textit{-IndSh}_X$ whose objects are degree-weight-graded VSA-inductive sheaves $\mathcal{V}=\bigl( \mathcal{F}, \underline{\mathbf{1}}, \underline{T}, \underline{(n)}; n\in \mathbb{Z}\bigr)$ such that $\mathcal{F}$ is a functor  from $A$ and whose morphisms are morphisms $\Phi: \mathcal{F} \to \mathcal{G}$ of degree-weight-graded VSA-inductive sheaves such that there exist a family of morphisms of sheaves $(\Phi_\alpha^\alpha: \mathcal{F}_\alpha \to \mathcal{G}_\alpha)_{\alpha\in A}$ satisfying $F=\bigl([\Phi_\alpha^\alpha]\bigr)_{\alpha\in A}$.
We denote this category by $\textit{DegWt-}\mathit{VSA_{\mathbb{K}}}\textit{-IndSh}_X^A$. We will often call a  morphism of this category a \textbf{strict morphism}. If two degree-weight-graded VSA-inductive sheaves are isomorphic via a strict isomorphism, by which we mean a isomorphism in $\textit{DegWt-}\mathit{VSA_{\mathbb{K}}}\textit{-IndSh}_X^A$, then we say they are \textbf{strictly isomorphic}.  
We also call a morphism of ind-objects $\Phi: \mathcal{F} \to \mathcal{G}$, not necessarily a morphism of VSA-inductive sheaves, \textit{strict} if the same condition above is satisfied.

\begin{remark}\label{rem: restriction of VSA-inductive sheaves}
Let $\mathcal{V}=\bigl(\mathcal{F},  \underline{\mathbf{1}}, \underline{T}, \underline{(n)}; n\in \mathbb{Z} \bigr)$ be a VSA-inductive sheaf on 
$X$. Let $U\subset X$ be an open subset. Consider the inductive system 
\begin{align*}
\mathcal{F}|_U: A &\to \mathit{Sh}_U(\mathit{Vec}^{\mathrm{super}}_{\mathbb{K}}), \\
\text{objects}:\quad \alpha &\mapsto \mathcal{F}_\alpha|_U, \\
\text{morphisms}:\quad f &\mapsto \mathcal{F}(f)|_U, 
\end{align*}
obtained by restricting $\mathcal{F}$ to $U$. The corresponding ind-object $``\displaystyle\varinjlim_{\alpha\in A}"(\mathcal{F}|_U)_\alpha$ is a VSA-inductive sheaf on $U$ with $\underline{\mathbf{1}}, \underline{T}, \underline{(n)}$ restricted to $U$. 
\end{remark}
For a VSA-inductive sheaf $\mathcal{V}$ on 
$X$ and an open subset $U$ of $X$, we denote by $\mathcal{V}|_U$ the VSA-inductive sheaf on $U$ given in Remark \ref{rem: restriction of VSA-inductive sheaves} and call it the \textbf{restriction} of the VSA-inductive sheaf $\mathcal{V}$. 

Let us glue VSA-inductive sheaves. 
Let $X=\bigcup_{\lambda \in \Lambda}U_\lambda$ be an open covering of $X$ and $(\mathcal{V}^\lambda)_{\lambda\in\Lambda}$  a family of degree-weight-graded VSA-inductive sheaves, where $\mathcal{V}^\lambda=\bigl(\mathcal{F}^\lambda, \underline{\mathbf{1}}^\lambda, \underline{T}^\lambda, \underline{(n)}^\lambda; n\in\mathbb{Z} \bigr)$ is an object of $\textit{DegWt-}\mathit{VSA_{\mathbb{K}}}\textit{-IndSh}_{U_\lambda}^A$. Let  $\underline{H}^\lambda$ and $\underline{J}^\lambda$ be the Hamiltonian and the degree-grading operator on $\mathcal{V}^\lambda$, respectively. Suppose given a family of strict isomorphisms of degree-weight-graded VSA-inductive sheaves $(\vartheta_{\lambda \mu}: \mathcal{V}^\mu|_{U_\mu \cap U_\lambda} \to \mathcal{V}^\lambda|_{U_\lambda \cap U_\mu})_{\lambda, \mu \in \Lambda}$ satisfying the following condition: 
\begin{align*}
(0)\quad \vartheta_{\lambda \lambda}=\mathrm{id},\quad \text{and}\quad (\vartheta_{\lambda \mu}|_{U_\lambda\cap U_\mu\cap U_\nu})\circ(\vartheta_{\mu \nu}|_{U_\mu\cap U_\nu\cap U_\lambda})=(\vartheta_{\lambda \nu}|_{U_\nu\cap U_\lambda\cap U_\mu}), \quad\\ \text{for all}\quad \lambda, \mu, \nu \in \Lambda.\quad\quad\quad\quad
\end{align*}
We will often omit the subscripts such as $|_{U_\lambda\cap U_\mu\cap U_\nu}$ in the sequel.

In addition to the condition $(0)$,  we assume the following conditions: 
\begin{enumerate}[$(1)$]
     \setlength{\topsep}{1pt}
     \setlength{\partopsep}{0pt}
     \setlength{\itemsep}{1pt}
     \setlength{\parsep}{0pt}
     \setlength{\leftmargin}{35pt}
     \setlength{\rightmargin}{0pt}
     \setlength{\listparindent}{0pt}
     \setlength{\labelsep}{3pt}
     \setlength{\labelwidth}{15pt}
     \setlength{\itemindent}{0pt}
\item For any $\alpha, \alpha' \in \mathrm{Ob}(A)$, we have $\mathrm{Hom}_A(\alpha, \alpha')\neq\emptyset$ or $\mathrm{Hom}_A(\alpha', \alpha)\neq\emptyset$,
\item There exist a $\alpha_0\in \mathrm{Ob}(A)$ and sheaf morphisms $\underline{\mathbf{1}}^{\lambda, \alpha_0}: \mathbb{K}_X\to \mathcal{F}^\lambda_{\alpha_0}$ with $\lambda\in\Lambda$ such that
$$
\underline{\mathbf{1}}^\lambda=[\underline{\mathbf{1}}^{\lambda, \alpha_0}],
$$
for all $\lambda\in\Lambda$. (Note that the element $\alpha_0$ can be taken independently of $\lambda\in\Lambda$.)
\item There exist a map $j_T: \mathrm{Ob}(A)\to \mathrm{Ob}(A)$ and sheaf morphisms $\underline{T}^{\lambda, j_T(\alpha)}_\alpha: \mathcal{F}^\lambda_\alpha \to \mathcal{F}^\lambda_{j_T(\alpha)}$ with $\alpha\in \mathrm{Ob}(A)$ and $\lambda \in \Lambda$ such that 
$$
\underline{T}^\lambda=\bigl([\underline{T}^{\lambda, j_T(\alpha)}_\alpha]\bigr)_{\alpha\in \mathrm{Ob}(A)},
$$
for all $\lambda \in \Lambda$. (Note that the map $j_T$ can be taken independently of $\lambda\in \Lambda$.)
\item For each $n\in \mathbb{Z}$, there exist a map $j_{(n)}: \mathrm{Ob}(A)\times\mathrm{Ob}(A)\to\mathrm{Ob}(A)$ and bilinear sheaf morphisms $\underline{(n)}_{(\alpha, \alpha')}^{\lambda,  j_{(n)}(\alpha, \alpha')}: \mathcal{F}^\lambda_\alpha\times \mathcal{F}^\lambda_{\alpha'}\to \mathcal{F}^\lambda_{j_{(n)}(\alpha, \alpha')}$ with $(\alpha, \alpha')\in \mathrm{Ob}(A)\times \mathrm{Ob}(A)$ and $\lambda\in\Lambda$ such that 
$$
\underline{(n)}^\lambda=\bigl([\underline{(n)}_{(\alpha, \alpha')}^{\lambda,  j_{(n)}(\alpha, \alpha')}]\bigr)_{(\alpha, \alpha')\in \mathrm{Ob}(A)\times\mathrm{Ob}(A)},
$$
for all $\lambda\in\Lambda$. (Note that the map $j_{(n)}$ can be taken independently of $\lambda \in \Lambda$.)
\item The degree and weight-grading on each  $\mathcal{F}_\alpha^\lambda$ are bounded from the above and the below  uniformly with respect to $\lambda$. Moreover the weight-grading on $\mathcal{F}_\alpha^\lambda$ is bounded from the below uniformly with respect to $\alpha$ as well as $\lambda$. In other words, there exist an integer $N$ and natural numbers $N_\alpha, L_\alpha$ with $\alpha\in\mathrm{Ob}(A)$ such that $\mathcal{F}_\alpha^\lambda=\bigoplus_{N\le n \le N_\alpha}\bigoplus_{|l|\le L_\alpha}(\mathcal{F}_\alpha^\lambda)^l[n]$, where $(\mathcal{F}_\alpha^\lambda)^l[n]$ is the subsheaf of degree $l$ and weight $n$. 
\end{enumerate}

The uniqueness in the following proposition means that if $\bigl(\mathcal{V}, (\Phi^\lambda: \mathcal{V}|_{U_\lambda}\to\mathcal{V}^\lambda)_{\lambda\in\Lambda}\bigr)$ and  $\bigl(\mathcal{V}', (\Phi'^\lambda: \mathcal{V}'|_{U_\lambda}\to\mathcal{V}^\lambda)_{\lambda\in\Lambda}\bigr)$ are the pairs as in the proposition below then there exists a strict isomorphism of degree-weight-graded VSA-inductive sheaves $F: \mathcal{V}\to \mathcal{V}'$ such that $\Phi'^\lambda\circ F|_{U_\lambda}=\Phi^\lambda$ for all $\lambda\in \Lambda$. 

\begin{proposition}\label{prop: GLUING VsaIndShs}
Under the above assumptions, 
there exists an object $\mathcal{V}$ of $\textit{DegWt-}\mathit{VSA_{\mathbb{K}}}\textit{-IndSh}_X^A$ and strict isomorphisms $(\Phi^\lambda: \mathcal{V}|_{U_\lambda}\to \mathcal{V}^\lambda)_{\lambda\in\Lambda}$ of degree-weight-graded VSA-inductive sheaves such that $(\Phi^\lambda|_{U_\lambda\cap U_\mu}) \circ (\Phi^\mu|_{U_\mu\cap U_\lambda})^{-1}=\vartheta_{\lambda \mu}$ for all $\lambda, \mu \in \Lambda$. Moreover such a pair is unique up to strict isomorphisms. 
\end{proposition}
\begin{proof}
First we see the existence part. Since $\vartheta_{\lambda \mu}$ is strict, we can write $\vartheta_{\lambda \mu}$ as
$
\vartheta_{\lambda \mu}=([\vartheta_{\lambda\mu, \alpha}^\alpha])_{\alpha\in\mathrm{Ob}(A)},
$
where $\vartheta_{\lambda\mu, \alpha}^\alpha$ is a sheaf morphism from $\mathcal{F}^\mu_\alpha|_{U_\mu\cap U_\lambda}$ to $\mathcal{F}^\lambda_\alpha|_{U_\lambda\cap U_\mu}$.
For each $\lambda\in \Lambda$ we have
$
\mathrm{id}=\vartheta_{\lambda\lambda}=([\vartheta_{\lambda\lambda, \alpha}^\alpha])_{\alpha\in\mathrm{Ob}(A)}
$
and therefore
$
\mathrm{id}\sim\vartheta_{\lambda\lambda, \alpha}^\alpha
$
for all $\alpha\in\mathrm{Ob}(A)$, where $\sim$ means that the two sheaf morphisms are equivalent in $\displaystyle\varinjlim_{\alpha'}\mathrm{Hom}(\mathcal{F}^\lambda_{\alpha}, \mathcal{F}^\lambda_{\alpha'})$. 
By the injectivity of the morphisms of the inductive system $\mathcal{F}^\lambda$, 
we have 
\begin{equation}\label{eq: id=theta_lambda-lambda}
\mathrm{id}=\vartheta_{\lambda\lambda, \alpha}^\alpha.
\end{equation}
Similarly we have 
$
\vartheta_{\lambda\mu, \alpha}^\alpha\circ\vartheta_{\mu\nu, \alpha}^\alpha=\vartheta_{\lambda\nu, \alpha}^\alpha
$
for all $\alpha\in\mathrm{Ob}(A)$ and $\lambda, \mu, \nu\in\Lambda$ by the assumption. 
Therefore for each $\alpha\in \mathrm{Ob}(A)$, we can glue the sheaves $(\mathcal{F}^\lambda_\alpha)_{\lambda\in\Lambda}$ with the  sheaf morphisms $(\vartheta_{\lambda\mu, \alpha}^\alpha)_{\lambda, \mu\in\Lambda}$. Denote the resulting sheaf on $X$ by $\mathcal{F}_\alpha$. 
For each $f_{\alpha\alpha'}: \alpha'\to \alpha$, we will glue the morphisms $(\mathcal{F}^\lambda(f_{\alpha \alpha'}))_{\lambda\in\Lambda}$ to obtain a sheaf morphism $\mathcal{F}_{\alpha'}\to\mathcal{F}_\alpha$. By the definition of morphisms of ind-objects, we have 
$$
\vartheta_{\lambda\mu, \alpha}^\alpha\circ\mathcal{F}^\mu(f_{\alpha \alpha'})\sim\vartheta_{\lambda\mu, \alpha'}^{\alpha'},
$$
where $\sim$ means that the two sheaf morphisms are equivalent in the inductive limit  $\displaystyle\varinjlim_{\alpha\in A}\mathrm{Hom}(\mathcal{F}^\mu_{\alpha'}|_{U_\mu\cap U_\lambda}, \mathcal{F}^\lambda_\alpha|_{U_\lambda\cap U_\mu})$. 
Therefore 
we have $\vartheta_{\lambda\mu, \alpha}^\alpha\circ\mathcal{F}^\mu(f_{\alpha \alpha'})=\mathcal{F}^\lambda(f_{\alpha\alpha'})\circ\vartheta_{\lambda\mu, \alpha'}^{\alpha'}$ by the injectivity of the morphisms in the inductive systems. 
Thus we can glue the morphisms $(\mathcal{F}^\lambda(f_{\alpha \alpha'}))_{\lambda\in\Lambda}$ to obtain the sheaf morphism $\mathcal{F}_{\alpha'}\to\mathcal{F}_\alpha$, which we denote by $\mathcal{F}(f_{\alpha \alpha'})$. Note that $\mathcal{F}(f_{\alpha \alpha'})$ is injective since $\mathcal{F}^\lambda(f_{\alpha \alpha'})$ are injective for all $\lambda\in\Lambda$. 
By the construction, the assignment 
\begin{align*}
\mathcal{F}&: A\to \mathit{Sh}_X(Vec_\mathbb{K}^\mathrm{super}), \\
\text{objects}&:  \alpha\mapsto \mathcal{F}_\alpha, \\
\text{morphisms}&:  f_{\alpha \alpha'}\mapsto \mathcal{F}(f_{\alpha \alpha'}),
\end{align*}
defines a functor. Thus we have an inductive system $\mathcal{F}$ in  $\mathit{Sh}_X(Vec_\mathbb{K}^{\mathrm{super}})$, therefore the corresponding ind-object $``\displaystyle\varinjlim_{\alpha\in A}"\mathcal{F}_\alpha$. 
By the assumption, for each $n\in \mathbb{Z}$, we have a map $j_{(n)}: \mathrm{Ob}(A)\times\mathrm{Ob}(A)\to\mathrm{Ob}(A)$ and bilinear sheaf morphisms $\underline{(n)}_{(\alpha, \alpha')}^{\lambda,  j_{(n)}(\alpha, \alpha')}: \mathcal{F}^\lambda_\alpha\times \mathcal{F}^\lambda_{\alpha'}\to \mathcal{F}^\lambda_{j_{(n)}(\alpha, \alpha')}$ with $(\alpha, \alpha')\in \mathrm{Ob}(A)\times \mathrm{Ob}(A)$ and $\lambda\in\Lambda$ such that 
$$
\underline{(n)}^\lambda=\bigl([\underline{(n)}_{(\alpha, \alpha')}^{\lambda,  j_{(n)}(\alpha, \alpha')}]\bigr)_{(\alpha, \alpha')\in \mathrm{Ob}(A)\times\mathrm{Ob}(A)},
$$
for all $\lambda\in\Lambda$. 
Fix $n\in\mathbb{Z}$. For each $(\alpha, \alpha')\in\mathrm{Ob}(A)\times\mathrm{Ob}(A)$, we will glue morphisms $\bigl(\underline{(n)}_{(\alpha, \alpha')}^{\lambda,  j_{(n)}(\alpha, \alpha')}\bigr)_{\lambda\in\Lambda}$ to get a bilinear sheaf morphism $\mathcal{F}_\alpha\times\mathcal{F}_{\alpha'}\to\mathcal{F}_{j_{(n)}(\alpha, \alpha')}$. It suffices to check that the morphisms $\bigl(\underline{(n)}_{(\alpha, \alpha')}^{\lambda,  j_{(n)}(\alpha, \alpha')}\bigr)_{\lambda\in\Lambda}$ commute with the gluing maps. Since each $\vartheta_{\lambda\mu}$ is a morphism of  VSA-inductive sheaves, we have 
$$
\underline{(n)}^\lambda \circ (\vartheta_{\lambda\mu}\times\vartheta_{\lambda\mu})=\vartheta_{\lambda\mu}\circ\underline{(n)}^\mu,
$$
and therefore
$$
\underline{(n)}^{\lambda, j_{(n)}(\alpha, \alpha')}_{(\alpha, \alpha')}\circ(\vartheta_{\lambda\mu, \alpha}^\alpha\times \vartheta_{\lambda\mu, \alpha'}^{\alpha'})\sim \vartheta_{\lambda\mu, j_{(n)}(\alpha, \alpha')}^{j_{(n)}(\alpha, \alpha')}\circ\underline{(n)}^{\mu, j_{(n)}(\alpha, \alpha')}_{(\alpha, \alpha')},
$$
for all $(\alpha, \alpha')\in \mathrm{Ob}(A)\times\mathrm{Ob}(A)$. 
By the same argument for proving \eqref{eq: id=theta_lambda-lambda}, we have 
$$
\underline{(n)}^{\lambda, j_{(n)}(\alpha, \alpha')}_{(\alpha, \alpha')}\circ(\vartheta_{\lambda\mu, \alpha}^\alpha\times \vartheta_{\lambda\mu, \alpha'}^{\alpha'})= \vartheta_{\lambda\mu, j_{(n)}(\alpha, \alpha')}^{j_{(n)}(\alpha, \alpha')}\circ\underline{(n)}^{\mu, j_{(n)}(\alpha, \alpha')}_{(\alpha, \alpha')},
$$
for all $(\alpha, \alpha')\in \mathrm{Ob}(A)\times\mathrm{Ob}(A)$. 
Thus we can glue the morphisms $\underline{(n)}_{(\alpha, \alpha')}^{\lambda,  j_{(n)}(\alpha, \alpha')}$ with ${\lambda\in\Lambda}$. We denote by $\underline{(n)}_{(\alpha, \alpha')}^{j_{(n)}(\alpha, \alpha')}$ the resulting bilinear morphism of sheaves. 
We claim that  $\underline{(n)}:=\bigl([\underline{(n)}_{(\alpha, \alpha')}^{j_{(n)}(\alpha, \alpha')}]\bigr)_{(\alpha, \alpha')\in \mathrm{Ob}(A)\times\mathrm{Ob}(A)}$ is a bilinear morphism of ind-objects. We must check  
\begin{equation}\label{eq: (n) f*f sim (n)}
\underline{(n)}_{(\alpha, \alpha')}^{j_{(n)}(\alpha, \alpha')}\circ(\mathcal{F}(f_{\alpha \Tilde{\alpha}})\times \mathcal{F}(f_{\alpha' \Tilde{\alpha}'})) \sim \underline{(n)}_{(\Tilde{\alpha}, \Tilde{\alpha}')}^{j_{(n)}(\Tilde{\alpha}, \Tilde{\alpha}')},
\end{equation}
for any object $(\Tilde{\alpha}, \Tilde{\alpha}')\in \mathrm{Ob}(A)\times\mathrm{Ob}(A)$ and morphism $f_{\alpha \Tilde{\alpha}}\times f_{\alpha' \Tilde{\alpha}'}$. 
Let $(\Tilde{\alpha}, \Tilde{\alpha}')\in \mathrm{Ob}(A)\times\mathrm{Ob}(A)$ be an arbitrary object  and $f_{\alpha \Tilde{\alpha}}\times f_{\alpha' \Tilde{\alpha}'}$  an morphism. For each $\lambda\in\Lambda$, we have 
$$
\underline{(n)}_{(\alpha, \alpha')}^{\lambda, j_{(n)}(\alpha, \alpha')}\circ(\mathcal{F}^\lambda(f_{\alpha \Tilde{\alpha}})\times \mathcal{F}^\lambda(f_{\alpha' \Tilde{\alpha}'})) \sim \underline{(n)}_{(\Tilde{\alpha}, \Tilde{\alpha}')}^{\lambda, j_{(n)}(\Tilde{\alpha}, \Tilde{\alpha}')}.
$$
Therefore for each $\lambda\in\Lambda$, we have an object $\alpha''(\lambda)\in\mathrm{Ob}(A)$ and morphisms $f_{\alpha''(\lambda)\, j_{(n)}(\alpha, \alpha')}: j_{(n)}(\alpha, \alpha') \to \alpha''(\lambda)$, $f_{\alpha''(\lambda)\, j_{(n)}(\Tilde{\alpha}, \Tilde{\alpha}')}: j_{(n)}(\Tilde{\alpha}, \Tilde{\alpha}')  \to \alpha''(\lambda)$ in $A$ such that 
\begin{multline}\label{eq: f (n) f*f = f (n) in lambda}
\mathcal{F}^\lambda(f_{\alpha''(\lambda)\, j_{(n)}(\alpha, \alpha')})\circ\underline{(n)}_{(\alpha, \alpha')}^{\lambda, j_{(n)}(\alpha, \alpha')}\circ(\mathcal{F}^\lambda(f_{\alpha \Tilde{\alpha}})\times \mathcal{F}^\lambda(f_{\alpha' \Tilde{\alpha}'}))\\
= \mathcal{F}^\lambda(f_{\alpha''(\lambda)\, j_{(n)}(\Tilde{\alpha}, \Tilde{\alpha}')})\circ\underline{(n)}_{(\Tilde{\alpha}, \Tilde{\alpha}')}^{\lambda, j_{(n)}(\Tilde{\alpha}, \Tilde{\alpha}')}.
\end{multline}
By the assumption of this proposition, we have
\begin{equation*}
\mathrm{Hom}_A(j_{(n)}(\alpha, \alpha'), j_{(n)}(\Tilde{\alpha}, \Tilde{\alpha}'))\neq\emptyset\quad \text{or} \quad \mathrm{Hom}_A(j_{(n)}(\Tilde{\alpha}, \Tilde{\alpha}'), j_{(n)}(\alpha, \alpha'))\neq\emptyset.
\end{equation*}
When $\mathrm{Hom}_A(j_{(n)}(\alpha, \alpha'), j_{(n)}(\Tilde{\alpha}, \Tilde{\alpha}'))\neq\emptyset$, we have a morphism $f_{j_{(n)}(\alpha, \alpha')\, j_{(n)}(\Tilde{\alpha}, \Tilde{\alpha}')}$ in this set. The right-hand side of \eqref{eq: f (n) f*f = f (n) in lambda} is equivalent to  
$$
\mathcal{F}^\lambda(f_{\alpha''(\lambda)\, j_{(n)}(\alpha, \alpha')})\circ\mathcal{F}^\lambda(f_{j_{(n)}(\alpha, \alpha')\, j_{(n)}(\Tilde{\alpha}, \Tilde{\alpha}')})\circ\underline{(n)}_{(\Tilde{\alpha}, \Tilde{\alpha}')}^{\lambda, j_{(n)}(\Tilde{\alpha}, \Tilde{\alpha}')},
$$
in  $\varinjlim_{\beta\in A}\mathrm{Bilin}_{\mathit{Presh}_{U_\lambda}(\mathit{Vec}^{\mathrm{super}}_{\mathbb{K}})}(\mathcal{F}_{\Tilde{\alpha}}^\lambda, \mathcal{F}_{\Tilde{\alpha}'}^\lambda; \mathcal{F}_\beta^\lambda)$.  
By the injectivity of the morphism of the inductive system $\mathcal{F}^\lambda$, we have 
$$
\underline{(n)}_{(\alpha, \alpha')}^{\lambda, j_{(n)}(\alpha, \alpha')}\circ(\mathcal{F}^\lambda(f_{\alpha \Tilde{\alpha}})\times \mathcal{F}^\lambda(f_{\alpha' \Tilde{\alpha}'}))
=\mathcal{F}^\lambda(f_{j_{(n)}(\alpha, \alpha')\, j_{(n)}(\Tilde{\alpha}, \Tilde{\alpha}')})\circ\underline{(n)}_{(\Tilde{\alpha}, \Tilde{\alpha}')}^{\lambda, j_{(n)}(\Tilde{\alpha}, \Tilde{\alpha}')}.
$$
We have this equality for any $\lambda\in\Lambda$. Note that $f_{j_{(n)}(\alpha, \alpha')\, j_{(n)}(\Tilde{\alpha}, \Tilde{\alpha}')}$ does  not depend on $\lambda\in\Lambda$. Therefore we have the following relation for glued morphisms: 
$$
\underline{(n)}_{(\alpha, \alpha')}^{j_{(n)}(\alpha, \alpha')}\circ(\mathcal{F}(f_{\alpha \Tilde{\alpha}})\times \mathcal{F}(f_{\alpha' \Tilde{\alpha}'}))
=\mathcal{F}(f_{j_{(n)}(\alpha, \alpha')\, j_{(n)}(\Tilde{\alpha}, \Tilde{\alpha}')})\circ\underline{(n)}_{(\Tilde{\alpha}, \Tilde{\alpha}')}^{j_{(n)}(\Tilde{\alpha}, \Tilde{\alpha}')}.
$$
This means \eqref{eq: (n) f*f sim (n)}.
When $\mathrm{Hom}_A(j_{(n)}(\Tilde{\alpha}, \Tilde{\alpha}'), j_{(n)}(\alpha, \alpha'))\neq\emptyset$, we can also obtain \eqref{eq: (n) f*f sim (n)} in a similar way. Thus we have a bilinear morphism of ind-objects $\underline{(n)}=\bigl([\underline{(n)}_{(\alpha, \alpha')}^{j_{(n)}(\alpha, \alpha')}]\bigr)_{(\alpha, \alpha')\in \mathrm{Ob}(A)\times\mathrm{Ob}(A)}: \mathcal{F}\times\mathcal{F}\to\mathcal{F}$. 

In a similar way, we obtain morphisms of ind-objects 
$
\underline{T}=\bigl([\underline{T}_\alpha^{j_T(\alpha)}]\bigr)_{\alpha\in\mathrm{Ob}(A)}: \mathcal{F}\to\mathcal{F},
$ 
$\underline{\mathbf{1}}=[\underline{\mathbf{1}}^{\alpha_0}]: ``\displaystyle\varinjlim"\mathbb{K}_X\to\mathcal{F},
$
$\underline{H}=([H_\alpha^{\alpha}])_{\alpha\in\mathrm{Ob}(A)}: \mathcal{F}\to\mathcal{F}$ and $\underline{J}=([J_\alpha^{\alpha}])_{\alpha\in\mathrm{Ob}(A)}: \mathcal{F}\to\mathcal{F}$  
from the morphisms $\underline{T}^\lambda=\bigl([\underline{T}_\alpha^{\lambda, j_T(\alpha)}]\bigr)_{\alpha\in\mathrm{Ob}(A)}$,  $\underline{\mathbf{1}}^\lambda=[\underline{\mathbf{1}}^{\lambda, \alpha_0}]$, $\underline{H}^\lambda=([H_\alpha^{\lambda, \alpha}])_{\alpha\in\mathrm{Ob}(A)}$ and $\underline{J}^\lambda=([J_\alpha^{\lambda, \alpha}])_{\alpha\in\mathrm{Ob}(A)}$ with $\lambda\in\Lambda$, respectively. 
Since the gluing maps commutes with the Hamiltonians and the degree-grading operators, we can glue the sheaves $(\mathcal{F}_\alpha^\lambda)^l[n]$ with $\lambda\in\Lambda$. Denote the resulting sheaf on $X$ by $\mathcal{F}_\alpha^l[n]$. By the assumption (5), $(\mathcal{F}_\alpha)^l[n]=0$ for all but finitely many $l$ and $n$. Therefore the presheaf $\bigoplus_{n, l\in\mathbb{Z}}(\mathcal{F}_\alpha)^l[n]$ is a sheaf. Thus the sheaf $\mathcal{F}_\alpha$ is canonically isomorphic to the sheaf $\bigoplus_{n, l\in\mathbb{Z}}(\mathcal{F}_\alpha)^l[n]$ since they are locally isomorphic.  This grading comes from the operators $H_\alpha^\alpha$ and $J_\alpha^\alpha$. In other words, the operators $H_\alpha^\alpha$ and $J_\alpha^\alpha$ are diagonalizable. The relation \eqref{eq: Hamiltonian} for $\underline{H}$ and the relation \eqref{eq: degree-grading operator} for $\underline{J}$ follow from the fact that the corresponding relations hold locally.  
 
We must check the quadruple $\mathcal{V}:=(\mathcal{F}, \underline{\mathbf{1}}, \underline{T}, \underline{(n)}; n\in\mathbb{Z})$ is a VSA-inductive sheaf on $X$. 
Let $V$ be an arbitrary open subset of $X$. 
It suffices to show that the quadruple $(\varinjlim_{\alpha\in A}\mathcal{F}_\alpha(V), \mathbf{1}, T, (n); n\in\mathbb{Z})$ is a vertex  superalgebra, where $\mathbf{1}:=\underrightarrow{\mathrm{Lim}}\, \underline{\mathbf{1}}(V)(1)$, $T:=\underrightarrow{\mathrm{Lim}}\, \underline{T}(V)$ and $(n):=\underrightarrow{\mathrm{Lim}}\, \underline{(n)}(V)$. 
The map induced by the restriction maps and the isomorphisms $\mathcal{F}_\alpha|_{U_\lambda}\cong\mathcal{F}^\lambda_\alpha$,
\begin{equation}\label{eq: inclusion into the product VSA}
\varinjlim_{\alpha\in A}\mathcal{F}_\alpha(V)\to\prod_{\lambda\in\Lambda}\varinjlim_{\alpha\in A}\mathcal{F}_\alpha(V\cap U_\lambda)\cong\prod_{\lambda\in\Lambda}\varinjlim_{\alpha\in A}\mathcal{F}_\alpha^\lambda(V\cap U_\lambda),
\end{equation}
is injective since the  morphisms $\mathcal{F}^\lambda(f_{\alpha\alpha'})$ are all injective. Via this map, 
we regard $\varinjlim_{\alpha\in A}\mathcal{F}_\alpha(V)$ as a subspace of $\prod_{\lambda\in\Lambda}\varinjlim_{\alpha\in A}\mathcal{F}_\alpha^\lambda(V\cap U_\lambda)$. Then by the construction, $\varinjlim_{\alpha\in A}\mathcal{F}_\alpha(V)$ is preserved by $\prod_{\lambda\in\Lambda}T^\lambda$ and  $\prod_{\lambda\in\Lambda}(n)^\lambda$ with $n\in\mathbb{Z}$, where $T^\lambda$ and $(n)^\lambda$ are the translation operator and the $n$-th product of $\varinjlim_{\alpha\in A}\mathcal{F}_\alpha^\lambda(V\cap U_\lambda)$, respectively. Moreover $(\mathbf{1}^\lambda)_{\lambda\in\Lambda}\in \varinjlim_{\alpha\in A}\mathcal{F}_\alpha(V)$, where $\mathbf{1}^\lambda$ is the vacuum vector of $\varinjlim_{\alpha\in A}\mathcal{F}_\alpha^\lambda(V\cap U_\lambda)$. Note that  $T=(\prod_{\lambda\in\Lambda}T^\lambda)|_{\varinjlim_{\alpha\in A}\mathcal{F}_\alpha(V)}$, $(n)=(\prod_{\lambda\in\Lambda}(n)^\lambda)|_{\varinjlim_{\alpha\in A}\mathcal{F}_\alpha(V)}$ and $\mathbf{1}=(\mathbf{1}^\lambda)_{\lambda\in\Lambda}$. Moreover by the assumption (5),  the weight-grading on $\varinjlim_{\alpha\in A}\mathcal{F}_\alpha(V)$ is bounded from the below. Therefore the formal distribution $\sum_{n\in\mathbb{Z}}A_{(n)}z^{-n-1}$ is a field for any $A\in\varinjlim_{\alpha\in A}\mathcal{F}_\alpha(V)$. Thus $(\varinjlim_{\alpha\in A}\mathcal{F}_\alpha(V), \mathbf{1}, T, (n); n\in \mathbb{Z})$ is a vertex superalgebra. 
Therefore the  quadruple $\mathcal{V}=(\mathcal{F}, \underline{\mathbf{1}}, \underline{T}, \underline{(n)}; n\in\mathbb{Z})$ with $\underline{H}$ and $\underline{J}$ is an object of the category  $\textit{DegWt-}\mathit{VSA_{\mathbb{K}}}\textit{-IndSh}_X^A$. It remains to construct a strict isomorphism $\mathcal{V}|_{U_\lambda}\cong\mathcal{V}^\lambda$ for each $\lambda\in \Lambda$.
Let $\lambda\in\Lambda$. We set  
$$
\Phi^\lambda:=\bigl([\Phi^{\lambda, \alpha}_\alpha]\bigr)_{\alpha\in \mathrm{Ob}(A)},
$$
where $\Phi^{\lambda, \alpha}_\alpha$ is the usual sheaf isomorphism from $\mathcal{F}_\alpha|_{U_\lambda}$ to $\mathcal{F}_\alpha^\lambda$, which preserves the degree-weight-grading.  
This defines a morphism of ind-objects, or equivalently, 
$$
\Phi_\alpha^{\lambda, \alpha}\circ\mathcal{F}(f_{\alpha \alpha'})\sim\Phi_{\alpha'}^{\lambda, \alpha'},
$$
for any object $\alpha, \alpha'$ of $A$ and morphism $f_{\alpha \alpha'}: \alpha'\to \alpha$ in $A$. Indeed by the construction of $\mathcal{F}(f_{\alpha \alpha'})$, we have 
$$
\Phi_\alpha^{\lambda, \alpha}\circ (\mathcal{F}(f_{\alpha \alpha'})|_{U_\lambda})=\mathcal{F}^\lambda(f_{\alpha \alpha'})\circ\Phi_{\alpha'}^{\lambda, \alpha'},
$$
for each object $\alpha, \alpha'$ of $A$ and morphism $f_{\alpha \alpha'}: \alpha'\to \alpha$ in $A$. 
The relation $(\Phi^\lambda|_{U_\lambda\cap U_\mu}) \circ (\Phi^\mu|_{U_\mu\cap U_\lambda})^{-1}=\vartheta_{\lambda \mu}$ holds since $(\Phi^{\lambda, \alpha}_\alpha|_{U_\lambda\cap U_\mu}) \circ (\Phi^{\mu, \alpha}_\alpha|_{U_\mu\cap U_\lambda})^{-1}=\vartheta_{\lambda \mu, \alpha}^\alpha$ for all $\alpha\in \mathrm{Ob}(A)$. Moreover $\Phi^\lambda$ is a strict isomorphism of VSA-inductive sheaves.  In other words, $\Phi^\lambda$ commutes with $\underline{(n)}$, $\underline{T}$ and $\underline{\mathbf{1}}$ and in addition the strict inverse morphism $(\Phi^\lambda)^{-1}$ exists. This follows from the construction of the operators $\underline{(n)}$, $\underline{T}$, $\underline{\mathbf{1}}$ and from the definition of  $\Phi^\lambda$. 
Moreover $\Phi^\lambda$ commutes with the Hamiltonians and the degree-grading operators since each $\Phi_\alpha^{\lambda, \alpha}$ does. 
Thus the existence part is proved. The uniqueness part is proved by the argument as in usual sheaf cases as well as the arguments used above with the fact that the morphisms $\mathcal{F}^\lambda(f_{\alpha'\alpha})$ are injective for all $f_{\alpha'\alpha}\in\mathrm{Hom}_A(\alpha, \alpha')$.
\end{proof}

We can also glue morphisms.
Let $X=\bigcup_{\lambda \in \Lambda}U_\lambda$ 
be as before.  
Let $\bigl((\mathcal{V}^\lambda)_{\lambda\in\Lambda},$ $ (\vartheta_{\lambda \mu})_{\lambda, \mu \in \Lambda}\bigr)$, $ \bigl((\mathcal{V}'^\lambda)_{\lambda\in\Lambda},$  $(\vartheta'_{\lambda \mu})_{\lambda, \mu \in \Lambda}\bigr)$ be families of degree-weight-graded VSA-inductive sheaves with strict isomorphisms as in Proposition \ref{prop: GLUING VsaIndShs}. In other words, 
$\mathcal{V}^\lambda=\bigl(\mathcal{F}^\lambda, \underline{\mathbf{1}}^\lambda, \underline{T}^\lambda,$ $ \underline{(n)}^\lambda; n\in\mathbb{Z} \bigr)$ and  $\mathcal{V}'^\lambda=\bigl(\mathcal{F}'^\lambda, \underline{\mathbf{1}'}^\lambda, \underline{T}'^\lambda, \underline{(n)}'^\lambda; n\in\mathbb{Z} \bigr)$ are objects of $\textit{DegWt-}\mathit{VSA_{\mathbb{K}}}\textit{-IndSh}_{U_\lambda}^A$, and  $\vartheta_{\lambda \mu}: \mathcal{F}^\mu|_{U_\mu \cap U_\lambda} \to \mathcal{F}^\lambda|_{U_\lambda \cap U_\mu}$ and  $\vartheta'_{\lambda \mu}: \mathcal{F}'^\mu|_{U_\mu \cap U_\lambda} \to \mathcal{F}'^\lambda|_{U_\lambda \cap U_\mu}$ are strict isomorphisms of degree-weight-graded VSA-inductive sheaves on $U_\lambda \cap U_\mu$ such that the conditions 
$(0)$-$(5)$ hold.
Let $\mathcal{V}$ with $(\Phi^\lambda)_{\lambda \in \Lambda}$ and $\mathcal{V}'$ with  $(\Phi'^\lambda)_{\lambda\in\Lambda}$ be objects of $\textit{DegWt-}\mathit{VSA_{\mathbb{K}}}\textit{-IndSh}_X^A$ with strict isomorphisms obtained by gluing $\bigl((\mathcal{V}^\lambda)_{\lambda\in\Lambda}, (\vartheta_{\lambda \mu})_{\lambda, \mu \in \Lambda}\bigr)$ and $\bigl((\mathcal{V}'^\lambda)_{\lambda\in\Lambda},$ $ (\vartheta'_{\lambda \mu})_{\lambda, \mu \in \Lambda}\bigr)$, respectively. 
Suppose given a family of morphisms of ind-objects of sheaves $(F^\lambda: \mathcal{V}^\lambda \to \mathcal{V}'^\lambda)_{\lambda\in\Lambda}$ such that $\vartheta'_{\lambda \mu}\circ (F^\mu|_{U_\mu\cap U_\lambda}) = (F^\lambda|_{U_\lambda\cap U_\mu})\circ \vartheta_{\lambda \mu}$
for all $\lambda, \mu \in \Lambda$. 
We assume that there exist a map $j_F: \mathrm{Ob}(A)\to\mathrm{Ob}(A)$ and sheaf morphisms $F^{\lambda, j_F(\alpha)}_\alpha$ with $\alpha\in \mathrm{Ob}(A)$ and $\lambda\in \Lambda$ such that 
$
F^\lambda=\bigl( [F^{\lambda, j_F(\alpha)}_\alpha] \bigr)_{\alpha\in \mathrm{Ob}(A)}
$
for all $\lambda \in \Lambda$.

\begin{proposition}\label{prop: GLUING MORPHISMS OF VsaIndSh}
In the above situation, there exists a unique strict morphism $F: \mathcal{V}\to \mathcal{V}'$ of ind-objects of sheaves on $X$
such that 
$\Phi'^\lambda\circ F|_{U_\lambda}=F^\lambda\circ \Phi^\lambda$ 
for any $\lambda\in\Lambda$. 
Moreover if the morphisms $F^\lambda$ are all morphisms of VSA-inductive sheaves, the resulting morphism of ind-objects $F$ is also a morphism of VSA-inductive sheaves.  
\end{proposition}
\begin{proof}
We can construct $F: \mathcal{V}\to \mathcal{V}'$, following the same argument as for the construction of $\underline{(n)}$ in Proposition \ref{prop: GLUING VsaIndShs}. The uniqueness part is proved by the same argument as in usual sheaf cases as well as the same arguments as in the proof of Proposition \ref{prop: GLUING VsaIndShs}. 
The latter half of this proposition  is also checked by  arguments similar to those above.
\end{proof}

Consider three families of degree-weight-graded VSA-inductive sheaves with strict isomorphisms as in Proposition \ref{prop: GLUING VsaIndShs},  $\bigl((\mathcal{V}^\lambda)_{\lambda\in\Lambda}, (\vartheta_{\lambda \mu})_{\lambda, \mu \in \Lambda}\bigr)$, $\bigl((\mathcal{V}'^\lambda)_{\lambda\in\Lambda},$ $ (\vartheta'_{\lambda \mu})_{\lambda, \mu \in \Lambda}\bigr)$ and $\bigl((\mathcal{V}''^\lambda)_{\lambda\in\Lambda},$ $ (\vartheta''_{\lambda \mu})_{\lambda, \mu \in \Lambda}\bigr)$.  
Suppose given families of morphisms of ind-objects of  sheaves $(F^\lambda: \mathcal{V}^\lambda \to \mathcal{V}'^\lambda)_{\lambda\in\Lambda}$ and $(F'^\lambda: \mathcal{V}'^\lambda \to \mathcal{V}''^\lambda)_{\lambda\in\Lambda}$ as in Proposition  \ref{prop: GLUING MORPHISMS OF VsaIndSh}. In other words, $\vartheta'_{\lambda \mu}\circ (F^\mu|_{U_\mu\cap U_\lambda}) = (F^\lambda|_{U_\lambda\cap U_\mu})\circ \vartheta_{\lambda \mu}$ and $\vartheta''_{\lambda \mu}\circ (F'^\mu|_{U_\mu\cap U_\lambda}) = (F'^\lambda|_{U_\lambda\cap U_\mu})\circ \vartheta'_{\lambda \mu}$ hold for all $\lambda, \mu \in \Lambda$, and moreover, there exist maps $j_F: \mathrm{Ob}(A)\to\mathrm{Ob}(A)$, $j_{F'}: \mathrm{Ob}(A)\to\mathrm{Ob}(A)$ and sheaf morphisms $F^{\lambda, j_F(\alpha)}_\alpha$, $F'^{\lambda, j_{F'}(\alpha)}_\alpha$ with $\alpha\in \mathrm{Ob}(A)$ and $\lambda\in \Lambda$ such that 
$
F^\lambda=\bigl( [F^{\lambda, j_F(\alpha)}_\alpha] \bigr)_{\alpha\in \mathrm{Ob}(A)}
$
and
$
F'^\lambda=\bigl( [F'^{\lambda, j_{F'}(\alpha)}_\alpha] \bigr)_{\alpha\in \mathrm{Ob}(A)}
$
for all $\lambda \in \Lambda$. 
Let $F: \mathcal{V}\to \mathcal{V}'$, $F': \mathcal{V}'\to\mathcal{V}''$ and $G: \mathcal{V}\to\mathcal{V}''$ be the morphisms obtained by gluing $(F^\lambda)_{\lambda\in\Lambda}$, $(F'^\lambda)_{\lambda\in\Lambda}$ and $(F'^\lambda\circ F^\lambda)_{\lambda\in\Lambda}$, respectively. 

\begin{proposition}\label{prop: FUNCTORIALITY OF GLUING}
In the above situation, the composite $F'\circ F$  agrees with the morphism $G$.
\end{proposition}
\begin{proof}
This proposition is a direct corollary of Proposition \ref{prop: GLUING MORPHISMS OF VsaIndSh}.  
\end{proof}

\begin{remark}\label{rem: Isomorphic as VSA-inductive sheaves if locally isomorphic}
Two VSA-inductive sheaves are strictly isomorphic if they are strictly isomorphic locally via strict isomorphisms which coincide on the overlaps of their domains.   
\end{remark}

\begin{remark}\label{rem: morphisms coincides if so locally}
Two morphisms of ind-objects between VSA-inductive sheaves  coincide with each other if they coincide locally.
\end{remark}

\subsection{From Presheaves to VSA-Inductive Sheaves}\label{subsection: From Presheaves to VSA-Inductive Sheaves}
We construct VSA-inductive sheaves from presheaves of vertex superalgebras with some properties. 

We denote by $\mathit{Presh}_X(\textit{DegWt-VSA}_\mathbb{K})_\mathrm{bdw, sh}$  the full subcategory of the category of presheaves  on $X$ of degree-weight-graded vertex superalgebras over $\mathbb{K}$ whose objects are presheaves $\Tilde{\mathcal{V}}$ of degree-weight-graded vertex superalgebras on $X$ such that the weight-grading on $\Tilde{\mathcal{V}}(U)$ is bounded from the below uniformly with respect to open subsets $U\subset X$ and the subpresheaf $\Tilde{\mathcal{V}}[n]$ defined by the assignment
$
U\mapsto \Tilde{\mathcal{V}}(U)[n]
$
is a sheaf of super vector spaces for any $n \in\mathbb{Z}$.

Let $\Tilde{\mathcal{V}}$ be an object of $\mathit{Presh}_X(\textit{DegWt-VSA}_\mathbb{K})_\mathrm{bdw, sh}$.
We set
$$
\Tilde{\mathcal{V}}[\le N]:=\bigoplus_{n\le N}\Tilde{\mathcal{V}}[n].
$$
for $N\in \mathbb{N}$. 
These are sheaves by the assumptions.  Consider the canonical inductive system of sheaves $(\Tilde{\mathcal{V}}[\le N])_{N\in \mathbb{N}}$. Then  the corresponding ind-object $``\displaystyle\varinjlim_{n\in \mathbb{N}}"\Tilde{\mathcal{V}}[\le N]$ has a canonical VSA-inductive sheaf structure induced by the  morphisms $\underline{\mathbf{1}}^0: \mathbb{K}_X\to \Tilde{\mathcal{V}}[\le 0]$ defined by $\mathbb{K}\to \Gamma(\Tilde{\mathcal{V}}[\le 0]), 1\mapsto \mathbf{1}$, $T_N: \Tilde{\mathcal{V}}[\le N]\to\Tilde{\mathcal{V}}[\le N+1]$, and $(n)_{N, M}: \Tilde{\mathcal{V}}[\le N]\times \Tilde{\mathcal{V}}[\le M]\to \Tilde{\mathcal{V}}[\le N+M-n-1]$, where $\mathbf{1}$, $T_N$, $(n)_{N, M}$ come from the vertex superalgebra structure on $\Tilde{\mathcal{V}}$. Note that the VSA-inductive sheaf  $``\displaystyle\varinjlim_{n\in \mathbb{N}}"\Tilde{\mathcal{V}}[\le N]$ has a canonical degree-weight-graded structure. We refer to this degree-weight-graded VSA-inductive sheaf as the degree-weight-graded VSA-inductive sheaf associated with $\Tilde{\mathcal{V}}$.

\begin{lemma}\label{lem: MAKE VsaIndSh from PRESHEAVES OF Z-GRADED Vsa}
There exists a canonical functor 
$$
\mathit{Presh}_X(\textit{DegWt-VSA}_\mathbb{K})_\mathrm{bdw, sh} \to \textit{DegWt-}\mathit{VSA_{\mathbb{K}}}\textit{-IndSh}_X^\mathbb{N},
$$
sending an object $\Tilde{\mathcal{V}}$ of $\mathit{Presh}_X(\textit{DegWt-VSA}_\mathbb{K})_\mathrm{bdw, sh}$ to the degree-weight-graded VSA-inductive sheaf associated with $\Tilde{\mathcal{V}}$.
\end{lemma}
\begin{proof}
For a morphism $\Tilde{F}: \Tilde{\mathcal{V}}\to\Tilde{\mathcal{W}}$ in $\mathit{Presh}_X(\textit{DegWt-VSA}_\mathbb{K})_\mathrm{bdw, sh}$, we assign the morphism $F:=\Bigl([\Tilde{F}|_{\Tilde{\mathcal{V}}[\le N]}\bigr]\Bigr)_{N\ge 0}$ in $\textit{DegWt-}\mathit{VSA_{\mathbb{K}}}\textit{-IndSh}_X^\mathbb{N}$. 
The functoriality follows  from the definition directly.
\end{proof}

\begin{remark}
The composite of the functor $\underrightarrow{\mathrm{Lim}}\, $ given in \eqref{eq: functor from VSA-ISh} and the one given in Lemma \ref{lem: MAKE VsaIndSh from PRESHEAVES OF Z-GRADED Vsa} is the identity functor.
\end{remark}

\begin{remark}\label{rem: homogeneous morphism of presheaves induces that of ind-objects}
Let $\Tilde{\mathcal{V}}, \Tilde{\mathcal{W}}$ be objects of the category $\mathit{Presh}_X(\textit{DegWt-VSA}_\mathbb{K})_\mathrm{bdw, sh}$ and $\mathcal{V}, \mathcal{W}$  the VSA-inductive sheaves associated with $\Tilde{\mathcal{V}}, \Tilde{\mathcal{W}}$, respectively. Let $\Tilde{F}: \Tilde{\mathcal{V}}\to\Tilde{\mathcal{W}}$ be a homogeneous linear  morphism of degree $d$. In other words, $\Tilde{F}(U): \Tilde{\mathcal{V}}(U)\to \Tilde{\mathcal{W}}(U)$ is a homogeneous linear map of degree $d$ for any open subset $U\subset X$. Then $\Tilde{F}$ induces a morphism $F$ of ind-objects:  $F:=\bigl([\Tilde{F}|_{\Tilde{\mathcal{V}}[\le N]}]\bigr)_{N\ge 0}: \mathcal{V}\to \mathcal{W}$. Moreover the corresponding morphism $\underrightarrow{\mathrm{Lim}}\,  F$  of presheaves is nothing but the morphism $\Tilde{F}: \Tilde{\mathcal{V}}=\underrightarrow{\mathrm{Lim}}\, \mathcal{V}\to\Tilde{\mathcal{W}}=\underrightarrow{\mathrm{Lim}}\, \mathcal{W}$.
\end{remark}

\begin{remark}
The functor given in Lemma \ref{lem: MAKE VsaIndSh from PRESHEAVES OF Z-GRADED Vsa} commutes with the restriction. More precisely, if $U\subset X$ is an open subset and 
$\mathcal{V}$ is the degree-weight-graded VSA-inductive sheaf  associated with an object $\Tilde{\mathcal{V}}$ of  $\mathit{Presh}_X(\textit{DegWt-VSA}_\mathbb{K})_\mathrm{bdw, sh}$ then we have $\mathcal{V}|_U=\mathcal{V}_U$, where $\mathcal{V}_U$ stands for the VSA-inductive sheaf associated with the presheaf  $\Tilde{\mathcal{V}}|_U$. Here we denote by $\Tilde{\mathcal{V}}|_U$ the presheaf, not its sheafification, obtained by restricting $\Tilde{\mathcal{V}}$ to $U$. 
\end{remark}

\subsection{More on VSA-Inductive Sheaves}
Let $\varphi: X\to Y$ be a continuous map between topological spaces.
Consider the functor induced by the push-forward functor $\varphi_*$ of presheaves: 
\begin{align}
\label{eq: push-forward functor of ind-objects}
\varphi_*&: \mathrm{Ind}(\mathit{Presh}_X(Vec_\mathbb{K}^\mathrm{super}))\longrightarrow \mathrm{Ind}(\mathit{Presh}_Y(Vec_\mathbb{K}^\mathrm{super})),  \\
\text{objects}&:\quad\quad\quad  \mathcal{F}=``\varinjlim_{\alpha\in A}\mathcal{F}"\longmapsto \varphi_*\mathcal{F}:=``\varinjlim_{\alpha\in A}"\varphi_*\mathcal{F},   \\
\label{df: push-forward of morphism of ind-objects} 
\text{morphisms}&:  F=\bigl([F_\alpha^{j(\alpha)}]\bigr)_{\alpha\in \mathrm{Ob}(A)}\longmapsto \varphi_*F:=\bigl([\varphi_*F_\alpha^{j(\alpha)}]\bigr)_{\alpha\in \mathrm{Ob}(A)}.
\end{align}
The push-forward of bilinear morphism of ind-objects  is given in a  way  similar to that  in \eqref{df: push-forward of morphism of ind-objects}.

\begin{lemma}
Let $\varphi: X\to Y$ be a continuous map between topological spaces and $\mathcal{V}=\bigl( \mathcal{F}, \underline
{\mathbf{1}}, \underline{T}, \underline{(n)}; n\in \mathbb{Z}\bigr)$  a VSA-inductive sheaf on $X$. The quadruple $\varphi_*\mathcal{V}:=\bigl( \varphi_*\mathcal{F}, \varphi_*\underline
{\mathbf{1}}, \varphi_*\underline{T}, \varphi_*\underline{(n)}; n\in \mathbb{Z}\bigr)$
is a VSA-inductive sheaf on $Y$. Moreover if $F$ is a morphism in $\mathit{VSA_{\mathbb{K}}}\textit{-IndSh}_X$ then $\varphi_*F$ is a morphism in $\mathit{VSA_{\mathbb{K}}}\textit{-IndSh}_Y$.
\end{lemma}
\begin{proof}
We can immediately see that $\varphi_*\mathcal{V}$ is a VSA-inductive sheaf. The latter half of this lemma follows from the functoriality of the push-forward functor of presheaves.
\end{proof}

By the above lemma, we can restrict the functor \eqref{eq: push-forward functor of ind-objects} to obtain a functor 
$\varphi_*: \mathit{VSA_{\mathbb{K}}}\textit{-IndSh}_X\to\mathit{VSA_{\mathbb{K}}}\textit{-IndSh}_Y$ sending an object $\mathcal{V}$ to $\varphi_*\mathcal{V}$ and a morphism $F$ to $\varphi_*F$. We call the VSA-inductive sheaf $\varphi_*\mathcal{V}$ the \textbf{push-forward} of $\mathcal{V}$.

\begin{remark}
The push-forward functor  commutes with the functor $\underrightarrow{\mathrm{Lim}}\, $ as well as the one given in Lemma \ref{lem: MAKE VsaIndSh from PRESHEAVES OF Z-GRADED Vsa}.
\end{remark}

Let $\mathcal{V}_1=\bigl( \mathcal{F}_1, \underline
{\mathbf{1}}_1, \underline{T}_1, \underline{(n)}_1; n\in \mathbb{Z}\bigr)$ and $\mathcal{V}_2=\bigl( \mathcal{F}_2, \underline
{\mathbf{1}}_2, \underline{T}_2, \underline{(n)}_2;$ $n\in \mathbb{Z}\bigr)$ be  VSA-inductive sheaves on topological spaces $X_1$ and $X_2$, respectively. A \textbf{morphism} of vertex superalgebra  inductive sheaves from $\mathcal{V}_1$ to $\mathcal{V}_2$ is by definition a pair $(\varphi, \Phi)$ of a continuous map $\varphi: X_2 \to X_1$ and a base-preserving morphism $\Phi: \mathcal{V}_1 \to \varphi_*\mathcal{V}_2$ of VSA-inductive sheaves on $Y$.

\begin{remark}\label{rem: ALL VSA-inductive sheaves form a category}
VSA-inductive sheaves form a category with morphisms defined  above, whose composition is defined by 
$
(\varphi', \Phi')\circ(\varphi, \Phi):=(\varphi'\circ\varphi, \Phi'\circ\varphi'_*\Phi)
$
for morphisms of VSA-inductive sheaves $(\varphi, \Phi): \mathcal{V}_1\to \mathcal{V}_2$ and $(\varphi', \Phi'): \mathcal{V}_2\to \mathcal{V}_3$. 
\end{remark}

\begin{notation}
Denote by $\mathit{VSA_{\mathbb{K}}}\textit{-IndSh}$ the category of VSA-inductive sheaves obtained in Remark \ref{rem: ALL VSA-inductive sheaves form a category}.
\end{notation}

\begin{remark}
If $\varphi: X\to Y$ is a continuous map and $(\mathcal{V},\underline{H} ,\underline{J})$ is a degree-weight-graded VSA-inductive sheaf on $X$, then $(\varphi_*\mathcal{V}, \varphi_*\underline{H}, \varphi_*\underline{J})$ is a degree-weight-graded VSA-inductive sheaf on $Y$. In a similar way above, the category of degree-weight-graded VSA-inductive sheaves is defined. 
\end{remark}

\begin{notation}
Let us denote by $\textit{DegWt-}\mathit{VSA_{\mathbb{K}}}\textit{-IndSh}$ the category of degree-weight-graded VSA-inductive sheaves. 
\end{notation}

Let $\mathcal{V}$ be a VSA-inductive sheaf and  $\mathcal{F}=``\displaystyle\varinjlim_{\alpha\in A}"\mathcal{F}_{\alpha}$ the underlying ind-object. Then we have two canonical ind-objects of sheaves,  $\mathcal{F}_{\bar{0}}:=``\displaystyle\varinjlim_{\alpha\in A}"(\mathcal{F}_\alpha)_{\bar{0}}$ and $\mathcal{F}_{\bar{1}}:=``\displaystyle\varinjlim_{\alpha\in A}"(\mathcal{F}_\alpha)_{\bar{1}}$.
In addition, suppose that $\mathcal{V}$ is degree-graded. Then we have ind-objects of sheaves, $\mathcal{F}^l:=``\displaystyle\varinjlim_{\alpha\in A}"\mathcal{F}_\alpha^l$, where $\mathcal{F}_\alpha^l$ is the subsheaf of degree $l\in\mathbb{Z}$.

\begin{definition}
Let $\mathcal{V}=\bigl( \mathcal{F}, \underline
{\mathbf{1}}, \underline{T}, \underline{(n)}; n\in \mathbb{Z}\bigr)$ be a degree-weight-graded VSA-inductive sheaf.  
A \textbf{differential} on $\mathcal{V}$ is an odd  morphism of ind-objects, $D: \mathcal{F}\to\mathcal{F}$ such that 
\begin{equation}\label{eq: deg 1 wt 0 condition of differential on VSA-inductive sheaf}
[\underline{H}, D]=0, \quad [\underline{J}, D]=D, 
\end{equation}
\begin{equation}\label{eq: square 0 condition of differential on VSA-inductive sheaf}
D^2=0, 
\end{equation}
\begin{equation}\label{eq: derivation condition of differential on VSA-inductive sheaf}
D\circ\underline{(n)}-(-1)^{i}\underline{(n)}\circ(\mathrm{id}\times D)=\underline{(n)}\circ (D\times\mathrm{id}), 
\end{equation}
on $\mathcal{F}_{\bar{i}}\times\mathcal{F}$ for all $n\in\mathbb{Z}$ and $i=0, 1$, and 
\begin{equation}\label{eq: degree derivation condition of differential on VSA-inductive sheaf}
D\circ\underline{(n)}-(-1)^l\underline{(n)}\circ(\mathrm{id}\times D)=\underline{(n)}\circ (D\times\mathrm{id}), 
\end{equation}
on $\mathcal{F}^l\times\mathcal{F}$ for all $l\in\mathbb{Z}$. Here $\underline{H}$ is the Hamiltonian of $\mathcal{V}$ and  $\underline{J}$ is the degree-grading operator of $\mathcal{V}$.  
\end{definition}

By a \textbf{differential degree-weight-graded VSA-inductive sheaf}, we mean a degree-weight-graded VSA-inductive sheaf given a differential.

\begin{remark}\label{rem: presheaf  of differential VSA-inductive sheaf}
Let $(\mathcal{V}, D)$ be a differential degree-weight-graded VSA-inductive sheaf. Then $(\underrightarrow{\mathrm{Lim}}\, \mathcal{V},$ $\underrightarrow{\mathrm{Lim}}\,  D)$ is a presheaf of differential degree-weight-graded vertex superalgebras. 
\end{remark}

\begin{notation}
We denote by $\textit{Diff-}\textit{DegWt-}\mathit{VSA_{\mathbb{K}}}\textit{-IndSh}_X$ the category of differential degree-weight-graded VSA-inductive sheaves on a topological space $X$, whose morphisms are morphisms of degree-weight-graded VSA-inductive sheaves on $X$ commuting with the differentials.   
\end{notation}

\begin{remark}
Let $\varphi: X\to Y$ be a continuous map between topological spaces and $(\mathcal{V}, D)$ a differential degree-weight-graded  VSA-inductive sheaf on $X$. Then the pair $(\varphi_*\mathcal{V}, \varphi_*D)$ is a differential degree-weight-graded VSA-inductive sheaf on $Y$. 
\end{remark}

\begin{notation}
We denote by $\textit{Diff-}\textit{DegWt-}\mathit{VSA_{\mathbb{K}}}\textit{-IndSh}$ the category of differential degree-weight-graded VSA-inductive sheaves, whose morphisms are morphisms of degree-weight-graded VSA-inductive sheaves which commute with the differentials. 
\end{notation}

\section{Chiral Lie Algebroid Cohomology}\label{section: Chiral Lie Algebroid Cohomology}
In this section, we will construct VSA-inductive sheaves associated with vector bundles. For that purpose, we will construct VSA-inductive sheaves on  an affine space. Then we will glue them to obtain a VSA-inductive sheaves on a manifold, using the facts proved in the preceding section.

By a manifold, we will mean a $C^\infty$-manifold.
Let $M$ be a manifold. We simply denote by $TM$ the tangent bundle tensored by $\mathbb{K}$, $TM\otimes_{\mathbb{R}}\mathbb{K}$.  We use a similar notation for  the cotangent bundle $T^*M$, the sheaf of functions $C^\infty=C^\infty_M$ and the sheaf of vector fields $\mathscr{X}=\mathscr{X}_M$. 
We will mean  a real or complex vector bundle simply by a vector bundle,  following $\mathbb{K}$ is $\mathbb{R}$ or $\mathbb{C}$.
 $TM\otimes_{\mathbb{R}}\mathbb{K}$.

\subsection{Lie Algebroids}\label{subsection: Lie Algebroids}
In this subsection, we recall the notion of Lie algebroids. We refer the reader to \cite{DZ05,HM90,Mac05,Pra67,Vai97,Vai91} for more details.

Let $M$ be a manifold. A \textit{Lie algebroid} on $M$ is a vector bundle $A$ together with a vector bundle map $a: A \to TM$ over $M$, called the anchor map of $A$, and a $\mathbb{K}$-linear Lie bracket $[\,,]$ on $\Gamma(A)$ such that
$
[X,fY]=f[X,Y]+a(X)(f)Y 
$
for all $X, Y\in \Gamma(A), f\in C^\infty(M).$
When $\mathbb{K}$ is $\mathbb{R}$, the corresponding Lie algebroids are called \textit{real Lie algebroids}. Similarly when $\mathbb{K}=\mathbb{C}$, the corresponding Lie algebroids are called \textit{complex Lie algebroids}.

\begin{example}[tangent bundles]\label{ex: tangent bundle Lie algebroid}
The tangent bundle $TM$ of a manifold $M$ with bracket the Lie bracket of vector fields and with anchor the identity of $TM$ is a Lie algebroid on $M.$
\end{example}

\begin{example}[transformation Lie algebroids]\label{ex: transformation Lie algebroid}
Let 
$
\rho: \mathfrak{g} \to \mathscr{X}(M)
$
be an infinitesimal action of a Lie algebra $\mathfrak{g}$ on a manifold $M.$ Then there is a natural Lie algebroid structure on the trivial vector bundle $M \times \mathfrak{g}$ with the anchor map 
$
a_\rho(m, \xi):=\rho(\xi)(m)
$
for $(m, \xi)\in M\times \mathfrak{g}$ and the bracket on $\Gamma(M\times \mathfrak{g})\cong C^\infty(M, \mathfrak{g})$
$$
[X,Y]_\rho:=[X,Y]_{\mathfrak{g}}+a_\rho(X)Y-a_\rho(Y)X,
$$
for $X, Y\in C^\infty(M, \mathfrak{g}).$
This Lie algebroid $(M\times \mathfrak{g}, a_\rho, [\,,]_\rho)$ is called the \textit{transformation Lie algebroid} associated with $\rho$.
\end{example}

\begin{example}[cotangent Lie algebroids]\label{ex: cotangent Lie algebroid}
Let $(M, \Pi)$ be a Poisson manifold with the Poisson bivector field $\Pi.$ Consider the map 
$$
\Pi^\sharp: T^*M \to TM,\ \Pi^\sharp(\alpha)(\beta):=\Pi(\alpha, \beta)\ \text{for}\ \alpha, \beta \in \Gamma(T^*M).
$$ 
Then, with $\Pi^\sharp$ as the anchor map, the cotangent bundle $T^*M$ becomes a Lie algebroid on $M$, where the Lie bracket on $\Gamma(T^*M)$ is given by
$$
[\alpha,\beta]:=d(\Pi(\alpha, \beta))+i_{\Pi^\sharp\alpha}d\beta-i_{\Pi^\sharp\beta}d\alpha,
$$
for $\alpha, \beta \in \Gamma(T^*M).$
The Lie algebroid $(T^*M, \Pi^\sharp, [\,,])$ is called the \textit{cotangent Lie algebroid} of $(M, \Pi).$ 
\end{example}

Let us recall the notion of the Lie algebroid representation. 
Let $(A, a, [\,,])$ be a Lie algebroid on a manifold $M$ and $E$ a vector bundle on $M$. 
An $A$-\textit{connection} on $E$ is a map 
$
\nabla: \Gamma(A) \times \Gamma(E) \to \Gamma(E)
$
such that
\begin{enumerate}[$\bullet$]
     \setlength{\topsep}{1pt}
     \setlength{\partopsep}{0pt}
     \setlength{\itemsep}{1pt}
     \setlength{\parsep}{0pt}
     \setlength{\leftmargin}{20pt}
     \setlength{\rightmargin}{0pt}
     \setlength{\listparindent}{0pt}
     \setlength{\labelsep}{3pt}
     \setlength{\labelwidth}{30pt}
     \setlength{\itemindent}{10pt}
\item $\nabla_{X+Y}s=\nabla_X s+\nabla_Y s,$
\item $\nabla_X(s+s')=\nabla_X s+\nabla_X s',$
\item $\nabla_{fX}s=f\nabla_X s,$
\item $\nabla_X(fs)=f\nabla_X s +a(X)(f)s,$
\end{enumerate}
for all $X, Y\in \Gamma(A),\ s, s'\in \Gamma(E),$ and $f\in C^\infty(M).$
An $A$-connection $\nabla$ on $E$ is said to be \textit{flat} if 
$
\nabla_{[X,Y]}s=\nabla_X(\nabla_Y s)-\nabla_Y(\nabla_X s)
$
for all $X, Y\in \Gamma(A)$ and $s\in \Gamma(E).$ 
A flat $A$-connection on $E$ is also called a \textit{representation} of $A$ on $E$.

\begin{example}[trivial representations]\label{ex: trivial representation of Lie algebroid}
Let $A$ be a Lie algebroid on $M$ and $V$  a vector space. The \textit{trivial representation} of $A$ on $M \times V$ is given by
$
\nabla_X f:=a(X)(f)
$
for $X\in \Gamma(A)$ and $f: M \to V.$
\end{example}

Let us now recall the definition of the Lie algebroid cohomology. 
For a Lie algebroid $(A, a, [\,,])$ on a manifold $M$ and a representation $(E, \nabla)$ of $A$ on a vector bundle $E$, consider a complex $\Omega^\bullet(A; E):=\Gamma((\wedge^\bullet A^*)\otimes E)$ and a differential $d_{\text{Lie}}^E:  \Omega^\bullet(A; E) \to \Omega^{\bullet+1}(A; E)$ defined by
\begin{multline*}
(d_{\text{Lie}}^E \omega)(X_1, \dots, X_{n+1}) = \sum_{i=1}^{n+1}(-1)^{i+1}\nabla_{X_i}(\omega(X_1, \dots, \check{X_i}, \dots, X_{n+1}))  \\
+\sum_{1\le i<j\le n+1}(-1)^{i+j}\omega([X_i,X_j], X_1, \dots, \check{X_i}, \dots, \check{X_j}, \dots, X_{n+1}),
\end{multline*}
for $\omega \in \Omega^n(A; E)$ and $X_1, \dots, X_{n+1}\in \Gamma(A).$
The cohomology space $H^\bullet(A; E)$ of the complex $(\Omega^\bullet(A; E), d_{\text{Lie}}^E)$ is called the \textit{Lie algebroid cohomology} with coefficients in $E$.
 
Any $X\in \Gamma(A)$ induces the \textit{Lie derivative} $L_X: \Omega^\bullet(A; E) \to \Omega^\bullet(A; E)$
and the \textit{interior product} $\iota_X: \Omega^\bullet(A; E) \to \Omega^{\bullet-1}(A; E):$
\begin{gather*}
(L_X\omega)(X_1, \dots, X_n)=\nabla_X(\omega(X_1, \dots, X_n))
-\sum_{i=1}^{n}\omega(X_1, \dots, [X,X_i], \dots, X_n), \\
(\iota_X\omega)(X_1, \dots, X_{n-1})=\omega(X, X_1, \dots, X_{n-1}),
\end{gather*}
where $\omega \in \Omega^n(A; E)$ and $X_1, \dots, X_n \in \Gamma(A).$
The Lie derivatives are derivations of degree $0$ and the interior products are derivations of degree $-1$. They satisfy the Cartan relations:
\begin{gather*}
[d_{\text{Lie}}^E,\iota_X]=L_X, \\
[L_X,L_Y]=L_{[X,Y]}, \\
[L_X,\iota_Y]=\iota_{[X,Y]}, \\
[\iota_X,\iota_Y]=0,
\end{gather*}
for all $X, Y\in \Gamma(A).$

When $(E, \nabla)$ is the trivial representation on the trivial line bundle, we simply denote by $(\Omega^\bullet(A), d_{\text{Lie}})$ and $H^\bullet(A)$ the corresponding complex and cohomology, respectively.

\subsection{VSA-Inductive Sheaves on $\mathbb{R}^m$}
In this subsection, we define an important object of the category  $\textit{DegWt-}\mathit{VSA_{\mathbb{K}}}\textit{-IndSh}_{\mathbb{R}^m}^\mathbb{N}$, which we will denote by  $\Omega_\mathrm{ch}(\mathbb{R}^{m|r})$.  
To this end,  we first construct an object, denoted by $\underrightarrow{\Omega_{\mathrm{ch}}}(\mathbb{R}^{m|r})$, of the category $\mathit{Presh}_{\mathbb{R}^m}(\textit{DegWt-VSA}_\mathbb{K})_\mathrm{bdw, sh}$, following the argument in \cite{LL}.

Fix natural numbers $m$ and $r$.
We consider the supermanifold $\mathbb{R}^{m|r}=\bigl(\mathbb{R}^m,$ $ C^\infty_{\mathbb{R}^m}\otimes_{\mathbb{K}} \bigwedge_{\mathbb{K}}(\theta^1, \dots, \theta^r)\bigr)$. 
We denote  the structure sheaf  $C^\infty_{\mathbb{R}^m}\otimes_\mathbb{K} \bigwedge_\mathbb{K}(\theta^1, \dots, \theta^r)$ by  $C^\infty_{\mathbb{R}^{m|r}}$.
Let $U$ be an open subset of $\mathbb{R}^m$. 
Consider the supercommutative superalgebra of functions on $U\subset \mathbb{R}^{m|r}$,
$
C^{\infty}_{\mathbb{R}^{m|r}}(U)=C^\infty_{\mathbb{R}^m}(U)\otimes_{\mathbb{K}}\bigwedge\nolimits_{\mathbb{K}}(\theta^1, \dots, \theta^r),
$
even derivations $\partial/\partial x^1, \dots, \partial/\partial x^m$
and odd derivations $\partial/\partial \theta^1, \dots, \partial/\partial \theta^r$ on it. Here $(x^1, \dots, x^m, \theta^1, \dots, \theta^r)$ is a standard supercoordinate on $U$. 
We consider the supercommutative Lie superalgebra 
$$
D^{m|r}(U):=\mathrm{Span}_\mathbb{K}\{ \partial/\partial x^i, \partial/\partial \theta^j | i=1, \dots, m,\ j=1, \dots, r\}, 
$$
acting on $C^{\infty}_{\mathbb{R}^{m|r}}(U)$ naturally, and set
$$
\Lambda^{m|r}(U):= D^{m|r}(U) \ltimes C^{\infty}_{\mathbb{R}^{m|r}}(U).
$$
We put 
$$
\underrightarrow{\Omega_{\mathrm{ch}}}(\mathbb{R}^{m|r})(U):=N\bigl(\Lambda^{m|r}(U), 0\bigr)/\mathcal{I}^{m|r}(U),
$$
where $\mathcal{I}^{m|r}(U)$ is the ideal of the affine vertex superalgebra  $N(\Lambda^{m|r}(U), 0) \cong O(\Lambda^{m|r}(U), 0)$  (see Example \ref{ex:affine va}) 
generated by 
\begin{gather*}
\frac{d}{dz}f(z)-\sum_{i=1}^m\,{\baselineskip0pt\lineskip0.3pt\vcenter{\hbox{$\cdot$}\hbox{$\cdot$}}\,{\frac{d}{dz}x^i(z) \frac{\partial f}{\partial x^i}(z)}\,\vcenter{\hbox{$\cdot$}\hbox{$\cdot$}}}\,-\sum_{j=1}^r\baselineskip0pt\lineskip0.3pt\vcenter{\hbox{$\cdot$}\hbox{$\cdot$}}\,\frac{d}{dz}\theta^j(z) \frac{\partial f}{\partial \theta^j}(z)\,\vcenter{\hbox{$\cdot$}\hbox{$\cdot$}}, \\
(fg)(z)-\baselineskip0pt\lineskip0.3pt\vcenter{\hbox{$\cdot$}\hbox{$\cdot$}}\,f(z)g(z)\,\vcenter{\hbox{$\cdot$}\hbox{$\cdot$}}, \ 
1(z)-\mathrm{id},
\end{gather*}
where $f, g\in C^\infty_{\mathbb{R}^{m|r}}(U)\subset\Lambda^{m|r}(U)$.
For $A\in N\bigl(\Lambda^{m|r}(U), 0\bigr)$, we denote by $\overline{A}$ the corresponding element in $\underrightarrow{\Omega_{\mathrm{ch}}}(\mathbb{R}^{m|r})(U)$. 

The Lie superalgebra $\Lambda^{m|r}(U)$ has compatible two gradings when we give $D^{m|r}(U)$, $C^{\infty}_{\mathbb{R}^{m|r}}(U)$ weight $0, -1$, respectively, and if we give $\partial/\partial x^i$, $\partial/\partial\theta^j$ and $C^\infty_{\mathbb{R}^m}(U)\otimes\bigwedge^l(\theta^1, \dots, \theta^r)$ degree $0$, $-1$ and $l$, respectively. This induces a degree-weight-grading on the vertex superalgebra $N\bigl(\Lambda^{m|r}(U), 0\bigr)$ (see Example  \ref{ex: another grading on N(g, 0)}). Then the ideal $\mathcal{I}^{m|r}(U)$ is a homogeneous ideal. Therefore the vertex superalgebra $\underrightarrow{\Omega_{\mathrm{ch}}}(\mathbb{R}^{m|r})(U)$ becomes a degree-weight-graded vertex superalgebra.

We set
$$
\Gamma_{m}(U):=N\bigl(D_m(U)\ltimes C^\infty_{\mathbb{R}^m}(U), 0\bigr)/\mathcal{I}_m(U),
$$
where $D_m(U)$ is the commutative Lie algebra $\mathrm{Span}_{\mathbb{K}}\{ \partial/\partial x^i | i=1, \dots, m\}$
 and $\mathcal{I}_m(U)$ is the ideal of the affine vertex algebra $N\bigl(D_m(U)\ltimes C^\infty_{\mathbb{R}^m}(U), 0\bigr)$ generated by 
\begin{gather}\label{eq: relations for Gamma_m(U)}
\frac{d}{dz}f(z)-\sum_{i=1}^m\baselineskip0pt\lineskip0.3pt\vcenter{\hbox{$\cdot$}\hbox{$\cdot$}}\,\frac{d}{dz}x^i(z)\frac{\partial f}{\partial x^i}(z)
\,\vcenter{\hbox{$\cdot$}\hbox{$\cdot$}}\,, \\
(fg)(z)-\baselineskip0pt\lineskip0.3pt\vcenter{\hbox{$\cdot$}\hbox{$\cdot$}}\,f(z)g(z)\,\vcenter{\hbox{$\cdot$}\hbox{$\cdot$}}, \ 
1(z)-\mathrm{id},
\end{gather}
where $f, g\in C^\infty_{\mathbb{R}^m}(U)$.
Regard $D_m(U)\ltimes C^\infty_{\mathbb{R}^m}$ as a degree-weight-graded Lie subalgebra of $\Lambda^{m|r}(U)$. Then $\Gamma_m(U)$ becomes a degree-weight-graded vertex algebra (the degree-grading is trivial).
For $A\in N\bigl(D_m(U)\ltimes C^\infty_{\mathbb{R}^m}(U), 0\bigr)$, we denote by $\overline{A}$ the corresponding element in $\Gamma_m(U)$.

\begin{remark}\label{rem: generators of Gamma_m(U)}
$\Gamma_m(U)$ is spanned by the vectors of the form
$$
\overline{\partial/\partial x^{i_1}}_{(n_1)}\dots\overline{\partial/\partial x^{i_k}}_{(n_k)}\overline{x^{i'_1}}_{(n'_1)}\dots\overline{x^{i'_{k'}}}_{(n'_{k'})}\overline{f}_{(-1)}\mathbf{1}, 
$$
with $n_1, \dots, n_k\le-1,\  n'_1, \dots, n'_{k'}\le-2,$ $\ i_1, \dots, i_k, i'_1, \dots,  i'_{k'}\in \{ 1, \dots, m\}$ and $f\in C^\infty_{\mathbb{R}^m}(U)$.
This follows from the relations \eqref{eq: relations for Gamma_m(U)} by induction.
\end{remark}

We can rewrite the vertex superalgebra $\underrightarrow{\Omega_{\mathrm{ch}}}(\mathbb{R}^{m|r})(U)$, using the vertex superalgebras $\Gamma_m(U)$ and $\mathcal{E}(W_r)$.  Here $\mathcal{E}(W_r)$ is the $bc$-system associated with  $W_r=\mathrm{Span}_{\mathbb{K}}\{ \theta^j |\ j=1, \dots, r\}$  regarded as an even vector space (see Example \ref{ex:bc-systems} for the definition of $bc$-systems and the notation used in the following proof). 
The $\mathbb{Z}$-graded vertex superalgebra $ \Gamma_m(U)\otimes\mathcal{E}(W_r)$ is degree-weight-graded when the degree-grading is given by the operator $j_{bc, (0)}$, where $j_{bc}=-\sum_{j=1}^r b^{\partial/\partial\theta^j}_{-1}c^{\theta^j}_{0}\mathbf{1}$. 

\begin{lemma}\label{lem: omega_ch=beta-gamma_bc}
There exists a canonical isomorphism of degree-weight-graded vertex superalgebras
$$
\underrightarrow{\Omega_{\mathrm{ch}}}(\mathbb{R}^{m|r})(U)\cong \Gamma_m(U)\otimes \mathcal{E}(W_r).
$$
\end{lemma}
\begin{proof}
The assertion follows by the same argument as in \cite{LL}, where the case when $m=r$ is considered. The canonical isomorphism is induced by the linear map $\alpha: \Lambda^{m|r}(U) \to \Gamma_m(U)\otimes \mathcal{E}(W_r)$ defined by
\begin{align*}
\alpha(\partial/\partial x^i)&:= \overline{\partial/\partial x^i}_{(-1)}\mathbf{1} \otimes \mathbf{1},\ i=1, \dots, m, \\
\alpha(\partial/\partial \theta^j)&:= \mathbf{1} \otimes b^{\partial/\partial\theta^{j}}_{-1}\mathbf{1},\ j=1, \dots, r, \\
\alpha(f\otimes \theta^{j_1}\dotsm\theta^{j_l})&:=\overline{f}_{(-1)}\mathbf{1} \otimes c^{\theta^{j_1}}_0\dotsm c^{\theta^{j_l}}_0\mathbf{1}, \\ 
&\text{for}\ f\otimes \theta^{j_1}\cdots\theta^{j_l}\in C^\infty_{\mathbb{R}^m}(U)\otimes \bigwedge(\theta^1, \dots, \theta^r),
\end{align*}
where we regard $(\partial/\partial\theta^{1}, \dots, \partial/\partial\theta^{r})$ as the basis dual to $(\theta^{1}, \dots, \theta^{r})$ for $W_r^*$.
\end{proof}

We identify $\underrightarrow{\Omega_{\mathrm{ch}}}(\mathbb{R}^{m|r})(U)$ with $\Gamma_m(U)\otimes \mathcal{E}(W_r)$ via this isomorphism.

We denote by $b^j_n$ and $c^j_n$ the element $b^{\partial/\partial\theta^{j}}_{n}=b^{\partial/\partial\theta^{j}}\otimes t^n$ and $c^{\theta^{j}}_n=\theta^{j}\otimes t^{n-1}dt$ in $\mathfrak{j}(W_r)$ (see Example \ref{ex:bc-systems} for the definition of the Lie superalgebra $\mathfrak{j}(W_r)$).

Consider the vector space
$V_m(U):=\oplus_{i=1}^m\mathbb{K}x^i\subset C^\infty(U)$. 
We regard $(\partial/\partial x^1,$ $\dots,$ $\partial/\partial x^m)$ as the basis dual to $(x^1, \dots, x^m)$ and identify the Lie algebra  $D_m(U)=\mathrm{Span}_\mathbb{K}\{  \partial/\partial x^1, \dots,  \partial/\partial x^m \}$ as the dual vector space $V_m(U)^*$.
Consider the Heisenberg Lie algebra associated with $D_m(U)$:
$$
\mathfrak{h}_m(U):=\mathfrak{h}(D_m(U))=(D_m(U)[t^{\pm1}]\oplus V_m(U)[t^{\pm1}]dt)\oplus \mathbb{K}\mathbf{\tau},
$$
with commutation relations 
$
[\beta^{\partial/\partial x^i}_p,\gamma^{x^{i'}}_q]=\delta_{p, -q}\frac{\partial}{\partial x^i}(x^{i'})\mathbf{\tau},
$
where $\beta^{\partial/\partial x^i}_p$ and $\gamma^{x^{i'}}_q$ stand for, respectively, $\partial/\partial x^i\otimes t^p$ and $x^{i'}\otimes t^{q-1}dt$ as in Example \ref{ex:beta_gamma-systems}.

Consider the Heisenberg Lie algebra associated with the vector space $\mathbb{K}^m$:
$$
\mathfrak{h}_m:=\mathfrak{h}(\mathbb{K}^m)=\mathbb{K}^m[t^{\pm1}]\oplus (\mathbb{K}^m)^*[t^{\pm1}]dt\oplus\mathbb{K}\mathbf{\tau}.
$$
Let $(e_1, \dots, e_m)$ be the standard basis of $\mathbb{K}^m$ and $(\phi_1, \dots, \phi_m)$ the dual basis.
We denote by $\beta^i_p$ and $\gamma^j_q$ the elements $e_1\otimes t^p$ and $\phi_i \otimes t^{q-1}dt$.
We identify $\mathfrak{h}_m(U)$ with $\mathfrak{h}_m$ through the isomorphism induced by 
$
V_m(U)\to \mathbb{K}^m,\ x^i\mapsto e^i,\ i=1, \dots, m.
$
We also denote $\beta^{\partial/\partial x^i}_p$ and $\gamma^{x^{i'}}_q$ by $\beta^i_p$ and $\gamma^{i'}_q$, respectively. 
We emphasize that the Lie algebra $\mathfrak{h}_m$ does not depend on the open subset $U$.

Let ${\mathfrak{h}_m(U)}_{\ge0}\subset \mathfrak{h}_m(U)$ be the Lie subalgebra generated by $\mathbf{\tau}$ and $\beta^i_p, \gamma^i_p$ with $i=1, \dots, m$ and $p, q \ge0$. 
We make $C^\infty_{\mathbb{R}^m}(U)$ a ${\mathfrak{h}_m(U)}_{\ge0}$-module by setting 
\begin{gather*}
\beta^i_p\cdot f=\gamma^i_p\cdot f=0,\  p>0, \\
\beta^i_0\cdot f=\frac{\partial}{\partial x^i}(f),\ \gamma^i_0\cdot f=x^i f,\ \mathbf{\tau} \cdot  f=f,
\end{gather*}
where $i=1, \dots, m$ and $f\in C^\infty_{\mathbb{R}^m}(U)$. 
As in \cite{LL}, we consider the $\mathfrak{h}_m(U)$-module
$$
\Tilde{\Gamma}_m(U):=U(\mathfrak{h}_m(U))\otimes_{U({\mathfrak{h}_m(U)}_{\ge0})}C^\infty_{\mathbb{R}^m}(U).
$$
Following \cite{LL}, we define the action of the commutative Lie algebra $C^\infty_{\mathbb{R}^m}(U)[t^{\pm1}]$ on $\Tilde{\Gamma}_m(U)\cong U(\mathfrak{h}_m(U)_{< 0})\otimes C^\infty_{\mathbb{R}^m}(U)$ as follows, where $\mathfrak{h}_m(U)_{<0}$ stands for the Lie subalgebra of $\mathfrak{h}_m(U)$ generated by $\beta^i_p, \gamma^i_p$ with $p<0$ and $i=1, \dots, m$. 
We define the action of $ft^k\in C^\infty_{\mathbb{R}^m}(U)[t^{\pm1}]$, denoted by $f_{(k)}$, using the PBW monomial basis of $U(\mathfrak{h}_m(U)_{< 0})$ for the basis $\beta^i_p, \gamma^i_p,\ p<0, i=1, \dots, m$ with an order such that $\beta^i_p>\gamma^{i'}_{q}$ for any $i, i'=1, \dots, m$ and $p, q<0$.
First, for $ft^k\in C^\infty_{\mathbb{R}^m}(U)[t^{\pm1}]$ with $k\ge -1$ we set 
$$
f_{(k)}(1\otimes g):=\delta_{k, -1}(1\otimes fg), \quad g\in C^\infty_{\mathbb{R}^m}(U),
$$
and for $ft^k\in C^\infty_{\mathbb{R}^m}(U)[t^{\pm1}]$ with $k < -1$ we inductively define $f_{(k)}$ on $\mathbb{K}1\otimes C^\infty_{\mathbb{R}^m}(U)$ by setting 
$$
f_{(k)}(1\otimes g):=\frac{1}{k+1}\sum_{i=1}^m\sum_{q<0}q\gamma^i_q\Bigl(\frac{\partial f}{\partial x^i}\Bigr)_{(k-q)}(1\otimes g), \quad g\in  C^\infty_{\mathbb{R}^m}(U).
$$
Next we define $f_{(k)}(P\otimes g)$ for $ft^k\in C^\infty_{\mathbb{R}^m}(U)[t^{\pm1}]$, a PBW monomial $P$ of  positive length and $g\in C^\infty_{\mathbb{R}^m}(U)$. The PBW monomial $P$ is of the form $\gamma^i_p P'$ or $\beta^i_p P'$ with a PBW monomial $P'$ of length less than that of $P$.
We set 
$$
f_{(k)}(P\otimes g):=
\begin{cases}
\gamma^i_pf_{(k)}(P'\otimes g), & \text{if}\ P=\gamma^i_p P', \\
\beta^i_pf_{(k)}(P'\otimes g)-\bigl(\frac{\partial f}{\partial x^i}\bigl)_{(k+p)}(P'\otimes g), & \text{if}\ P=\beta^i_p P'.
\end{cases}
$$
Thus we get operators $f_{(k)}$ on $\Tilde{\Gamma}_m(U)$.

We put 
$$
\gamma^i(z):=\sum_{n\in\mathbb{Z}}\gamma^i_n z^{-n}, \quad
\beta^i(z):=\sum_{n\in\mathbb{Z}}\beta^i_n z^{-n-1},
$$
for $i=1, \dots, m$ and 
$$
f(z):=\sum_{k\in\mathbb{Z}}f_{(k)}z^{-k-1}, 
$$
for $f \in C^\infty_{\mathbb{R}^m}(U)$.
Note that when $f=x^i$, we have $x^i(z)=\gamma^i(z)$, or equivalently, $x^i_{(k)}=\gamma^i_{k+1}$ for all $k\in \mathbb{Z}$.
For $f\in C^\infty_{\mathbb{R}^m}(U)$,  we also use the notation 
$$
f(z)=\sum_{n\in\mathbb{Z}}f_n z^{-n},
$$
or equivalently, $f_{(k)}=f_{k+1}$ for $k\in \mathbb{Z}$ and 
simply denote by $f$ the element $1\otimes f\in \Tilde{\Gamma}(U)$. 
Since $\beta^i(z)$ and $\gamma^i(z)$ come from the action of the
Heisenberg Lie algebra $\mathfrak{h}_m(U)$, we have 
\begin{gather}
\label{eq: [beta(z),beta(w)]=0}
[\beta^i(z),\beta^{i'}(w)]=0, \\
[\gamma^i(z),\gamma^{i'}(w)]=0, \quad [\beta^i(z),\gamma^{i'}(w)]=\delta_{i, i'}\delta(z-w),  \notag 
\end{gather}
for $i, i'=1, \dots, m$.
The following lemmas are proved in \cite[Section 2.5]{LL}. 

\begin{lemma}[Lian-Linshaw] \label{lem: [beta(z),f(w)]=df/dx(w)delta(z-w)}
The following hold:
\begin{equation}\label{eq: [beta(z),f(w)]=df/dx(w)delta(z-w)}
[\gamma^i(z),f(w)]=0, \quad [\beta^i(z),f(w)]=\frac{\partial f}{\partial x^i}(w)\delta(z-w).
\end{equation}
for all $f \in C^\infty_{\mathbb{R}^m}(U)$ and $i=1, \dots, m$.
\end{lemma}

\begin{lemma}[Lian-Linshaw]\label{lem: relations for f(z)}
The following hold:
\begin{gather}
\label{eq: [f(z),g(z)]=0}
[f(z),g(w)]=0, \\
\label{eq: df(z)=df/dx(z)dx(z)}
\frac{d}{dz}f(z)=\sum_{i=1}^m\frac{d}{dz}\gamma^i(z)\frac{\partial f}{\partial x^i}(z),
\end{gather}
for $f, g \in C^\infty_{\mathbb{R}^m}(U)$ and
\begin{equation}
\label{eq: 1(z)=1}
1(z)=\mathrm{id}.
\end{equation}
\end{lemma}

As a corollary of the above two lemmas, we have the following proposition, which corresponds to \cite[Corollary 2.21]{LL}.  

\begin{proposition}\label{prop: VERTEX ALGEBRA STRUCTURE ON GAMMA-TILDE}
$\Tilde{\Gamma}_m(U)$ has a unique vertex algebra structure such that
\begin{enumerate}[$\bullet$]
     \setlength{\topsep}{1pt}
     \setlength{\partopsep}{0pt}
     \setlength{\itemsep}{1pt}
     \setlength{\parsep}{0pt}
     \setlength{\leftmargin}{20pt}
     \setlength{\rightmargin}{0pt}
     \setlength{\listparindent}{0pt}
     \setlength{\labelsep}{3pt}
     \setlength{\labelwidth}{30pt}
     \setlength{\itemindent}{0pt}
\item the vacuum vector is $\mathbf{1} :=1\otimes 1,$
\item the translation operator is $T:=\mathrm{Res}_{z=0}\sum_{i=1}^m\,{\baselineskip0pt\lineskip0.3pt\vcenter{\hbox{$\cdot$}\hbox{$\cdot$}}\,\beta^i(z)\partial_z\gamma^i(z)\,\vcenter{\hbox{$\cdot$}\hbox{$\cdot$}}}$, 
\item the vertex operators satisfy $Y(1\otimes f, z)=f(z)$ for $f\in C^\infty_{\mathbb{R}^m}(U)$ and $Y(\beta^i_{-1}\otimes1, z)=\beta^i(z)$ for $i=1, \dots, m.$
\end{enumerate}
\end{proposition}
\begin{proof}
We will apply the existence theorem of Frenkel-Kac-Radul-Wang (see \cite[Proposition 3.1]{FKRW95}). The relations $[T,f(z)]=\partial_zf(z)$ and $[T,\beta^i(z)]$ $=\partial_z\beta^i(z)$ are checked by direct computations with \eqref{eq: [beta(z),f(w)]=df/dx(w)delta(z-w)}. By Lemmas \ref{lem: [beta(z),f(w)]=df/dx(w)delta(z-w)} and \ref{lem: relations for f(z)}, the other conditions in the existence theorem are satisfied.
\end{proof}

By Proposition \ref{prop: VERTEX ALGEBRA STRUCTURE ON GAMMA-TILDE} and $f_{(-1)}g=fg$, we have
\begin{equation}\label{eq: (fg)(z)=f(z)g(z)}
(fg)(z)=f(z)g(z),
\end{equation}
for $f, g\in C^\infty_{\mathbb{R}^m}(U)$. 
By OPEs  \eqref{eq: [beta(z),beta(w)]=0}, \eqref{eq: [f(z),g(z)]=0} and \eqref{eq: (fg)(z)=f(z)g(z)}, we have a vertex algebra morphism
$
N\bigl(D_m(U)\ltimes C^\infty_{\mathbb{R}^m}(U), 0\bigr)\to \Tilde{\Gamma}_m(U)
$
sending $f$ and $\partial/\partial x^i$ to $f=1\otimes f$ and $\beta^i_{-1}\mathbf{1}=\beta^i_{-1}\otimes1$, respectively. 
Moreover by the relations \eqref{eq: df(z)=df/dx(z)dx(z)}, \eqref{eq: 1(z)=1} and \eqref{eq: (fg)(z)=f(z)g(z)}, this morphism factors through the ideal $\mathcal{I}_m(U)$, and hence we have a morphism from $\Gamma_m(U)$. We can see this map is bijective, by taking into consideration the form of the basis of $\Tilde{\Gamma}_m(U)$ and Lemma \ref{rem: generators of Gamma_m(U)}. Thus we have the following.

\begin{proposition}\label{prop: GAMMA=GAMMA-TILDE}
The linear map 
\begin{align*}
D_m(U)\ltimes C^\infty_{\mathbb{R}^m}(U)&\to \Tilde{\Gamma}_m(U), \\
 f &\to 1\otimes f, \\
\partial/\partial x^i &\to \beta^i_{-1}\otimes1,
\end{align*}
induces an isomorphism of vertex algebras
$$
\Gamma_m(U) \xrightarrow{\cong} \Tilde{\Gamma}_m(U).
$$
\end{proposition}

We identify $\Gamma_m(U)$ with $\Tilde{\Gamma}_m(U)$ through this isomorphism.

Let 
$\mathfrak{h}_{m, \ge0}$ and $\mathfrak{h}_{m, <0}$ be the Lie subalgebras of $\mathfrak{h}_m$ generated by $\beta^i_p, \gamma^i_p$ with $p\ge 0,\ i=1, \dots, m$, and $\beta^i_p, \gamma^i_p$ with $p<0,\ i=1, \dots, m$, respectively.
By Poincar\'{e}-Birkhoff-Witt theorem, the decomposition $\mathfrak{h}_m=\mathfrak{h}_{m, <0}\oplus \mathfrak{h}_{m, \ge0}$ induces an isomorphism as vector spaces 
$$
\Tilde{\Gamma}_m(U)\cong U(\mathfrak{h}_{m, <0})\otimes C^\infty_{\mathbb{R}^m}(U).
$$
Using the isomorphism in Lemma \ref{prop: GAMMA=GAMMA-TILDE}, we have an isomorphism
\begin{equation}\label{eq: basis of Gamma}
\Gamma_m(U)\cong U(\mathfrak{h}_{m, <0})\otimes C^\infty_{\mathbb{R}^m}(U).
\end{equation}
Therefore, by Lemma \ref{lem: omega_ch=beta-gamma_bc} we have 
\begin{equation}\label{eq: omega_ch=beta-gamma_bc-C_infty}
\underrightarrow{\Omega_{\mathrm{ch}}}(\mathbb{R}^{m|r})(U)\cong U(\mathfrak{h}_{m, <0})\otimes C^\infty_{\mathbb{R}^m}(U)\otimes \mathcal{E}(W_r).
\end{equation}
By restricting this isomorphism to the space of  weight $n\in \mathbb{N}$, $\underrightarrow{\Omega_{\mathrm{ch}}}(\mathbb{R}^{m|r})(U)[n]$, we get an isomorphism
\begin{equation}\label{eq: omega_ch=beta-gamma_bc-C_infty[n]}
\underrightarrow{\Omega_{\mathrm{ch}}}(\mathbb{R}^{m|r})(U)[n]\cong (U(\mathfrak{h}_{m, <0})\otimes \mathcal{E}(W_r))[n]\otimes C^\infty_{\mathbb{R}^m}(U),
\end{equation}
where $(U(\mathfrak{h}_{m, <0})\otimes \mathcal{E}(W_r))[n]$ stands for the weight $n$ space of $U(\mathfrak{h}_{m, <0})\otimes \mathcal{E}(W_r)$ with respect to the grading induced by the $\mathbb{Z}_{\ge0}$-grading on $\mathcal{E}(W_r)$ and the grading on $U(\mathfrak{h}_{m, <0})$ given by $\mathrm{wt}\,\beta^i_p=\mathrm{wt}\,\gamma^i_p=-p$. Note that $(U(\mathfrak{h}_{m, <0})\otimes \mathcal{E}(W_r))[n]$ is finite-dimensional.

For two open subsets $V\subset U\subset \mathbb{R}^m$, we define the restriction map 
$
\mathrm{Res}^{m|r}_{V, U}: \underrightarrow{\Omega_{\mathrm{ch}}}(\mathbb{R}^{m|r})(U)\to \underrightarrow{\Omega_{\mathrm{ch}}}(\mathbb{R}^{m|r})(V)
$ 
as follows. 
The restriction map $C^\infty_{\mathbb{R}^{m|r}}(U)\to C^\infty_{\mathbb{R}^{m|r}}(V)$ induces a Lie superalgebra morphism
$
\Lambda^{m|r}(U)\to\Lambda^{m|r}(V)
$
and this morphism induces a  morphism
$
N(\Lambda^{m|r}(U), 0) \to N(\Lambda^{m|r}(V), 0)
$ of vertex superalgebras. 
Since this vertex superalgebra morphism is induced by an algebra morphism, namely, a morphism preserving the product and the unit, the generators of the ideal $\mathcal{I}^{m|r}(U)$ is mapped into $\mathcal{I}^{m|r}(V)$.
Therefore we have a vertex superalgebra morphism 
from $\underrightarrow{\Omega_{\mathrm{ch}}}(\mathbb{R}^{m|r})(U)$ to $\underrightarrow{\Omega_{\mathrm{ch}}}(\mathbb{R}^{m|r})(V)$.
We define the restriction map
$$
\mathrm{Res}^{m|r}_{V, U}: \underrightarrow{\Omega_{\mathrm{ch}}}(\mathbb{R}^{m|r})(U)\to \underrightarrow{\Omega_{\mathrm{ch}}}(\mathbb{R}^{m|r})(V),
$$
as this vertex superalgebra morphism. Note that this restriction map  preserves the degree-weight-grading.

The assignment 
$$
U\mapsto\underrightarrow{\Omega_{\mathrm{ch}}}(\mathbb{R}^{m|r})(U),
$$
and the maps $\mathrm{Res}^{m|r}_{V, U}$ define a presheaf of degree-weight-graded vertex superalgebras. 
We denote by $\underrightarrow{\Omega_{\mathrm{ch}}}(\mathbb{R}^{m|r})$ the presheaf.
We claim that this presheaf is an object of the category $\mathit{Presh}_{\mathbb{R}^m}(\textit{DegWt-VSA}_\mathbb{K})_\mathrm{bdw, sh}$ (see Section \ref{subsection: From Presheaves to VSA-Inductive Sheaves} for the definition of this category). It suffices to check the presheaves $\underrightarrow{\Omega_{\mathrm{ch}}}(\mathbb{R}^{m|r})[n]$ are sheaves for all $n\in \mathbb{N}$, where
the presheaf $\underrightarrow{\Omega_{\mathrm{ch}}}(\mathbb{R}^{m|r})[n]$ is defined by the assignment 
$$
U\mapsto \Gamma(U, \underrightarrow{\Omega_{\mathrm{ch}}}(\mathbb{R}^{m|r}))[n],
$$
and the maps $\mathrm{Res}^{m|r}_{V, U}\big|_{\Gamma(U, \underrightarrow{\Omega_{\mathrm{ch}}}(\mathbb{R}^{m|r}))[n]}$.

The restriction map 
$
\mathrm{Res}^{m|r}_{V, U}\big|_{\Gamma(U, \underrightarrow{\Omega_{\mathrm{ch}}}(\mathbb{R}^{m|r}))[n]}
$
becomes
$$
\mathrm{id}\otimes\mathrm{Res}_{V, U}\!\!: \!(U(\mathfrak{h}_{m, <0})\otimes \mathcal{E}(W_r))[n]\,\otimes C^\infty_{\mathbb{R}^m}(U)\!\to\! (U(\mathfrak{h}_{m, <0})\otimes \mathcal{E}(W_r))[n]\,\otimes C^\infty_{\mathbb{R}^m}(V),
$$
through the isomorphism \eqref{eq: omega_ch=beta-gamma_bc-C_infty[n]}, where $\mathrm{Res}_{V, U}$ stands for the restriction map of $C^\infty_{\mathbb{R}^m}$. 
Therefore we have an isomorphism of presheaves
$$
\underrightarrow{\Omega_{\mathrm{ch}}}(\mathbb{R}^{m|r})[n] \cong (U(\mathfrak{h}_{m, <0})\otimes \mathcal{E}(W_r))[n]\,\otimes_\mathbb{K} C^\infty_{\mathbb{R}^m}.
$$
Since $C^\infty_{\mathbb{R}^m}$ is a sheaf and $(U(\mathfrak{h}_{m, <0})\otimes \mathcal{E}(W_r))[n]$ is finite-dimensional, the presheaf $(U(\mathfrak{h}_{m, <0})\otimes \mathcal{E}(W_r))[n]\,\otimes C^\infty_{\mathbb{R}^m}$ is a sheaf of super vector spaces, and therefore $\underrightarrow{\Omega_{\mathrm{ch}}}(\mathbb{R}^{m|r})[n]$ is also a sheaf.
Thus $\underrightarrow{\Omega_{\mathrm{ch}}}(\mathbb{R}^{m|r})$ is an object of the category $\mathit{Presh}_{\mathbb{R}^m}(\textit{DegWt-VSA}_\mathbb{K})_\mathrm{bdw, sh}$.

\begin{notation}
We denote by $\Omega_\mathrm{ch}(\mathbb{R}^{m|r})$ the VSA-inductive sheaf  associated with the object $\underrightarrow{\Omega_{\mathrm{ch}}}(\mathbb{R}^{m|r})$ of $\mathit{Presh}_{\mathbb{R}^m}(\textit{DegWt-VSA}_\mathbb{K})_\mathrm{bdw, sh}$. (See Lemma  \ref{lem: MAKE VsaIndSh from PRESHEAVES OF Z-GRADED Vsa}.)
\end{notation}

\begin{remark}\label{rem: SUPERFUNCTIONS INCLUDED IN OMEGA_CHIRAL}
By the isomorphism \eqref{eq: omega_ch=beta-gamma_bc-C_infty[n]} with $n=0$, there exists an isomorphism, 
$$
\Gamma(U, \underrightarrow{\Omega_{\mathrm{ch}}}(\mathbb{R}^{m|r}))[0]\cong C^\infty_{\mathbb{R}^m}(U)\otimes \bigwedge(c^1_0, \dots, c^r_0),
$$
for any open subset $U\subset \mathbb{R}^m$.
Therefore the following isomorphism of sheaves exists:
\begin{equation}\label{eq: WEIGHT ZERO SPACE IS CLASSICAL}
\underrightarrow{\Omega_{\mathrm{ch}}}(\mathbb{R}^{m|r})[0] \cong C^\infty_{\mathbb{R}^m}\otimes\bigwedge(\theta^1, \dots, \theta^r).
\end{equation}
Here $\theta^j$ is identified with $c^j_0=c^{\theta^j}_0
$ for $j=1, \dots, r$. 
\end{remark}
We identify $\underrightarrow{\Omega_{\mathrm{ch}}}(\mathbb{R}^{m|r})[0]$ with $C^\infty_{\mathbb{R}^m}\otimes\bigwedge(\theta^1, \dots, \theta^r)$ via the isomorphism \eqref{eq: WEIGHT ZERO SPACE IS CLASSICAL}.

\subsection{VSA-Inductive Sheaves for Vector Bundles}
Let $E$ be a vector bundle of rank $r$ on a manifold $M$ of dimension $m$. 

Let $\mathcal{U}=\bigl(U_\lambda \bigr)_{\lambda \in  \Lambda}$ be an arbitrary family of 
open subsets $U_\lambda$ in $M$ with a chart $\mathbf{x}_\lambda=(x_\lambda^1, \dots, x_\lambda^m)$ of $U_\lambda$, and a frame  $\mathbf{e}_\lambda=(e^1_\lambda, \dots, e^r_\lambda)$ of $E|_{U_\lambda}$ such that $\bigl\{ (U_\lambda, \mathbf{x}_\lambda)\bigr\}_{\lambda \in \Lambda}$ is an altas on $M$ and  is contained in the $C^\infty$-structure of $M$. We call such a family $\mathcal{U}$ a \textbf{framed covering} of $E$.
Let $\bigl(f^E_{\lambda \mu}=(f^{E, j j'}_{\lambda \mu})_{1\le j, j'\le r}\bigr)_{\lambda, \mu\in\Lambda}$ be the transition functions of $E$ associated with $(\mathbf{e}_\lambda)_{\lambda\in\Lambda}$. In other words, we have $e^j_\mu=\sum_{j'=1}^r f^{E, j' j}_{\lambda \mu}e^{j'}_\lambda$.
We denote by $\mathbf{c_\lambda}=(c_\lambda^1, \dots, c_\lambda^r)$ the frame dual to $\mathbf{e}_\lambda=(e^1_\lambda, \dots, e^ r_\lambda)$ for $E^*|_{U_\lambda}$.

We set
$$
\Omega_\mathrm{ch}(E; \mathcal{U})_\lambda:=(\mathbf{x_\lambda}^{-1})_*(\Omega_\mathrm{ch}(\mathbb{R}^{m|r})|_{\mathbf{x_\lambda}(U_\lambda)}).
$$
Note that   $\Omega_\mathrm{ch}(E; \mathcal{U})_\lambda$ is nothing but the VSA-inductive sheaf  associated with the object $(\mathbf{x_\lambda}^{-1})_*(\underrightarrow{\Omega_{\mathrm{ch}}}(\mathbb{R}^{m|r})|_{\mathbf{x_\lambda}(U_\lambda)})$ of  $\mathit{Presh}_{U_\lambda}(\textit{DegWt-VSA}_\mathbb{K})_\mathrm{bdw, sh}$, where  $\underrightarrow{\Omega_{\mathrm{ch}}}(\mathbb{R}^{m|r})|_{\mathbf{x_\lambda}(U_\lambda)}$ is the presheaf  restricted to $\mathbf{x_\lambda}(U_\lambda)$ but not its sheafification. Notice that $\underrightarrow{\mathrm{Lim}}\,\Omega_\mathrm{ch}(E; \mathcal{U})_\lambda=(\mathbf{x_\lambda}^{-1})_*(\underrightarrow{\Omega_{\mathrm{ch}}}(\mathbb{R}^{m|r})|_{\mathbf{x_\lambda}(U_\lambda)})$.

We consider the following two subpresheaves of $\underrightarrow{\Omega_{\mathrm{ch}}}(\mathbb{R}^{m|r})$:
\begin{align}
\underrightarrow{\Omega^{\gamma c}_{\mathrm{ch}}}(\mathbb{R}^{m|r})&: U \mapsto \Gamma(U, \underrightarrow{\Omega^{\gamma c}_{\mathrm{ch}}}(\mathbb{R}^{m|r})):=\bigl\langle C^\infty_{\mathbb{R}^m}(U)\bigr\rangle \otimes \langle c^1_0\mathbf{1}, \dots, c^r_0\mathbf{1}\rangle, \\
\underrightarrow{\Omega^{\gamma bc}_{\mathrm{ch}}}(\mathbb{R}^{m|r})&: U \mapsto \Gamma(U, \underrightarrow{\Omega^{\gamma bc}_{\mathrm{ch}}}(\mathbb{R}^{m|r})):=\bigl\langle C^\infty_{\mathbb{R}^m}(U)\bigr\rangle \otimes \mathcal{E}(W_r),
\end{align}
where $\bigl\langle C^\infty_{\mathbb{R}^m}(U)\bigr\rangle \subset\Gamma_m(U)$ and $\langle c^1_0\mathbf{1}, \dots, c^r_0\mathbf{1}\rangle\subset\mathcal{E}(W_r)$ stand for the subalgebra generated by $C^\infty_{\mathbb{R}^m}(U)$ and $\{c^1_0\mathbf{1}, \dots, c^r_0\mathbf{1}\}$, respectively. The presheaves $\underrightarrow{\Omega^{\gamma c}_{\mathrm{ch}}}(\mathbb{R}^{m|r})[n]$ and $\underrightarrow{\Omega^{\gamma bc}_{\mathrm{ch}}}(\mathbb{R}^{m|r})[n]$ are sheaves for all $n\in\mathbb{N}$ since we have an isomorphism $\bigl\langle C^\infty_{\mathbb{R}^m}(U)\bigr\rangle\cong U\bigl(\langle\gamma^i_p\ |\ p<0, i=1, \dots, m\rangle\bigr) \otimes C^\infty_{\mathbb{R}^m}(U)$ from the isomorphism in Proposition \ref{prop: GAMMA=GAMMA-TILDE}, where $ 
U\bigl(\langle\gamma^i_p\ |\ p<0, i=1, \dots, m\rangle\bigr)$ is the universal enveloping algebra of the commutative Lie subalgebra of $\mathfrak{h}_m$ generated by $\gamma^i_p$ with $p<0, i=1, \dots, m$. Therefore the presheaves $\underrightarrow{\Omega^{\gamma c}_{\mathrm{ch}}}(\mathbb{R}^{m|r})$ and  $\underrightarrow{\Omega^{\gamma bc}_{\mathrm{ch}}}(\mathbb{R}^{m|r})$ are objects of the category $\mathit{Presh}_{\mathbb{R}^m}(\textit{DegWt-VSA}_\mathbb{K})_\mathrm{bdw, sh}$.
Thus we have VSA-inductive sheaves  $\Omega^{\gamma c}_\mathrm{ch}(\mathbb{R}^{m|r})$ and $\Omega^{\gamma bc}_{\mathrm{ch}}(\mathbb{R}^{m|r})$, where $\Omega^{\gamma c}_\mathrm{ch}(\mathbb{R}^{m|r})$ and $\Omega^{\gamma bc}_{\mathrm{ch}}(\mathbb{R}^{m|r})$ are the VSA-inductive sheaves associated with $\underrightarrow{\Omega^{\gamma c}_{\mathrm{ch}}}(\mathbb{R}^{m|r})$ and $\underrightarrow{\Omega^{\gamma bc}_{\mathrm{ch}}}(\mathbb{R}^{m|r})$, respectively. Note that the vertex superalgebra $\Gamma(U, \underrightarrow{\mathrm{Lim}}\,  \Omega^{\gamma c}_\mathrm{ch}(\mathbb{R}^{m|r}))=\Gamma(U, \underrightarrow{\Omega^{\gamma c}_{\mathrm{ch}}}(\mathbb{R}^{m|r}))$  is  generated by the weight $0$ space for any open subset $U\subset \mathbb{R}^m$.


Set 
$$
\Gamma_{m}'(U):=N\bigl(C^\infty_{\mathbb{R}^m}(U), 0\bigr)/\mathcal{I}'_m(U),
$$
where $\mathcal{I}'_m(U)$ is the ideal of the affine vertex algebra $N\bigl(C^\infty_{\mathbb{R}^m}(U), 0\bigr)$ generated by
\begin{gather*}
\frac{d}{dz}f(z)-\sum_{i=1}^m\,{\baselineskip0pt\lineskip0.3pt\vcenter{\hbox{$\cdot$}\hbox{$\cdot$}}\,\frac{d}{dz}x^i(z)\frac{\partial f}{\partial x^i}(z)\,\vcenter{\hbox{$\cdot$}\hbox{$\cdot$}}}, \\
(fg)(z)-\,{\baselineskip0pt\lineskip0.3pt\vcenter{\hbox{$\cdot$}\hbox{$\cdot$}}\,f(z)g(z)\,\vcenter{\hbox{$\cdot$}\hbox{$\cdot$}}}, \quad
1(z)-\mathrm{id},
\end{gather*}
with $f, g\in C^\infty_{\mathbb{R}^m}(U)$.

\begin{remark}
The canonical morphism, 
$
\Gamma_m'(U)\to\Gamma_m(U),
$
induced by the inclusion of Lie algebras $C^\infty_{\mathbb{R}^m}(U)\to D_m(U)\ltimes C^\infty_{\mathbb{R}^m}(U)$ is injective. This follows from the form of the basis of $\Gamma_m(U)$. Therefore there exists an isomorphism of degree-weight-graded vertex algebras from $\Gamma_m'(U)$ to $\bigl\langle C^\infty_{\mathbb{R}^m}(U)\bigr\rangle \subset \Gamma_m(U)$.  
\end{remark}

We set
\begin{align}
\Omega^{\gamma c}_\mathrm{ch}(E; \mathcal{U})_\lambda&:=(\mathbf{x_\lambda}^{-1})_*(\Omega^{\gamma c}_\mathrm{ch}(\mathbb{R}^{m|r})|_{\mathbf{x_\lambda}(U_\lambda)}), \\
\Omega^{\gamma bc}_{\mathrm{ch}}(E; \mathcal{U})_\lambda&:=(\mathbf{x_\lambda}^{-1})_*(\Omega^{\gamma bc}_{\mathrm{ch}}(\mathbb{R}^{m|r})|_{\mathbf{x_\lambda}(U_\lambda)}).
\end{align}
We will glue $\bigl(\Omega^{\gamma c}_\mathrm{ch}(E; \mathcal{U})_\lambda\bigr)_{\lambda\in\Lambda}$  and $\bigl(\Omega^{\gamma bc}_{\mathrm{ch}}(E; \mathcal{U})_\lambda\bigr)_{\lambda\in\Lambda}$.

Fix $U_\lambda$ with $\mathbf{x}_\lambda, \mathbf{e}_\lambda$ and $U_{\Tilde{\lambda}}$ with $\mathbf{x}_{\Tilde{\lambda}}, \mathbf{e}_{\Tilde{\lambda}}$ such that $U_\lambda\cap U_{\Tilde{\lambda}}\neq \emptyset$. 
Set $U:=\mathbf{x}_\lambda(U_\lambda\cap U_{\Tilde{\lambda}})$,  $\Tilde{U}:=\mathbf{x}_{\Tilde{\lambda}}(U_{\Tilde{\lambda}}\cap U_\lambda)$ and 
$\varphi=\varphi_{\Tilde{\lambda} \lambda}:= (\mathbf{x}_{\Tilde{\lambda}}|_{U_{\Tilde{\lambda}}\cap U_\lambda})\circ(\mathbf{x}_\lambda|_{U_\lambda\cap U_{\Tilde{\lambda}}})^{-1}: U\to \Tilde{U}$. 
We construct strict isomorphisms of VSA-inductive sheaves
\begin{gather*}
\Omega^{\gamma c}_\mathrm{ch}(E; \mathcal{U})_{\Tilde{\lambda}}\big|_{U_{\Tilde{\lambda}}\cap U_\lambda} \to \Omega^{\gamma c}_\mathrm{ch}(E; \mathcal{U})_\lambda\big|_{U_\lambda\cap U_{\Tilde{\lambda}}}, \\
\Omega^{\gamma bc}_{\mathrm{ch}}(E; \mathcal{U})_{\Tilde{\lambda}}\big|_{U_{\Tilde{\lambda}}\cap U_\lambda} \to \Omega^{\gamma bc}_{\mathrm{ch}}(E; \mathcal{U})_\lambda\big|_{U_\lambda\cap U_{\Tilde{\lambda}}},
\end{gather*}
motivated by gluing vector fields on the supermanifold $\Pi E=\bigr(M, \bigwedge(E^*)\bigr)$.

Let $V\subset U_\lambda\cap U_{\Tilde{\lambda}}$ be an open subset. 
We denote by $\Tilde{\beta}^i_n$, $\Tilde{\gamma}^i_n$, $\Tilde{b}^j_n$ and $\Tilde{c}^j_n$ the operators $\beta^i_n$, $\gamma^i_n$, $b^j_n$ and $c^j_n$ on $\Gamma\Bigl(\mathbf{x}_{\Tilde{\lambda}}(V), \underrightarrow{\Omega_{\mathrm{ch}}}(\mathbb{R}^{m|r})\Bigr)$, respectively.
We define elements of the vertex superalgebra 
$\Gamma\Bigl(\mathbf{x}_\lambda(V), \underrightarrow{\Omega_{\mathrm{ch}}}(\mathbb{R}^{m|r})\Bigr)$ as follows:

\begin{gather}
\label{df: phi^*f}
\varphi^*\Tilde{f}:=\Tilde{f}, \\
\label{df: phi^*b}
\varphi^*\Tilde{b}^j:=\sum_{1\le j'\le r}f^{E, j' j}_{\lambda \Tilde{\lambda}}b^{j'}_{-1}\mathbf{1}, \\
\label{df: phi^*c}
\varphi^*\Tilde{c}^j:=\sum_{1\le j'\le r}f^{E, j j'}_{\Tilde{\lambda} \lambda}c^{j'}_0\mathbf{1},
\end{gather}
for $\Tilde{f}\in C^\infty(V)$, $i=1, \dots, m$ and $j=1, \dots, r$. Here we use a usual notation for functions, identifying    $C^\infty_{\mathbb{R}^m}\bigl(\mathbf{x}_\lambda(V)\bigr)$ and $C^\infty_{\mathbb{R}^m}\bigl(\mathbf{x}_{\Tilde{\lambda}}(V)\bigr)$ with $C^\infty(V)$ via $\mathbf{x}_\lambda$ and $\mathbf{x}_{\Tilde{\lambda}}$, respectively.

\begin{lemma}\label{lem: OPEs of OMEGA_CHIRAL}
The following OPEs hold:
\begin{gather}
(\varphi^*\Tilde{f})(z)(\varphi^*\Tilde{g})(w)\sim 0, \\
(\varphi^*\Tilde{b}^j)(z)(\varphi^*\Tilde{b}^{j'})(w)\sim 0, \\
(\varphi^*\Tilde{c}^j)(z)(\varphi^*\Tilde{c}^{j'})(w)\sim 0, \\
(\varphi^*\Tilde{b}^j)(z)(\varphi^*\Tilde{c}^{j'})(w)\sim \frac{\delta_{j, j'}}{z-w}, \\
(\varphi^*\Tilde{f})(z)(\varphi^*\Tilde{b}^j)(w)\sim0, \\
(\varphi^*\Tilde{f})(z)(\varphi^*\Tilde{c}^j)(w)\sim0.
\end{gather}
\end{lemma}
\begin{proof}
These OPEs are checked by direct computations.
\end{proof}

\begin{remark}
When we consider the transformation rules of vector fields, it is natural to set 
$$
\varphi^*\Tilde{\beta}^i:=\sum_{1\le i'\le m}\beta^{i'}_{-1}\frac{\partial x^{i'}}{\partial \Tilde{x}^i}+\sum_{1\le j, k, l\le r}\frac{\partial f^{E, j k}_{\lambda \Tilde{\lambda}}}{\partial \Tilde{x}^i}f^{E, k l}_{\Tilde{\lambda} \lambda}c^l_0b^j_{-1}\mathbf{1}. 
$$
But the required relation $(\varphi^*\Tilde{\beta}^i)(z)(\varphi^*\Tilde{\beta}^{i'})(w) \sim 0$ does not hold in general. 
\end{remark}

By Lemma \ref{lem: OPEs of OMEGA_CHIRAL}, we have morphisms of degree-weight-graded vertex superalgebras
\begin{align}
\underrightarrow{\vartheta'_{\lambda \Tilde{\lambda}}}(V): \Gamma\bigl(\mathbf{x}_{\Tilde{\lambda}}(V), \underrightarrow{\Omega^{\gamma c}_{\mathrm{ch}}}(\mathbb{R}^{m|r})\bigr) &\to \Gamma\bigl(\mathbf{x}_\lambda(V), \underrightarrow{\Omega^{\gamma c}_{\mathrm{ch}}}(\mathbb{R}^{m|r})\bigr),  \\ 
\Tilde{f} &\mapsto \varphi^*\Tilde{f}, \notag  \\ 
\Tilde{c}^j_0\mathbf{1} &\mapsto \varphi^*\Tilde{c}^j, \notag  \\ 
\intertext{and,}
\underrightarrow{\vartheta''_{\lambda \Tilde{\lambda}}}(V): \Gamma\bigl(\mathbf{x}_{\Tilde{\lambda}}(V), \underrightarrow{\Omega^{\gamma bc}_{\mathrm{ch}}}(\mathbb{R}^{m|r})\bigr) &\to \Gamma\bigl(\mathbf{x}_\lambda(V), \underrightarrow{\Omega^{\gamma bc}_{\mathrm{ch}}}(\mathbb{R}^{m|r})\bigr), \\
\Tilde{f} &\mapsto \varphi^*\Tilde{f}, \notag  \\ 
\Tilde{c}^j_0\mathbf{1} &\mapsto \varphi^*\Tilde{c}^j, \notag  \\ 
\Tilde{b}^j_0\mathbf{1} &\mapsto \varphi^*\Tilde{b}^j. \notag
\end{align}
These morphisms $\Bigl(\underrightarrow{\vartheta'_{\lambda \Tilde{\lambda}}}(V)\Bigr)_V$ and $\Bigl(\underrightarrow{\vartheta''_{\lambda \Tilde{\lambda}}}(V)\Bigr)_V$ form morphisms of presheaves of degree-weight-graded vertex superalgebras $\underrightarrow{\vartheta'_{\lambda \Tilde{\lambda}}}$ and  $\underrightarrow{\vartheta''_{\lambda \Tilde{\lambda}}}$, respectively.  Therefore by Lemma \ref{lem: MAKE VsaIndSh from PRESHEAVES OF Z-GRADED Vsa} we have strict morphisms of VSA-inductive sheaves
\begin{gather}
\vartheta'_{\lambda \Tilde{\lambda}}: \Omega^{\gamma c}_\mathrm{ch}(E; \mathcal{U})_{\Tilde{\lambda}}\big|_{U_{\Tilde{\lambda}}\cap U_\lambda} \to \Omega^{\gamma c}_\mathrm{ch}(E; \mathcal{U})_\lambda\big|_{ U_\lambda\cap U_{\Tilde{\lambda}}}, \\
\vartheta''_{\lambda \Tilde{\lambda}}: \Omega^{\gamma bc}_{\mathrm{ch}}(E; \mathcal{U})_{\Tilde{\lambda}}\big|_{U_{\Tilde{\lambda}}\cap U_\lambda} \to \Omega^{\gamma bc}_{\mathrm{ch}}(E; \mathcal{U})_\lambda\big|_{U_\lambda\cap U_{\Tilde{\lambda}}}.
\end{gather}
Thus we have families of strict morphisms $(\vartheta'_{\lambda \mu})_{\lambda, \mu \in \Lambda}$ and $(\vartheta''_{\lambda \mu})_{\lambda, \mu \in \Lambda}$. These morphisms satisfy the following.

\begin{lemma}
The following hold:
\begin{gather*}
\vartheta'_{\lambda \lambda}= \mathrm{id}, \quad \vartheta'_{\lambda \mu}\circ\vartheta'_{\mu \nu}=\vartheta'_{\lambda \nu}, \\
\vartheta''_{\lambda \lambda}=\mathrm{id}, \quad \vartheta''_{\lambda \mu}\circ\vartheta''_{\mu \nu}=\vartheta''_{\lambda \nu}, 
\end{gather*}
for all $\lambda, \mu, \nu\in \Lambda$.
\end{lemma}
\begin{proof}
It suffices to check
\begin{gather*}
\underrightarrow{\vartheta'_{\lambda \lambda}}= \mathrm{id}, \quad \underrightarrow{\vartheta'_{\lambda \mu}}\circ\underrightarrow{\vartheta'_{\mu \nu}}=\underrightarrow{\vartheta'_{\lambda \nu}}, \\
\underrightarrow{\vartheta''_{\lambda \lambda}}=\mathrm{id}, \quad \underrightarrow{\vartheta''_{\lambda \mu}}\circ\underrightarrow{\vartheta''_{\mu \nu}}=\underrightarrow{\vartheta''_{\lambda \nu}}, 
\end{gather*}
for all $\lambda, \mu, \nu\in \Lambda$. We can see that these relations hold on the generators from \eqref{df: phi^*f}, \eqref{df: phi^*b} and \eqref{df: phi^*c}. Thus we get the above relations.
\end{proof}

By the construction, $\bigl((\Omega^{\gamma c}_\mathrm{ch}(E; \mathcal{U})_\lambda)_{\lambda\in\Lambda}, (\vartheta'_{\lambda \mu})_{\lambda, \mu \in \Lambda}\bigr)$ satisfies the assumptions (1)-(5) in Proposition \ref{prop: GLUING VsaIndShs}. Therefore 
by Proposition \ref{prop: GLUING VsaIndShs}, we can glue the degree-weight-graded VSA-inductive sheaves  $\bigl((\Omega^{\gamma c}_\mathrm{ch}(E; \mathcal{U})_\lambda)_{\lambda\in\Lambda}, (\vartheta'_{\lambda \mu})_{\lambda, \mu \in \Lambda}\bigr)$ to a degree-weight-graded VSA-inductive sheaf on $M$. We denote by $\Omega^{\gamma c}_\mathrm{ch}(E; \mathcal{U})$ the resulting one. In the same way, we have a weight-degree-graded VSA-inductive sheaf on $M$, which we denote by $\Omega^{\gamma bc}_{\mathrm{ch}}(E; \mathcal{U})$, by gluing $\bigl((\Omega^{\gamma bc}_{\mathrm{ch}}(E; \mathcal{U})_\lambda)_{\lambda\in\Lambda},$ $(\vartheta''_{\lambda \mu})_{\lambda, \mu \in \Lambda}\bigr)$.

\begin{remark}\label{rem: we can take any framed covering}
By Remark \ref{rem: Isomorphic as VSA-inductive sheaves if locally isomorphic}, the objects $\Omega^{\gamma c}_\mathrm{ch}(E; \mathcal{U})$ and $\Omega^{\gamma bc}_{\mathrm{ch}}(E; \mathcal{U})$ of the category  $\textit{DegWt-}\mathit{VSA_{\mathbb{K}}}\textit{-IndSh}_M^\mathbb{N}$ \textit{do not} depend on the choice of  framed coverings and are  unique up to isomorphism.
\end{remark}
Therefore we simply denote by $\Omega^{\gamma c}_\mathrm{ch}(E)$ and $\Omega^{\gamma bc}_{\mathrm{ch}}(E)$ the VSA-inductive sheaves $\Omega^{\gamma c}_\mathrm{ch}(E; \mathcal{U})$ and $\Omega^{\gamma bc}_{\mathrm{ch}}(E; \mathcal{U})$, respectively. 

\begin{remark}
From the way of gluing and Remark \ref{eq: WEIGHT ZERO SPACE IS CLASSICAL}, the sheaf $\underrightarrow{\mathrm{Lim}}\, {\Omega^{\gamma c}_\mathrm{ch}(E)}[0]$  is canonically isomorphic to the sheaf of local  sections of $\bigwedge E^*$, denoted by $\underline{\bigwedge E^*}$.
\end{remark}

\begin{remark}
Let $E$ be the tangent bundle $TM$ of a manifold $M$.  
The presheaf associated with the VSA-inductive sheaf $\Omega^{\gamma c}_\mathrm{ch}(TM; \mathcal{U}^M)$ coincides with the small CDR for $M$ constructed in \cite{LL}, denoted by $\mathcal{Q}'_M$, where $\mathcal{U}_M=(U_\lambda)_{\lambda\in\Lambda}$ is the framed covering consisting of all open subsets $U_\lambda$ with a chart $\mathbf{x}_\lambda$ of $U_\lambda$ and the standard frame $(\partial/\partial x_\lambda^i)_{i=1, \dots, m}$.
\end{remark}

\subsection{Chiral Lie Algebroid Complex}
We consider the case when the vector bundle $E$ is a Lie algebroid. We construct a differential on $\Omega^{\gamma c}_\mathrm{ch}(E)$.

Let $(A, a, [ , ])$ be a Lie algebroid on a manifold $M$. We set $m:=\dim M$ and $r:=\mathrm{rank}\, A$. Let $\mathcal{U}=(U_\lambda)_{\lambda\in\Lambda}$ be a  framed covering of $A$. As before we denote by $\mathbf{x}_\lambda$ and $\mathbf{e}_\lambda$ the chart and the frame associated with $U_\lambda$, respectively. 
First we define a differential on each $\Omega^{\gamma c}_\mathrm{ch}(A; \mathcal{U})_\lambda$. Then we will glue them.

Fix $\lambda \in \Lambda$. 
We write the anchor map and the bracket  as 
$$
a(e_\lambda^j)=\sum_{\substack{1\le i\le m}}f^{\lambda, {i j}}\frac{\partial}{\partial x_\lambda^i}, 
$$
and
$$
[e_\lambda^j,e_\lambda^k]=\sum_{1\le l \le r}\Gamma^{\lambda, j k}_l e_\lambda^l,
$$
for $j, k=1, \dots, r$, where $f^{\lambda, {i j}}, \Gamma^{\lambda, j k}_l \in C^\infty(U_\lambda)$.

Let $V$ be an open subset of $U_\lambda$.
Motivated by the differential $d_{\mathrm{Lie}}$ for the Lie algebroid cohomology, we define an odd element $Q^\lambda(V)$ of degree $1$ and weight $1$ in $\Gamma\Bigl(V, \underrightarrow{\mathrm{Lim}}\, {\Omega_\mathrm{ch}(A; \mathcal{U})_\lambda}\Bigr)=\Gamma\Bigl(\mathbf{x}_\lambda(V), \underrightarrow{\Omega_{\mathrm{ch}}}(\mathbb{R}^{m|r})\Bigr)$ by setting 
$$
Q^\lambda(V):=\sum_{\substack{1\le i \le m \\ 1\le j \le r}}\beta^i_{-1}f^{\lambda, {i j}} c^j_{0}\mathbf{1} -\frac{1}{2}\sum_{1\le j, k, l \le r}\Gamma^{\lambda, j k}_{l} c^j_0 c^k_0 b^l_{-1}\mathbf{1}.
$$

Notice that the corresponding vertex operator $\underrightarrow{D^\lambda_{\mathrm{Lie}}}(V):=Q^\lambda(V)_{(0)}$ preserves the subspace $\Gamma\Bigl(V, \underrightarrow{\mathrm{Lim}}\, {\Omega^{\gamma c}_\mathrm{ch}(A; \mathcal{U})_\lambda}\Bigr)=\Gamma\Bigl( \mathbf{x}_\lambda(V), \underrightarrow{\Omega^{\gamma c}_{\mathrm{ch}}}(\mathbb{R}^{m|r}) \Bigr)$ and moreover coincides with the differential $d_{\mathrm{Lie}}$ on the weight $0$ space identified with $\Gamma(V, \bigwedge E^*)$.

\begin{lemma}\label{lem: D^2=0 PRESHEAF VERSION}
$$
(\underrightarrow{D^\lambda_{\mathrm{Lie}}}(V))^2=0.
$$
\end{lemma}
\begin{proof}
It suffices to check $
(\underrightarrow{D^\lambda_{\mathrm{Lie}}}(V))^2=0
$ on the weight $0$ subspace by Lemma \ref{lem: odd derivation is 0 if so on generators}. This follows from the fact that the differential $\underrightarrow{D^\lambda_{\mathrm{Lie}}}(V)$ on the weight $0$ subspace is nothing but the differential $d_{\mathrm{Lie}}$ for the Lie algebroid cohomology.
\end{proof}

We can see that the linear maps $\underrightarrow{D^\lambda_{\mathrm{Lie}}}(V)$ with open subsets $V\subset U_\lambda$ define a linear endomorphism $\underrightarrow{D^\lambda_{\mathrm{Lie}}}$ of the presheaf $\underrightarrow{\mathrm{Lim}}\, {\Omega^{\gamma c}_\mathrm{ch}(A; \mathcal{U})_\lambda}$ by the definition of the restriction maps. Thus we have a differential $\underrightarrow{D^\lambda_{\mathrm{Lie}}}$ on  $\underrightarrow{\mathrm{Lim}}\, {\Omega^{\gamma c}_\mathrm{ch}(A; \mathcal{U})_\lambda}$.

\begin{lemma}\label{lem: THE DIFFERENTIAL IS GLUED PRESHEAF VERSION}
The following holds: 
$$
\underrightarrow{D^\lambda_{\mathrm{Lie}}}|_{U_\lambda\cap\mu}\circ\underrightarrow{\vartheta'_{\lambda \mu}}=\underrightarrow{\vartheta'_{\lambda \mu}}\circ \underrightarrow{D^\mu_{\mathrm{Lie}}}|_{U_\mu\cap U_\lambda},
$$
for any $\lambda, \mu \in \Lambda$.
\end{lemma}
\begin{proof}
By Lemma \ref{lem: sufficient condition for B-linearilty} with $f=\underrightarrow{\vartheta'_{\lambda \mu}}$ and $N=0$, it suffices to check the relation on the weight $0$ subspace. This follows from the fact that $\underrightarrow{\vartheta'_{\lambda \mu}}$ and $\underrightarrow{D^\lambda_{\mathrm{Lie}}}$ coincide with the gluing map of the sheaf of sections $\underline{\bigwedge A^*}$ and the differential for the Lie algebroid cohomology, respectively.
\end{proof}

The morphism $\underrightarrow{D^\lambda_{\mathrm{Lie}}}$ consists of homogeneous maps. By Remark \ref{rem: homogeneous morphism of presheaves induces that of ind-objects}, the operator  $\underrightarrow{D^\lambda_{\mathrm{Lie}}}$ induces a strict  morphism of ind-objects of sheaves $D^\lambda_{\mathrm{Lie}}\!:  \Omega^{\gamma c}_\mathrm{ch}(A; \mathcal{U})_\lambda\!\to \Omega^{\gamma c}_\mathrm{ch}(A; \mathcal{U})_\lambda$.

\begin{lemma}\label{lem: THE DIFFERENTIAL IS GLUED}
The following holds: 
$$
D^\lambda_{\mathrm{Lie}}|_{U_\lambda\cap\mu}\circ\vartheta'_{\lambda \mu}=\vartheta'_{\lambda \mu}\circ D^\mu_{\mathrm{Lie}}|_{U_\mu\cap U_\lambda},
$$
for any $\lambda, \mu \in \Lambda$.
\end{lemma}
\begin{proof}
This relation follows from \ref{lem: THE DIFFERENTIAL IS GLUED PRESHEAF VERSION}.
\end{proof}

By Lemma \ref{lem: THE DIFFERENTIAL IS GLUED}, we can glue the strict morphisms $(D^\lambda_{\mathrm{Lie}})_{\lambda\in\Lambda}$ into a morphism $D_{\mathrm{Lie}}$ on $\Omega^{\gamma c}_\mathrm{ch}(A)$. Note that  by Remark \ref{rem: morphisms coincides if so locally}, the morphism $D_{\mathrm{Lie}}$ is independent of the choice of framed coverings.

\begin{lemma}
The morphism $D_{\mathrm{Lie}}$ is a differential on $\Omega^{\gamma c}_\mathrm{ch}(A)$. 
\end{lemma}
\begin{proof}
By Remark \ref{rem: morphisms coincides if so locally}, it suffices to check the relations \eqref{eq: deg 1 wt 0 condition of differential on VSA-inductive sheaf}, \eqref{eq: square 0 condition of differential on VSA-inductive sheaf}, \eqref{eq: derivation condition of differential on VSA-inductive sheaf} and \eqref{eq: degree derivation condition of differential on VSA-inductive sheaf} locally. This follows from the fact that $\underrightarrow{D_{\mathrm{Lie}}^\lambda}$ is a differential. 
\end{proof}

Let $\Gamma\bigl(M, \underrightarrow{\mathrm{Lim}}\, {\Omega^{\gamma c}_\mathrm{ch}(A)}\bigr)$ be the space of all global sections of the presheaf  $\underrightarrow{\mathrm{Lim}}\, {\Omega^{\gamma c}_\mathrm{ch}(A)}$ obtained by applying the functor $\underrightarrow{\mathrm{Lim}}\, $ to the degree-weight-graded  VSA-inductive sheaf $\Omega^{\gamma c}_\mathrm{ch}(A)$. 
By the lemma above and Remark \ref{rem: presheaf  of differential VSA-inductive sheaf}, we have the following. 

\begin{theorem}
The pair $\Bigl(\Gamma\bigl(M, \underrightarrow{\mathrm{Lim}}\, {\Omega^{\gamma c}_\mathrm{ch}(A)}\bigr), \underrightarrow{\mathrm{Lim}}\, {D_{\mathrm{Lie}}}(M)\Bigr)$ is a differential degree-weight-graded vertex superalgebra. 
\end{theorem}

The theorem above leads us to the following definition. 

\begin{definition}
Let $A$ be a Lie algebroid on a manifold $M$. Its \textbf{chiral Lie algebroid cohomology}, denoted by $H_{\mathrm{ch}}(A)$, is the cohomology of the complex 
$\Bigl(\Gamma\bigl(M, \underrightarrow{\mathrm{Lim}}\, {\Omega^{\gamma c}_\mathrm{ch}(A)}\bigr), \underrightarrow{\mathrm{Lim}}\, {D_{\mathrm{Lie}}}(M)\Bigr)$.
\end{definition}

\begin{remark}
From the construction, the chiral Lie algebroid cohomology $\!H_{\mathrm{ch}}(A)$ is a  $\mathbb{Z}_{\ge 0}$-graded vertex superalgebra and the subspace of weight $0$, $H_{\mathrm{ch}}(A)[0]$, coincides with the classical Lie algebroid cohomology. 
\end{remark}

\section{Chiral Equivariant Lie Algebroid Cohomology}\label{section: Chiral Equivariant Lie Algebroid Cohomology}

\subsection{Definition of Chiral Equivariant Lie Algebroid Cohomology}
Let $(A, a, [ , ])$ be a Lie algebroid on a manifold $M$. 
We will first equip $\Bigl(\Gamma\bigl(M, \underrightarrow{\mathrm{Lim}}\, {\Omega^{\gamma c}_\mathrm{ch}(A)}\bigr),$ $\underrightarrow{\mathrm{Lim}}\, {D_{\mathrm{Lie}}}(M)\Bigr)$ with a differential $\mathfrak{s}\Gamma(M, A)[t]$-module structure, where $\Gamma(M, A)$ is the Lie algebra of global sections of $A$. 

Set $m:=\dim M$ and $r:=\mathrm{rank}\, A$. Let $\mathcal{U}=(U_\lambda)_{\lambda\in\Lambda}$ be a  framed covering of $A$. As before we denote by $\mathbf{x}_\lambda$ and $\mathbf{e}_\lambda$ the chart and the frame associated with $U_\lambda$, respectively. 
Let $X\in \Gamma(M, A)$ be a global section.
Fix $\lambda\in\Lambda$.
We can write $X$ on $U_\lambda$ as $X|_{U_\lambda}=\sum_{j=1}^r f^{\lambda, j} e_\lambda^j$, where $f^{\lambda, j}$ is a function on $U_\lambda$.  
We set 
$$
\iota_{X}(V):=\sum_{j=1}^r f^{\lambda, j} b^j_{-1}\mathbf{1} \in \Gamma\Bigl(\mathbf{x}_\lambda (V), \underrightarrow{\Omega^{\gamma bc}_{\mathrm{ch}}}(\mathbb{R}^{m|r})\Bigr)=\Gamma(V, \underrightarrow{\mathrm{Lim}}\, {\Omega^{\gamma bc}_{\mathrm{ch}}(A; \mathcal{U}}_{\lambda})),
$$
for an open subset $V\subset U_\lambda$.
For each $n\in\mathbb{Z}$, the corresponding vertex operators $\iota_X(V)_{(n)}$ with open subsets $V\subset U_\lambda$ form a morphism $\underrightarrow{\iota_{X, (n)}^\lambda}$ on the presheaf $\underrightarrow{\mathrm{Lim}}\, {\Omega^{\gamma bc}_{\mathrm{ch}}(A; \mathcal{U}}_{\lambda})$. Note that for $n\ge0$, the morphism $\underrightarrow{\iota_{X, (n)}^\lambda}$ preserves the subpresheaf  $\underrightarrow{\mathrm{Lim}}\, {\Omega^{\gamma c}_\mathrm{ch}(A; \mathcal{U})_{\lambda}}$. Let $n\ge0$. Since the morphism $\underrightarrow{\iota_{X, (n)}^\lambda}$ is  homogeneous of weight $-n$, it induces a strict morphism of ind-objects $\iota_{X, (n)}^\lambda: \Omega^{\gamma c}_\mathrm{ch}(A; \mathcal{U})_{\lambda}\to \Omega^{\gamma c}_\mathrm{ch}(A; \mathcal{U})_{\lambda}$ by Remark \ref{rem: homogeneous morphism of presheaves induces that of ind-objects}.

\begin{lemma}
Let $n\in \mathbb{Z}_{\ge0}$. Then
$$
\iota_{X, (n)}^\lambda|_{U_\lambda\cap\mu}\circ\vartheta'_{\lambda \mu}=\vartheta'_{\lambda \mu}\circ \iota_{X, (n)}^\mu|_{U_\mu\cap U_\lambda},
$$
hold for all $\lambda, \mu\in\Lambda$.
\end{lemma}
\begin{proof}
It suffices to show $\underrightarrow{\iota_{X, (n)}^\lambda}\big|_{U_\lambda\cap\mu}\circ\underrightarrow{\vartheta'_{\lambda \mu}}=\underrightarrow{\vartheta'_{\lambda \mu}}\circ \underrightarrow{\iota_{X, (n)}^\mu}\big|_{U_\mu\cap U_\lambda}$ for each $\lambda, \mu\in\Lambda$. By Lemma \ref{lem: sufficient condition for B-linearilty}, it suffices to check the relation on the weight $0$ subspace. This follows from the fact that the operator  $\underrightarrow{\iota_{X, (0)}^\lambda}$ coincides with the inner product for $X$ on the weight $0$ subspace and the fact that the operator $\underrightarrow{\iota_{X, (n)}^\lambda}$ has weight $n$. 
\end{proof}

Thus for $n\ge0$, we have a morphism of ind-objects $\iota_{X, (n)}: \Omega^{\gamma c}_\mathrm{ch}(A)\to\Omega^{\gamma c}_\mathrm{ch}(A)$ by gluing the strict morphisms $(\iota_{X, (n)}^\lambda)_{\lambda\in\Lambda}$, which is independent of the choice of framed covering as in the case of $D_{\mathrm{Lie}}$. 
We set 
$$
L_{X, (n)}:= [D_{\mathrm{Lie}},\iota_{X, (n)}],
$$
for $n \ge 0$.

\begin{lemma}\label{lem: CHIRAL CARTAN RELATIONS}
Let $X, Y$ be global sections of $A$ and $n, k$  non-negative integers. Then the following relations hold:
\begin{enumerate}[(i)]
 \setlength{\topsep}{1pt}
     \setlength{\partopsep}{0pt}
     \setlength{\itemsep}{1pt}
     \setlength{\parsep}{0pt}
     \setlength{\leftmargin}{20pt}
     \setlength{\rightmargin}{0pt}
     \setlength{\listparindent}{0pt}
     \setlength{\labelsep}{3pt}
     \setlength{\labelwidth}{15pt}
     \setlength{\itemindent}{0pt}
     \renewcommand{\makelabel}{\upshape}
\item $[L_{X, (n)},\iota_{Y, (k)}]=\iota_{[X,Y], (n+k)}$,
\item $[L_{X, (n)},L_{Y, (k)}]=L_{[X,Y], (n+k)}$,
\item $[D_{\mathrm{Lie}},L_{X, (n)}+\iota_{Y, (k)}]=L_{Y, (k)}$.
\end{enumerate}
\end{lemma}
\begin{proof}
By Remark \ref{rem: morphisms coincides if so locally}, it suffices to check the relations locally. 
Since the operators $D_{\mathrm{Lie}}^\lambda=D_{\mathrm{Lie}}|_{U_\lambda}$, $\iota_{Y, (k)}^\lambda=\iota_{Y, (k)}|_{U_\lambda}$ and $L_{X, (n)}^\lambda:=L_{X, (n)}|_{U_\lambda}$ come from morphisms of presheaves, 
it suffices to check the corresponding morphisms $\underrightarrow{D_{\mathrm{Lie}}^\lambda}=\underrightarrow{\mathrm{Lim}}\,  D_{\mathrm{Lie}}^\lambda$,  $\underrightarrow{\iota_{Y, (k)}^\lambda}=\underrightarrow{\mathrm{Lim}}\, \iota_{Y, (k)}^\lambda$ and $\underrightarrow{L_{X, (n)}^\lambda}:=\underrightarrow{\mathrm{Lim}}\,  L_{X, (n)}^\lambda$ on $\underrightarrow{\mathrm{Lim}}\, \Omega^{\gamma c}_\mathrm{ch}(A; \mathcal{U})_\lambda$ satisfy the same relations. 
This is done by direct computations of OPEs. 
\end{proof}

\begin{theorem}
Let $\mathfrak{g}$ be a Lie algebra.
Suppose given a morphism of Lie algebras 
$
x^A: \mathfrak{g}\to \Gamma(M, A), \quad \xi\mapsto x^A_\xi.
$
Then 
the assignment
$$
\mathfrak{sg}[t]\ni(\xi, \eta)t^n\mapsto \underrightarrow{\mathrm{Lim}}\, {L_{x^A_\xi, (n)}}(M)+\underrightarrow{\mathrm{Lim}}\, {\iota_{x^A_\eta, (n)}}(M)\in\mathrm{End}\bigl(\Gamma\bigl(M, \underrightarrow{\mathrm{Lim}}\, {\Omega^{\gamma c}_\mathrm{ch}(A)}\bigr)\bigr),
$$
defines a differential $\mathfrak{sg}[t]$-module structure on  the complex $\Bigl(\Gamma\bigl(M, \underrightarrow{\mathrm{Lim}}\, {\Omega^{\gamma c}_\mathrm{ch}(A)}\bigr),$ $ \underrightarrow{\mathrm{Lim}}\, {D_{\mathrm{Lie}}}(M)\Bigr)$. 
\end{theorem}
\begin{proof}
By the above lemma, it suffices to check that the operator  $\underrightarrow{\mathrm{Lim}}\, {\iota_{X, (n)}}(M)$ has degree $-1$ and weight $-n$ for any $n\in\mathbb{N}$ and $X\in\Gamma(M, A)$.  Note that the continuity of the action of $\mathfrak{sg}[t]$ follows from the fact that the weight-grading on $\Gamma\bigl(M, \underrightarrow{\mathrm{Lim}}\, {\Omega^{\gamma c}_\mathrm{ch}(A)}\bigr)$ is bounded from the below.  For that purpose, it suffices to show that $[\underline{J}, \iota_{X, (n)}]=-\iota_{X, (n)}$ and $[\underline{H}, \iota_{X, (n)}]=-n\iota_{X, (n)}$ for any $n\in\mathbb{N}$ and $X\in\Gamma(M, A)$, where $\underline{J}$ and $\underline{H}$ is the degree-grading operator and the Hamiltonian of $\Gamma\bigl(M, \underrightarrow{\mathrm{Lim}}\, {\Omega^{\gamma c}_\mathrm{ch}(A)}\bigr)$. These relations follow from the fact that the same relations hold locally. 
\end{proof}

This leads us to the following definitions. (See Section \ref{subsection: Chiral Equivariant Cohomology} for the definitions of differential $\mathfrak{sg}[t]$-modules, the chiral basic cohomology and the chiral equivariant cohomology)

\begin{definition}
Let $G$ be a compact connected Lie group and  $\mathfrak{g}$ the Lie algebra $\mathrm{Lie}(G)^\mathbb{K}$.
Let $A$ be a Lie algabroid on a manifold $M$ with a Lia algebra morphism $\mathfrak{g}\to \Gamma(M, A)$. 
The \textbf{chiral basic Lie algebroid cohomology} of $A$, denoted by $H_{\mathrm{ch}, bas}(A)$, is the chiral basic cohomology of  the differential $\mathfrak{sg}[t]$-module 
$\Bigl(\Gamma\bigl(M, \underrightarrow{\mathrm{Lim}}\, {\Omega^{\gamma c}_\mathrm{ch}(A)}\bigr), \underrightarrow{\mathrm{Lim}}\, {D_{\mathrm{Lie}}}(M)\Bigr)$. The \textbf{chiral equivariant Lie algebroid cohomology} of $A$, denoted by $H_{\mathrm{ch}, G}(A)$, is the chiral equivariant cohomology of  the differential $\mathfrak{sg}[t]$-module 
$\Bigl(\Gamma\bigl(M, \underrightarrow{\mathrm{Lim}}\, {\Omega^{\gamma c}_\mathrm{ch}(A)}\bigr), \underrightarrow{\mathrm{Lim}}\, {D_{\mathrm{Lie}}}(M)\Bigr)$.
\end{definition}

\begin{remark}
The chiral basic Lie algebroid cohomology $H_{\mathrm{ch}, bas}(A)$ and the chiral equivariant Lie algebroid cohomology $H_{\mathrm{ch}, G}(A)$ are  $\mathbb{Z}_{\ge0}$-graded vertex superalgebras. Indeed, $[\underrightarrow{\mathrm{Lim}}\, {L_{x^A_\xi, (n)}}(M), v_{(k)}]=\sum_{i\ge0}\binom{n}{i}(\underrightarrow{\mathrm{Lim}}\, {L_{x^A_\xi, (i)}}(M)v)_{(n+k-i)}$ and $[\underrightarrow{\mathrm{Lim}}\, {\iota_{x^A_\xi, (n)}}(M), v_{(k)}]=\sum_{i\ge0}\binom{n}{i}(\underrightarrow{\mathrm{Lim}}\, {\iota_{x^A_\xi, (i)}}(M)v)_{(n+k-i)}$ hold for any $n\ge0$, $k\in\mathbb{Z}$,  $v\in\Gamma\bigl(M, \underrightarrow{\mathrm{Lim}}\, {\Omega^{\gamma c}_\mathrm{ch}(A)}\bigr)$ and $\xi\in\mathfrak{g}$ since the same relations hold locally. Then the assertion follows from Lemma \ref{lem: generalized commutant} and Lemma \ref{lem: tensor commutant} together with the fact that $\Gamma\bigl(M, \underrightarrow{\mathrm{Lim}}\, {\Omega^{\gamma c}_\mathrm{ch}(A)}\bigr)$ is $\mathbb{Z}_{\ge0}$-graded. Moreover $H_{\mathrm{ch}, bas}(A)[0]$ and  $H_{\mathrm{ch}, G}(A)[0]$ coincide with  the classical basic Lie algebroid cohomology and the classical equivariant Lie algebroid cohomology, respectively. This follows from the fact that $\Gamma\bigl(M, \underrightarrow{\mathrm{Lim}}\, {\Omega^{\gamma c}_\mathrm{ch}(A)}\bigr)$ is $\mathbb{Z}_{\ge0}$-graded and the fact that $\underrightarrow{\mathrm{Lim}}\, {L_{x^A_\xi, (0)}}(M)$ and $\underrightarrow{\mathrm{Lim}}\, {\iota_{x^A_\xi, (0)}}(M)$ coincides with the classical Lie derivative and the classical interior product, respectively. 
\end{remark}

\subsection{Transformation Lie Algebroid Cases}
Let $M$ be a manifold with an infinitesimal action of a finite-dimensional Lie algebra $\mathfrak{g}$:
$$
x^M: \mathfrak{g}\to \mathscr{X}(M), \quad \xi\mapsto x^M_\xi.
$$
Let $A=M\times \mathfrak{g}$ be the corresponding transformation Lie algebroid (see Example \ref{ex: transformation Lie algebroid}).
Let $(\xi_j)_j$ be a basis of $\mathfrak{g}$, $(\xi^*_j)_j$ the  dual basis for $\mathfrak{g}^*$  and $(\Gamma^k_{i j})_{i, j, k}$  the structure constants, that is, constants satisfying $[\xi_i,\xi_j]=\sum_{k=1}^{\dim \mathfrak{g}}\Gamma^k_{i j}\xi_k$ for each $i, j=1, \dots, \dim\mathfrak{g}$.

Let $\mathcal{U}=(U_\lambda)_{\lambda\in\Lambda}$ be a framed covering of $A$ consisting of  open subsets with the constant frame $(\xi_j)_j$ as the frame on them. We can use the VSA-inductive sheaf $\Omega^{\gamma c}_\mathrm{ch}(A; \mathcal{U})$ for computations of the  cohomologies $H_{\mathrm{ch}}(A)$, $H_{\mathrm{ch}, bas}(A)$ and $H_{\mathrm{ch}, G}(A)$, since they do not depend on the choice of framed coverings. 
By the construction of $\Omega^{\gamma c}_\mathrm{ch}(A; \mathcal{U})$, we see $\underrightarrow{\mathrm{Lim}}\, {\Omega^{\gamma c}_\mathrm{ch}(A; \mathcal{U})}=\underrightarrow{\mathrm{Lim}}\, {C^{\infty, \gamma c}_{\mathrm{ch}, M}}\otimes\langle c \rangle$, where $C^{\infty, \gamma c}_{\mathrm{ch}, M}:=\Omega^{\gamma c}_\mathrm{ch}(M\times \{0\})$, and $\langle c \rangle$ is the subalgebra of $\mathcal{E}(\mathfrak{g})$  generated by $c^{\xi^*}_0\mathbf{1}$ with $\xi^*\in\mathfrak{g}^*$. Therefore 
by the construction, the differential for the corresponding chiral Lie algebroid cohomology $H_{\mathrm{ch}}(A)$ is nothing but the differential  for the continuous Lie algebra cohomology with coefficients in the $\mathfrak{g}[t]$-module $\underrightarrow{\mathrm{Lim}}\, {C^{\infty, \gamma c}_{\mathrm{ch}, M}}(M)$. The action of the element $\xi_j t^n\in\mathfrak{g}[t]$ on $\underrightarrow{\mathrm{Lim}}\, {C^{\infty, \gamma c}_{\mathrm{ch}, M}}(M)$ is given by the operator 
$$
\sum_{k\ge0}\sum_{i=1}^{\dim M}f^{i j}_{-k-n} \beta^i_k,
$$
on each space $\underrightarrow{\mathrm{Lim}}\, {C^{\infty, \gamma c}_{\mathrm{ch}, M}}(U_\lambda)$ of local sections, 
where the vector field $x^M_{\xi_j}|_{U_\lambda}$ is written as $\sum_{i=1}^{\dim M}f^{i j} \partial/\partial x^i$ with $f^{i j}\in C^\infty(U_\lambda)$. Thus we have 
$$
H_{\mathrm{ch}}(A)=H\bigl(\mathfrak{g}[t]; \underrightarrow{\mathrm{Lim}}\, {C^{\infty, \gamma c}_{\mathrm{ch}, M}}(M)\bigr).
$$

Consider the Lie algebra morphism 
\begin{equation}\label{eq: canonical Lie algebra morphism of transformation Lie algebroids}
\mathfrak{g}\to \Gamma(M, A)=C^\infty(M)\otimes \mathfrak{g}, \quad \xi \mapsto 1\otimes \xi.
\end{equation}
We will compute the corresponding chiral equivariant cohomology  of $\mathcal{A}:=\Gamma(M, \underrightarrow{\mathrm{Lim}}\, {\Omega^{\gamma c}_\mathrm{ch}(A; \mathcal{U})})$,  namely, $H_{\mathrm{ch}, G}(A)$. Notice that 
\begin{equation}\label{eq: iota for transformation Lie algebroids}
\iota^\mathcal{A}_{\xi, (n)}=b^{\xi}_n,
\end{equation}
for all  $\xi\in\mathfrak{g}$ and $n\ge0$. 
Then the chiral basic cohomology $H_{\mathrm{ch}, bas}(A)$ of $\mathcal{A}=\Gamma(M, \underrightarrow{\mathrm{Lim}}\, {\Omega^{\gamma c}_\mathrm{ch}(A; \mathcal{U})})$ is  as follows:
\begin{equation}\label{eq: chiral basic cohomology of transformation Lie algebroid}
H_{\mathrm{ch}, bas}^i(A)=
\begin{cases}
\underrightarrow{\mathrm{Lim}}\, {C^{\infty, \gamma c}_{\mathrm{ch}, M}}(M)^{\mathfrak{g}[t]}, & \text{when $i=0$,} \\
0,  & \text{otherwise.}
\end{cases}
\end{equation}
Indeed, from \eqref{eq: iota for transformation Lie algebroids} we have $\bigl(\Gamma(M, \underrightarrow{\mathrm{Lim}}\, {\Omega^{\gamma c}_\mathrm{ch}(A; \mathcal{U})})\bigr)_{hor}=\underrightarrow{\mathrm{Lim}}\, {C^{\infty, \gamma c}_{\mathrm{ch}, M}}(M)$ and therefore 
$$
\bigl(\Gamma(M, \underrightarrow{\mathrm{Lim}}\, {\Omega^{\gamma c}_\mathrm{ch}(A; \mathcal{U})})\bigr)_{bas}=\underrightarrow{\mathrm{Lim}}\, {C^{\infty, \gamma c}_{\mathrm{ch}, M}}(M)^{\mathfrak{g}[t]}.
$$

Equip $\mathcal{A}=\Gamma(M, \underrightarrow{\mathrm{Lim}}\, {\Omega^{\gamma c}_\mathrm{ch}(A; \mathcal{U})})=\underrightarrow{\mathrm{Lim}}\, {C^{\infty, \gamma c}_{\mathrm{ch}, M}}(M)\otimes\langle c \rangle$ with the $\langle c \rangle$-module structure given by the left multiplication. 
We claim that this $\langle c \rangle$-module structure induces a chiral $W^*$-module structure. Proposition \ref{prop: construction of chiral W^*-modules} will be applied. Let $d_\mathcal{A}$ be the differential for $\mathcal{A}$, that is, that for the Lie algebra cohomology $H\bigl(\mathfrak{g}[t], \underrightarrow{\mathrm{Lim}}\, {C^{\infty, \gamma c}_{\mathrm{ch}, M}}(M)\bigr)$. From the definition of $d_\mathcal{A}$ and the commutation relations in $\mathcal{E}(\mathfrak{g})$, we have 
\begin{align}
[d_\mathcal{A},c^{\xi^*_l, \mathcal{A}}(z)]&=\Bigl[-\frac{1}{2}\sum_{i, j , k =1}^{\dim \mathfrak{g}}\sum_{\substack{s, t\le0,\\ \notag s+t+u=0}}\Gamma_{i j}^kc^{\xi^*_i}_sc^{\xi^*_j}_tb^{\xi_k}_u,\ c^{\xi^*_l, \mathcal{A}}(z)\Bigr] \\ \notag
&=-\frac{1}{2}\sum_{i, j , k =1}^{\dim \mathfrak{g}}\sum_{\substack{s, t\le0,\\ s+t+u=0}}\Gamma_{i j}^kc^{\xi^*_i}_sc^{\xi^*_j}_t\langle \xi^*_l,  \xi_k\rangle z^u \\ \notag
&=-\frac{1}{2}\sum_{i, j=1}^{\dim \mathfrak{g}}\sum_{\substack{s, t\le0,\\ s+t+u=0}}\Gamma_{i j}^lc^{\xi^*_i}_sc^{\xi^*_j}_t z^u \\ 
&=-\frac{1}{2}\sum_{i, j=1}^{\dim \mathfrak{g}}\Gamma_{i j}^lc^{\xi^*_i, \mathcal{A}}(z)c^{\xi^*_j, \mathcal{A}}(z).\label{eq: the formula of [d_A,c] for translation Lie algebroid}
\end{align}
Therefore we have $[c^{\xi^*, \mathcal{A}}(z),[d_\mathcal{A}, c^{\eta^*, \mathcal{A}}(w)]]=0$ for all $\xi^*, \eta^*\in\mathfrak{g}$.
By Lemma \ref{prop: construction of chiral W^*-modules}, we obtain a  $\langle c, \gamma \rangle$-module structure $Y^\mathcal{A}$ on $\mathcal{A}$ by extending the above $\langle c \rangle$-module structure.  From the definition of $\gamma^{\xi^*_l, \mathcal{A}}(z)$ (see the proof of Proposition \ref{prop: construction of chiral W^*-modules}) 
and \eqref{eq: the formula of [d_A,c] for translation Lie algebroid}, we have 
$$
\gamma^{\xi^*_l, \mathcal{A}}(z)=[d_\mathcal{A},c^{\xi^*_l, \mathcal{A}}(z)]+\frac{1}{2}\sum_{i, j=1}^{\dim \mathfrak{g}}\Gamma_{i j}^lc^{\xi^*_i, \mathcal{A}}(z)c^{\xi^*_j, \mathcal{A}}(z)=0,
$$
for all $l=1, \dots, \dim \mathfrak{g}$. Therefore $[\iota_\xi^\mathcal{A}(z)_-,\gamma^{\xi^*, \mathcal{A}}(x)]=0$ for all $\xi\in\mathfrak{g}$ and $\xi^*\in\mathfrak{g}^*$. 
Recall that $\iota^\mathcal{A}_{\xi, (n)}=b^{\xi}_n$ for all  $\xi\in\mathfrak{g}$ and $n\ge0$. From this, we have $[\iota_\xi^\mathcal{A}(z)_-,c^{\xi^*, \mathcal{A}}(w)]=\langle\xi^*, \xi\rangle\delta(z-w)_-$ for all $\xi\in\mathfrak{g}$ and $\xi^*\in\mathfrak{g}^*$. Thus we can apply Proposition \ref{prop: construction of chiral W^*-modules} and we see that the triple $(\mathcal{A}, d_\mathcal{A}, Y^\mathcal{A})$ is a chiral $W^*$-module. 
We have proved the following.

\begin{theorem}
For a transformation Lie algebroid $A=M\times\mathfrak{g}$ with the  Lie algebra morphism \eqref{eq: canonical Lie algebra morphism of transformation Lie algebroids}, the differential  $\mathfrak{sg}[t]$-module $\Gamma(M, \underrightarrow{\mathrm{Lim}}\, {\Omega^{\gamma c}_\mathrm{ch}(A)})$ has a canonical structure of a chiral $W^*$-module.
\end{theorem}

Therefore by Theorem \ref{thm: CHIRAL BASIC=CHIRAL EQUIVARIANT} and \eqref{eq: chiral basic cohomology of transformation Lie algebroid}, we have the following.

\begin{corollary}\label{prop: chiral equivariant Lie algebroid cohomology for tramsf. Lie algebroids when comm}

Let $G$ be a compact connected Lie group, $\mathfrak{g}$ the Lie algebra $\mathrm{Lie}(G)^\mathbb{K}$ and $A=M\times\mathfrak{g}$ a transformation Lie algebroid. 
Assume that $G$ is commutative. Then the following holds:
\begin{equation}
H_{\mathrm{ch}, G}^{i}(A)=
\begin{cases}
\underrightarrow{\mathrm{Lim}}\, {C^{\infty, \gamma c}_{\mathrm{ch}, M}}(M)^{\mathfrak{g}[t]}, & \text{when $i=0$,} \\
0,  & \text{otherwise.}
\end{cases}
\end{equation}
\end{corollary}

We consider a special case. 
Let $(G, \Pi)$ be a compact connected Poisson-Lie group with the Lie algebra $\mathfrak{g}$ and $(G^*, \Pi^*)$ the dual Poisson-Lie group of $G$. Recall the Lie algebra morphism 
$$
\mathfrak{g}\to\Gamma(G^*, T^*G^*), \quad \xi\mapsto \xi^l, 
$$
where we denote by $\xi^l$ the left invariant $1$-form on $G^*$ whose value at $e$ is $\xi\in\mathfrak{g}$.  
Consider the corresponding chiral equivariant cohomology. 
By \cite[Proposition 5.25]{Lu90}, we have an isomorphism of Lie algebroids
$$
T^*G^*\cong G^*\times \mathfrak{g},
$$
using the left invariant one-forms on $G^*$. Here we equip the trivial  bundle $G^*\times \mathfrak{g}$ with the transformation Lie algebroid structure defined by the infinitesimal  left dressing action. 
The following is a chiral analogue of \cite[Corollary 4.20]{Gin99}.

\begin{proposition}
In the above setting, 
 assume that $(G, \Pi)$ is commutative. Then the following holds:
\begin{equation}
H_{\mathrm{ch}, G}^{i}(T^*G^*)=
\begin{cases}
\underrightarrow{\mathrm{Lim}}\, {C^{\infty, \gamma c}_{\mathrm{ch}, G^*}}(G^*), & \text{when $i=0$,} \\
0,  & \text{otherwise.}
\end{cases}
\end{equation}
\end{proposition}
\begin{proof}
The infinitesimal left dressing action is trivial since $T$ is commutative. Therefore our assertion follows from Corollary  \ref{prop: chiral equivariant Lie algebroid cohomology for tramsf. Lie algebroids when comm}.
\end{proof}

\section*{Acknowledgments}
The author wishes to express his sincere gratitude to his advisor
Professor Atsushi Matsuo for helpful advice and continuous encouragement
during the course of this work. He is also thankful to Professor Hiroshi
Yamauchi (Tokyo Women's Christian University) for advice and
encouragement. 

\providecommand{\bysame}{\leavevmode\hbox to3em{\hrulefill}\thinspace}
\providecommand{\MR}{\relax\ifhmode\unskip\space\fi MR }
\providecommand{\MRhref}[2]{%
  \href{http://www.ams.org/mathscinet-getitem?mr=#1}{#2}
}
\providecommand{\href}[2]{#2}

\end{document}